\input ./style/arxiv-general.cfg
\documentclass[MSNbibl,number,citesort,seceqn,dvips]{arxbj}
\makeatletter
   \@ifpackageloaded{graphicx}{}{\usepackage{graphicx}}
\makeatother
\usepackage{mathbh,euscript,upgreek}
\usepackage{mathrsfs}

\volume{22}
\issue{3}
\pubyear{2016}
\firstpage{1839}
\lastpage{1893}
\doi{10.3150/15-BEJ713}
\docsubty{FLA}

\makeatletter
\newcommand{\eqref}[1]{(\ref{#1})}
\newcommand{\rrvert}{\vert}
\newcommand{\rrVert}{\Vert}
\newcommand{\llvert}{\vert}
\newcommand{\llVert}{\Vert}
\def\pi{\uppi}
\def\mathds{\mathbh}
\newtheorem{thmm}{Theorem}[section]
\newremark{rem}{Remark}[section]
\newtheorem{lem}[thmm]{Lemma}
\newtheorem{pro}[thmm]{Proposition}

\newcommand{\bb}{\mathbb}
\newcommand{\eu}{\EuScript}
\newcommand{\Scr}{\mathscr}
\newcommand{\Cal}{\mathcal}
\newcommand{\fr}{\mathfrak}
\newproclaim{fin}{Definition}
\newremark{exm}{Example}
\newtheorem{appxthm}{Theorem}[section]
\newtheorem{appxlem}[appxthm]{Lemma}
\newtheorem{appxpro}[appxthm]{Proposition}
\newproclaim{appxfin}[appxthm]{Definition}
\makeatother

\begin{document}
\begin{frontmatter}

\title{On the optimal estimation of probability measures in weak and
strong topologies}
\runtitle{On the optimal estimation of a probability measure}

\begin{aug}
\author[A]{\inits{B.}\fnms{Bharath}~\snm{Sriperumbudur}\corref{}\ead[label=e1]{bks18@psu.edu}}
\address[A]{Department of Statistics,
Pennsylvania State University, University Park, PA 16802, USA.\\ \printead{e1}}
\end{aug}

%
\received{\smonth{7} \syear{2014}}
%
\revised{\smonth{2} \syear{2015}}

%
\begin{abstract}
Given random samples drawn i.i.d.~from a probability measure $\mathbb{P}$
(defined on say, $\mathbb{R}^d$), it is well-known that the empirical
estimator is
an
optimal estimator of $\mathbb{P}$ in weak topology but not even a consistent
estimator of its density (if it exists) in the strong topology (induced
by the
total variation distance). On the other
hand, various popular density estimators such as kernel and wavelet
density estimators are optimal in the strong topology in the sense of achieving
the minimax rate over all estimators for a Sobolev ball of densities.
Recently, it has been shown in a series of papers by Gin\'{e} and Nickl that
these density estimators on $\mathbb{R}$ that are optimal in strong
topology are
also optimal in
$\|\cdot\|_\mathcal{F}$ for certain choices of $\mathcal{F}$ such that
$\|\cdot\|_\mathcal{F}$ metrizes the weak topology, where
$\|\mathbb{P}\|_\mathcal{F}:=\sup\{\int f \,\mathrm{d}\mathbb{P}\dvt f\in\mathcal{F}\}$. In this paper,
we investigate this problem of optimal estimation in weak and strong topologies
by choosing $\mathcal{F}$ to be a unit ball in a reproducing kernel Hilbert space
(say $\mathcal{F}_H$ defined over $\mathbb{R}^d$), where this choice is both of
theoretical
and computational interest. Under some mild conditions on the
reproducing kernel, we show that $\|\cdot\|_{\mathcal{F}_H}$ metrizes
the weak
topology and the kernel density estimator (with $L^1$ optimal bandwidth)
estimates $\mathbb{P}$ at dimension independent optimal rate of $n^{-1/2}$ in
$\|\cdot\|_{\mathcal{F}_H}$ along with providing a uniform central limit
theorem for the kernel density estimator.
\end{abstract}

%
\end{frontmatter}

\section{Introduction}\label{Sec:intro}
Let $X_1,\ldots,X_n$ be independent random variables distributed
according to a
Borel probability measure $\bb{P}$ defined on a separable metric
space $\Cal{X}$ with $\bb{P}_n:=\frac{1}{n}\sum^n_{i=1}\delta_{X_i}$
being the
empirical measure induced by them. It is
well known that $\bb{P}_n$ is a consistent estimator of
$\bb{P}$ in weak
sense as $n\rightarrow\infty$, that is, for every bounded continuous real-valued
function $f$ on $\Cal{X}$, $\int f \,\mathrm{d}\bb{P}_n\stackrel{\mathrm{a.s.}}{\rightarrow
} \int
f \,\mathrm{d}\bb{P}$ as
$n\rightarrow\infty$, written as $\bb{P}_n\leadsto\bb{P}$. In
fact, if nothing is known about $\bb{P}$, then $\bb{P}_n$ is probably
the most appropriate
estimator to use as it is asymptotically efficient and minimax in
the
sense of van der Vaart \cite{Vaart-98}, Theorem~25.21, equation (25.22); also see Example~25.24. In addition,
for any Donsker class of
functions, $\Cal{F}$, $\Vert
\bb{P}_n-\bb{P}\Vert_\Cal{F}=\mathrm{O}_{\bb{P}}(n^{-1/2})$, where
\[
\Vert\bb{P}_n-\bb{P}\Vert_\Cal{F}:=\sup
_{f\in\Cal{F}}\biggl\llvert \int f \,\mathrm{d}\bb{P}_n-\int f \,\mathrm{d}\bb{P}
\biggr\rrvert ,
\]
that is, $\bb{P}_n-\bb{P}$ is asymptotically of the order of $n^{-1/2}$
uniformly
in $\Cal{F}$ and the processes $f\mapsto\sqrt{n}\int
f \,\mathrm{d}(\bb{P}_n-\bb{P}), f\in\Cal{F}$
converge in law to a Gaussian process in $\ell^\infty(\Cal{F})$, called the
$\bb{P}$-Brownian bridge indexed by $\Cal{F}$,
where $\ell^\infty(\Cal{F})$ denotes the Banach space of bounded real-valued
functions on $\Cal{F}$.
On the other hand, if $\bb{P}$ has a density $p$ with
respect to Lebesgue measure (assuming $\Cal{X}=\bb{R}^d$), then $\bb{P}_n$,
which is a random atomic measure, is not appropriate to estimate~$p$. However,
various estimators, $p_n$ have been proposed in literature to estimate
$p$, the
popular ones being the kernel density estimator and wavelet estimator, which
under suitable conditions have been shown to be optimal with respect to the
$L^r$ loss
($1\le r\le\infty$) in the sense of achieving the minimax rate over all
estimators for densities in certain classes
(Devroye and Gy{\"{o}}rfi \cite{Devroye-85}, Hardle \textit{et al.} \cite
{Hardle-98}, van der Vaart \cite{Vaart-98}). Therefore, depending on
whether $\bb{P}$
has a density or not, there are two different estimators (i.e., $\bb
{P}_n$ and
$p_n$) that are optimal in two different performance measures, that is,
$\Vert\cdot\Vert_\Cal{F}$ and $L^r$. While $\bb{P}_n$ is not adequate to
estimate $p$, the question arises as to whether
$\bb{P}^\star_n$ defined as $\bb{P}^\star_n(A):=\int_A p_n(x) \,\mathrm{d}x$ for every
Borel set $A\subset\bb{R}^d$, estimates $\bb{P}$
as good as $\bb{P}_n$ in the sense that $\Vert
\bb{P}^\star_n-\bb{P}\Vert_\Cal{F}=\mathrm{O}_{\bb{P}}(n^{-1/2})$, that is,
%
\begin{equation}
\sup_{f \in\Cal{F}}\biggl\llvert \int f(x)p_n(x) \,\mathrm{d}x-\int
f(x)p(x) \,\mathrm{d}x\biggr\rrvert =\mathrm{O}_{\bb{P}}\bigl(n^{-1/2}
\bigr),\label{Eq:smooth}
\end{equation}
and whether the processes $f\mapsto\sqrt{n}\int
f(x)(p_n-p)(x) \,\mathrm{d}x, f\in\Cal{F}$ converge in law to a Gaussian process in
$\ell^\infty(\Cal{F})$ for $\bb{P}$-Donsker class, $\Cal{F}$. If $p_n$
satisfies these properties, then it is a \emph{plug-in} estimator in
the sense
of Bickel and Ritov \cite{Bickel-03}, Definition~4.1, as it is simultaneously
optimal in two different performance measures. The question of whether
(\ref{Eq:smooth}) holds has been addressed for the kernel density estimator
(Yukich \cite{Yukich-92}, van der Vaart \cite{Vaart-94}, Gin{\'{e}} and
Nickl \cite{Gine-08a}) and wavelet density estimator
(Gin{\'{e}} and Nickl \cite{Gine-09b})
where a uniform central limit theorem as stated above has been proved for
various $\bb{P}$-Donsker classes, $\Cal{F}$ (and also for non-Donsker but
pre-Gaussian classes in Radulovi{\'{c}} and
Wegkamp \cite{Radulovic-00} and Gin{\'{e}} and Nickl \cite{Gine-08a},
Section~4.2).
For a $\bb{P}$-Donsker class $\Cal{F}$, it easy to show
that (\ref{Eq:smooth})
and the corresponding uniform central limit theorem (UCLT) hold if
$\Vert\bb{P}^\star_n-\bb{P}_n\Vert_\Cal{F}=\mathrm{o}_{\bb{P}}(n^{-1/2})$, that is,
%
\begin{equation}
\sup_{f \in\Cal{F}}\biggl\llvert \int f(x)p_n(x) \,\mathrm{d}x-\int
f(x) \,\mathrm{d}\bb{P}_n(x)\biggr\rrvert =\mathrm{o}_{\bb{P}}
\bigl(n^{-1/2}\bigr).\label{Eq:smooth-1}
\end{equation}
Several recent works
(Bickel and Ritov \cite{Bickel-03},
Nickl \cite{Nickl-07},
Gin{\'{e}} and Nickl \cite{Gine-08a,Gine-09a,Gine-09b,Gine-10}) have
shown that
many popular density estimators on $\Cal{X}=\bb{R}$, such as maximum
likelihood
estimator, kernel density estimator and wavelet estimator satisfy
(\ref{Eq:smooth-1}) if $\Cal{F}$ is $\bb{P}$-Donsker~-- the Donsker classes
that
were considered in these works are: functions of bounded variation,
$\{\mathds{1}_{(-\infty,t]}\dvt t\in\bb{R}\}$, H\"{o}lder, Lipschitz and Sobolev
classes on $\bb{R}$. In other words, these works show that there exists
estimators that are within a
$\Vert\cdot\Vert_\Cal{F}$-ball of size $\mathrm{o}_{\bb{P}}(n^{-1/2})$
around $\bb{P}_n$ such that they estimate $\bb{P}$ consistently
in $\Vert\cdot\Vert_\Cal{F}$ at the rate of $n^{-1/2}$, that is, they have
a statistical behavior similar to that of $\bb{P}_n$.

The main contribution of this paper is to generalize the above
behavior of
kernel
density estimators to any $d$ by showing that $\bb{P}$ can be
estimated
optimally in $\Vert\cdot\Vert_{\Cal{F}_H}$ using a kernel density
estimator, $p_n$ (with $L^1$
optimal bandwidth) on $\bb{R}^d$ where under
certain conditions on $\Cal{K}$, $\Vert\cdot\Vert_{\Cal{F}_H}$ with
%
\begin{equation}
\Cal{F}_H:= \bigl\{f\dvtx\bb{R}^d\rightarrow\bb{R} |
\Vert f\Vert_{\Cal{H}_k}\le 1\dvt f\in\Cal{H}_k, k\in\Cal{K}
\bigr\}\label{Eq:FH}
\end{equation}
metrizes the
weak
topology on the space of Borel probability measure on $\bb{R}^d$.
Here,
$\Cal{H}_k$ denotes a
reproducing kernel Hilbert space (RKHS) (Aronszajn \cite
{Aronszajn-50}); also see Berlinet and Thomas-Agnan~\cite{Berlinet-04} and
Steinwart and Christmann \cite{Steinwart-08}, Chapter~4, for a nice
introduction to RKHS and its applications in probability, statistics and
learning theory~-- with
$k\dvtx\bb{R}^d\times\bb{R}^d\rightarrow\bb{R}$ as the reproducing
kernel (and,
therefore, positive definite) and $\Cal{K}$ is a cone of positive definite
kernels. To elaborate, the paper
shows that the kernel density estimator on $\bb{R}^d$ with
an appropriate choice of bandwidth is not only optimal in the
strong topology (i.e., in total variation distance or $L^1$) but also optimal
in the weak
topology induced by $\Vert\cdot\Vert_{\Cal{F}_H}$ (i.e., has a similar
statistical behavior to that of $\bb{P}_n$). On
the other hand, note that $\bb{P}_n$ is an optimal estimator of $\bb
{P}$ only in
the weak topology and is far from
optimal in the strong
topology as it is not even a consistent estimator of~$\bb{P}$. A~similar
result~-- optimality of kernel density estimator in both
weak and strong topologies~-- was shown by Gin{\'{e}} and Nickl \cite{Gine-08b}
for only $d=1$ where $\Cal{F}$ is chosen to be a unit ball of
bounded Lipschitz functions, $\Cal{F}_{\mathrm{BL}}$ on $\bb{R}^d$, defined
as
%
\begin{equation}
\Cal{F}_{\mathrm{BL}}:= \biggl\{f\dvtx\bb{R}^d\rightarrow\bb{R} \Big|
 \Vert f\Vert_{\mathrm{BL}}:=\sup_{x\in\bb{R}^d}\bigl|f(x)\bigr|+\sup
_{x\ne
y}\frac{|f(x)-f(y)|}{\Vert x-y\Vert_2}\le1 \biggr\},\label{Eq:FB}
\end{equation}
with $\Vert\cdot\Vert_2$ being the Euclidean norm. In comparison, our work
generalizes the result of Gin{\'{e}} and Nickl \cite{Gine-08b} to any
$d$ by working
with $\Cal{F}_H$.

Before presenting our results, in Section~\ref{Sec:rkhs}, we provide a
brief introduction to reproducing kernel Hilbert spaces, discuss
some relevant properties of $\Vert\cdot\Vert_{\Cal{F}_H}$ and provide
concrete examples for $\Cal{F}_H$ through some concrete choices of $\Cal
{K}$. We
then present our first main result in Theorem~\ref{Thm:weak} which
shows that
under certain conditions on $\Cal{K}$, $\Vert
\cdot\Vert_{\Cal{F}_H}$ metrizes the weak topology on the space of
probability measures. Since $\bb{P}_n$ is a consistent estimator
of $\bb{P}$ in weak sense, we then obtain a rate for this convergence
by showing in Theorem~\ref{Thm:consistency} that $\Vert
\bb{P}_n-\bb{P}\Vert_{\Cal{F}_H}=\mathrm{O}_{\mathrm{a.s.}}(n^{-1/2})$ by bounding the expected
suprema of U-processes~-- specifically, the homogeneous Rademacher
chaos process
of degree 2~-- indexed by a uniformly bounded Vapnik-\u{C}ervonenkis
(VC)-subgraph class $\Cal{K}$ (see de la Pe{\~{n}}a and
Gin{\'{e}} \cite{delaPena-99}, Chapter~5 for details on
U-processes).
Since Theorems \ref{Thm:weak} and
\ref{Thm:consistency} are very general, we provide examples (see
Example~\ref{Exm:weak}) to show
that a
large family of $\Cal{K}$ satisfy the assumptions in these results and,
therefore,
yield a variety of probability metrics that metrize the weak
convergence while
ensuring a dimension independent rate of $n^{-1/2}$ for
$\bb{P}_n$ converging to $\bb{P}$.

In Theorem~\ref{Thm:estim-to-p}, we present our second main result which
provides an
exponential inequality for the tail probabilities of $\Vert
\bb{P}^\star_n-\bb{P}_n\Vert_{\Cal{F}_H}=\Vert\bb{P}_n\ast
K_h-\bb{P}_n\Vert_{\Cal{F}_H}$, where $\bb{P}_n\ast K_h$ is the kernel
density estimator with bandwidth $h$, $\ast$ represents the convolution and
$K_h=h^{-d}K(\cdot/h)$ with $K\dvtx\bb{R}^d\rightarrow\bb{R}$. The
proof is based
on an application of McDiarmid's inequality, together with expectation
bounds on
the suprema of homogeneous Rademacher chaos process of degree 2,
indexed over
VC-subgraph classes. For
sufficiently smooth reproducing kernels (see Theorem~\ref
{Thm:estim-to-p} for
details),
this result shows that the kernel density estimator on $\bb{R}^d$ is
within a
$\Vert\cdot\Vert_{\Cal{F}_H}$-ball of size $\mathrm{o}_{\bb{P}}(n^{-1/2})$ around
$\bb{P}_n$ (which means $\Cal{F}_H$ ensures (\ref{Eq:smooth-1})) and,
therefore, combining Theorems~\ref{Thm:weak} and \ref{Thm:estim-to-p} yields
that the
kernel density estimator with $L^1$ optimal
bandwidth is a consistent estimator of $\bb{P}$ in weak sense with a
convergence rate of $n^{-1/2}$ (and hence is optimal in both strong and
weak topologies). We then provide concrete examples of
$\Cal{K}$ in Theorem~\ref{thmm:examples} (also see
Remark~\ref{rem:gauss-exm}) that
guarantee this behavior for the kernel density estimator. Gin{\'{e}}
and Nickl~\cite{Gine-08b}
proved a similar
result
for $\Cal{F}_{\mathrm{BL}}$ with $d=1$ which can be generalized to
any $d\ge2$ using Corollary~3.5 in Sriperumbudur \textit{et al.} \cite
{Sriperumbudur-12}. However, for $d>
2$, it can only be shown that the kernel density estimator
with $L^1$ optimal bandwidth is within $\Vert
\cdot\Vert_{\Cal{F}_{\mathrm{BL}}}$-ball of size $\mathrm{o}_{\bb{P}}(n^{-1/d})$ -- it
is $\mathrm{o}_{\bb{P}}(\sqrt{\log n}/\sqrt{n})$ for $d=2$ -- around
$\bb{P}_n$ instead of $\mathrm{o}_{\bb{P}}(n^{-1/2})$ as with
$\Vert\cdot\Vert_{\Cal{F}_H}$.

Now given that (\ref{Eq:smooth}) holds for $\Cal{F}=\Cal{F}_H$ (see
Theorem~\ref{Thm:estim-to-p} for detailed conditions and
Theorem~\ref{thmm:examples} for examples), it is of interest to know
whether the
processes $f\mapsto\sqrt{n}\int f \,\mathrm{d}(\bb{P}_n\ast
K_h-\bb{P}), f\in\Cal{F}_H$ converge in law to a Gaussian process in
$\ell^\infty(\Cal{F}_H)$. While it is not easy to verify the $\bb{P}$-Donsker
property of $\Cal{F}_H$ or the conditions in Gin{\'{e}} and Nickl (\cite{Gine-08a}, Theorem~3) which
ensure this UCLT in $\ell^\infty(\Cal{F}_H)$ for any general $\Cal{K}$
that induces $\Cal{F}_H$, in Theorem~\ref{pro:singleton}, we present concrete
examples of $\Cal{K}$ for which $\Cal{F}_H$ is $\bb{P}$-Donsker so that
the following UCLTs are
obtained:
\[
\sqrt{n} (\bb{P}_n-\bb{P})\leadsto_{\ell^\infty(\Cal{F}_H)}
\bb{G}_\bb{P} \quad\mbox{and}\quad \sqrt{n} (\bb{P}_n\ast
K_h-\bb{P})\leadsto _{\ell^\infty(\Cal{F}_H)} \bb{G}_\bb{P},
\]
where $\bb{G}_\bb{P}$ denotes the $\bb{P}$-Brownian bridge
indexed by $\Cal{F}_H$ and $\leadsto_{\ell^\infty(\Cal{F}_H)}$
denotes the convergence in law of random elements in
$\ell^\infty(\Cal{F}_H)$. A similar result was presented in
Gin{\'{e}} and Nickl (\cite{Gine-08b}, Theorem~1) for $\Cal{F}_{\mathrm{BL}}$ with
$d=1$ under the condition
that $\bb{P}$ satisfies $\int_{\bb{R}}|x|^{2\gamma} \,\mathrm{d}\bb{P}(x)<\infty$ for
some $\gamma>1/2$, which shows that additional conditions are required on
$\bb{P}$ to obtain a UCLT while working with $\Cal{F}_{\mathrm{BL}}$ in contrast to
$\Cal{F}_H$ where no such conditions are needed.

While the choice of $\Cal{F}_H$ is abstract, there are significant
computational advantages associated with this choice (over say $\Cal{F}_{\mathrm{BL}}$),
which we discuss in Section~\ref{Sec:adaptive}, where we show that
for certain~$\Cal{K}$, it is very easy to compute $\Vert\bb{P}_n\ast
K_h-\bb{P}_n\Vert_{\Cal{F}_H}$ compared to $\Vert
\bb{P}_n\ast K_h-\bb{P}_n\Vert_{\Cal{F}_{\mathrm{BL}}}$ as in the former case, the
problem reduces to a maximization problem over $\bb{R}$ in contrast to an
infinite
dimensional optimization problem in $\Cal{F}_{\mathrm{BL}}$. The need to compute
$\Vert
\bb{P}_n\ast
K_h-\bb{P}_n\Vert_{\Cal{F}}$ occurs while constructing adaptive
estimators that
estimate
$\bb{P}$ efficiently in $\Cal{F}$ and at the
same time estimates the density of $\bb{P}$ (if it exists, but without
a priori
assuming its existence) at the best possible convergence rate in some relevant
loss over prescribed class of densities, for example, sup-norm loss
over the H\"{o}lder balls and $L^1$-loss over Sobolev balls. The
construction of these adaptive estimators involves applying Lepski's
method (Lepski, Mammen and Spokoiny \cite{Lepski-97}) to kernel density
estimators that are within a $\Vert
\cdot\Vert_\Cal{F}$-ball of size smaller than $n^{-1/2}$ around
$\bb{P}_n$, which in turn involves computing $\Vert\bb{P}_n\ast
K_h-\bb{P}_n\Vert_{\Cal{F}}$ (see Gin{\'{e}} and Nickl
\cite{Gine-08b}, Theorem~1, \cite{Gine-09a}, Theorem~2 and
\cite{Gine-10}, Theorem~3). Along the lines of Gin{\'{e}} and Nickl
\cite{Gine-08b}, Theorem~1, in Section~\ref{Sec:adaptive}, we also discuss
the optimal adaptive estimation of $\bb{P}$ in weak and strong topologies.

Various notation
and definitions that are used throughout the paper are collected in
Section~\ref{Sec:notation}. The missing proofs of the results are
provided in
Section~\ref{Sec:proofs} and supplementary results are collected in the \hyperref[app]{Appendix}.

\section{Definitions and notation}\label{Sec:notation}
Let $\Cal{X}$ be a topological space. $\ell^\infty(\Cal{X})$ denotes the
Banach space of bounded real-valued functions $F$ on $\Cal{X}$ normed by
$\Vert F\Vert_\Cal{X}:=\sup_{x\in\Cal{X}}|F(x)|$. $C(\Cal{X})$ denotes
the space
of all continuous real-valued functions on $\Cal{X}$. $C_b(\Cal{X})$ is the
space of all bounded, continuous real-valued functions on~$\Cal{X}$.
For a
locally compact Hausdorff space, $\Cal{X}$, $f\in C(\Cal{X})$ is said to
\emph{vanish at infinity} if for every $\epsilon>0$ the set
$\{x\in\Cal{X}\dvt |f(x)|\ge\epsilon\}$ is compact. The class of all
continuous $f$
on $\Cal{X}$ which vanish at infinity is denoted as $C_0(\Cal{X})$. The spaces
$C_b(\Cal{X})$ and $C_0(\Cal{X})$ are endowed with the uniform norm,
$\Vert\cdot\Vert_\Cal{X}$, which we alternately denote as $\Vert
\cdot\Vert_\infty$. $M^1_+(\Cal{X})$ denotes the space of all Borel probability
measures defined on $\Cal{X}$ while $M_b(\Cal{X})$ denotes the space of all
finite signed Borel measures on $\Cal{X}$. $L^r(\Cal{X},\mu)$ denotes
the Banach
space of $r$-power ($r\ge1$) $\mu$-integrable functions where $\mu$ is
a Borel
measure defined on $\Cal{X}$. We will write $L^r(\Cal{X})$ for
$L^r(\Cal{X},\mu)$ if $\mu$ is a Lebesgue measure on $\Cal{X}\subset\bb{R}^d$.
$W^s_1(\bb{R}^d)$ denotes the space of functions $f\in L^1(\bb{R}^d)$
whose partial derivatives up to order $s\in\bb{N}$ exist and are in
$L^1(\bb{R}^d)$. $\Cal{F}_{\mathrm{BL}}$ denotes the unit ball of
bounded Lipschitz functions on $\Cal{X}=\bb{R}^d$ as shown in (\ref
{Eq:FB}). A
function $k\dvtx\Cal{X}\times\Cal{X}\rightarrow\bb{C}$ is called a
positive definite (p.d.) kernel if, for all $n\in\bb{N}$,
$(\alpha_1,\ldots,\alpha_n)\in\bb{C}^n$ and all ($x_1,\ldots,x_n)\in
\Cal{X}^n$,
we
have
\[
\sum^n_{i,j=1}
\alpha_i\overline{\alpha}_jk(x_i,x_j)
\ge0,
\]
where $\overline{\alpha}$ is the complex conjugate of
$\alpha\in\bb{C}$. $\Cal{H}_k$ denotes a reproducing kernel Hilbert space
(RKHS) (see
Definition~\ref{fin:rkhs}) of functions with a positive definite $k$ as
the reproducing kernel and $\langle\cdot,\cdot\rangle_{\Cal{H}_k}$
denotes the
inner product on $\Cal{H}_k$. $\Cal{F}_H$ denotes the unit ball of RKHS
functions
indexed
by a cone of positive definite kernels, $\Cal{K}$ as shown in (\ref{Eq:FH}).
The convolution $f\ast g$
of two measurable functions $f$ and $g$ on $\bb{R}^d$ is defined as
$(f\ast
g)(x):=\int_{\bb{R}^d}f(y)g(x-y) \,\mathrm{d}y$, provided the integral exists for all
$x\in\bb{R}^d$. Similarly, the convolution of $\mu\in M_b(\bb{R}^d)$ and
measurable $f$ is defined as
\[
(f\ast\mu) (x):=\int_{\bb{R}^d}f(x-y) \,\mathrm{d}\mu(y)
\]
if
the integral exists for all $x\in\bb{R}^d$. For $f\in
L^1(\bb{R}^d)$, its Fourier transform is defined
as
\[
\widehat{f}(y)=(2\pi)^{-d/2}\int_{\bb{R}^d}f(x)
\mathrm{e}^{-\sqrt{-1}\langle
y,x\rangle} \,\mathrm{d}x.
\]
%
A sequence of probability measures,
$(\bb{P}_{(n)})_{n\in\bb{N}}$ is said to
converge weakly to $\bb{P}$ (denoted
as $\bb{P}_{(n)}\leadsto\bb{P}$) if and only if
$\int f \,\mathrm{d}\bb{P}_{(n)}\rightarrow\int f \,\mathrm{d}\bb{P}$ for all $f\in C_b(\Cal{X})$
as
$n\rightarrow
\infty$. For a Borel-measurable real-valued function $f$ on $\Cal{X}$ and
$\mu\in M_b(\Cal{X})$, we define $\mu f:=\int_{\Cal{X}} f \,\mathrm{d}\mu$. The empirical
process indexed by $\Cal{F}\subset L^2(\Cal{X},\bb{P})$ is given by
$f\mapsto
\sqrt{n}(\bb{P}_n-\bb{P})f=n^{-1/2}\sum^n_{i=1}(f(X_i)-\bb{P}f)$, where
$\bb{P}_n:=\frac{1}{n}\sum^n_{i=1}\delta_{X_i}$ with $(X_i)^n_{i=1}$ being
random samples drawn i.i.d. from $\bb{P}$ and $\delta_x$ represents the Dirac
measure at $x$. $\Cal{F}$ is said to be \emph{$\bb{P}$-Donsker} if
$\sqrt{n}(\bb{P}_n-\bb{P})\leadsto_{\ell^\infty(\Cal{F})}\bb{G}_\bb
{P}$, where
$\bb{G}_\bb{P}$ is the Brownian bridge indexed by $\Cal{F}$, that is, a centered
Gaussian process with covariance
$\bb{E}\bb{G}_\bb{P}(f)\bb{G}_\bb{P}(g)=\bb{P}((f-\bb{P}f)(g-\bb{P}g))$
and if
$\bb{G}_\bb{P}$ is sample-bounded and sample-continuous w.r.t. the covariance
metric. $\leadsto_{\ell^\infty(\Cal{F})}$ denotes the convergence in
law (or
weak convergence) of random elements in $\ell^\infty(\Cal{F})$. $\Cal
{F}$ is
said to be \emph{universal Donsker} if it is $\bb{P}$-Donsker for all
$\bb{P}\in M^1_+(\Cal{X})$.

Let $\Cal{C}$ be a collection of subsets of a set $\Cal{X}$. The collection
$\Cal{C}$ is said to \emph{shatter} an arbitrary set of $n$
points, $\{x_1,\ldots,x_n\}$, if for each of its $2^n$ subsets, there exists
$C\in\Cal{C}$ such that $C\cap\{x_1,\ldots,x_n\}$ yields the subset. The
\emph{Vapnik-\u{C}ervonenkis (VC)-index}, $\mathit{VC}(\Cal{C})$ of the class
$\Cal{C}$
is
the maximal $n$ for which an $n$-point set is shattered by $\Cal{C}$.
If $\mathit{VC}(\Cal{C})$ is
finite, then $\Cal{C}$ is said to be a \emph{VC-class}. A~collection
$\Cal{F}$
of real-valued functions on $\Cal{X}$ is called a \emph{VC-subgraph
class} if
the collection of all subgraphs of the functions in $\Cal{F}$, that is,
$ \{\{(x,t)\dvt t<f(x)\}\dvt f\in\Cal{F} \}$ forms a VC-class of sets in
$\Cal{X}\times\bb{R}$. The \emph{covering number}
$\Cal{N}(\Cal{F},\rho,\epsilon)$ is the minimal number of balls
$\{g\dvt \rho(f,g)<\epsilon\}$ of radius $\epsilon$ needed to cover $\Cal
{F}$, where $\rho$ is a metric on $\Cal{F}$.

Given random samples $(X_i)^n_{i=1}\subset\bb{R}^d$ drawn i.i.d. from
$\bb{P}$,
the kernel density estimator is defined as
\[
(\bb{P}_n\ast K_h) (x)=\frac{1}{nh^d}\sum
^n_{i=1}K \biggl(\frac{x-X_i}{h} \biggr),\qquad x\in
\bb{R}^d,
\]
where $K\dvtx\bb{R}^d\rightarrow\bb{R}$ is the smoothing kernel that satisfies
$K(x)=K(-x), x\in\bb{R}^d$, $K\in L^1(\bb{R}^d)$ and $\int_{\bb{R}^d}
K(x) \,\mathrm{d}x=1$ with $K_h(x):=h^{-d}K(x/h)$ and
$0<h:=h_n\rightarrow0$ as $n\rightarrow\infty$. $K$ is said
to be of order $r>0$ if
\begin{eqnarray*}
\int_{\bb{R}^d}\prod^d_{i=1}y^{\alpha_i}_i
K(y) \,\mathrm{d}y&=&0 \qquad\mbox{for } 0<|\alpha|\le r-1\quad\mbox{and}\\
 \int_{\bb{R}^d}
\prod^d_{i=1}|y_i|^{\alpha_i}\bigl|K(y)\bigr|
\,\mathrm{d}y&<&\infty \qquad\mbox{for } |\alpha|=r,
\end{eqnarray*}
where $y=(y_1,\ldots,y_d)$,
$\alpha=(\alpha_1,\ldots,\alpha_d)$, $\alpha_i\ge0, \forall i=1,\ldots
,d$ and
$|\alpha|:=\sum^d_{i=1}\alpha_i$. We refer the reader to
Berlinet and Thomas-Agnan (\cite{Berlinet-04}, Chapter~3, Section~8) for
details about the construction of
kernels of arbitrary order, $r$.

We would like to mention that throughout the paper, we ignore the measurability
issues that are associated with the suprema of an empirical process (or a
U-process) and therefore the probabilistic statements about these objects
should be considered in the outer measure.

\section{Reproducing kernel Hilbert spaces and
\texorpdfstring{$\Vert\cdot\Vert_{\Cal{F}_{H}}$}{$||cdot||_{F_{H}}$}}\label{Sec:rkhs}

In this section, we present a brief overview of RKHS along with some properties
of $\Vert\cdot\Vert_{\Cal{F}_H}$ with a goal to provide an
intuitive understanding of, otherwise an abstract class $\Cal{F}_H$ and its
associated distance, $\Vert\cdot\Vert_{\Cal{F}_H}$. Throughout this
section, we assume that $\Cal{X}$ is a topological space.

\subsection{Preliminaries}
We start with the definition of an
RKHS, which we quote from Berlinet and Thomas-Agnan \cite{Berlinet-04}.
For the purposes of this paper,
we deal with real-valued RKHS though the following definition can be extended
to the complex-valued case (see Berlinet and Thomas-Agnan
\cite{Berlinet-04}, Chapter~1, Definition~1).

\begin{fin}[(Reproducing kernel Hilbert space)]\label{fin:rkhs}
Let $(\Cal{H}_k,\langle
\cdot,\cdot\rangle_{\Cal{H}_k})$ be a Hilbert space of real-valued
functions on
$\Cal{X}$. A function $k\dvtx\Cal{X}\times\Cal{X}\rightarrow\bb{R},
(x,y)\mapsto
k(x,y)$ is
called a \emph{reproducing kernel} of the Hilbert space $\Cal{H}_k$ if
and only
if the
following hold:
\begin{longlist}[(ii)]
\item[(i)] $\forall y\in\Cal{X}$, $k(\cdot,y)\in\Cal{H}_k$;
\item[(ii)] $\forall y\in\Cal{X}$, $\forall f\in\Cal{H}_k$, $\langle
f,k(\cdot,y)\rangle_{\Cal{H}_k}=f(y)$.
\end{longlist}
If such a $k$ exists, then $\Cal{H}_k$ is called a \emph{reproducing kernel
Hilbert space}.
\end{fin}

Using the Riesz representation theorem, the
above definition can be shown to be
equivalent to defining $\Cal{H}_k$ as an RKHS if for all
$x\in\Cal{X}$, the evaluation functional, $\delta_x\dvtx\Cal
{H}_k\rightarrow\bb{R}$,
$\delta_x(f):=f(x), f\in\Cal{H}_k$ is continuous (Berlinet and
Thomas-Agnan \cite{Berlinet-04}, Chapter~1, Theorem~1). Starting from Definition~\ref{fin:rkhs}, it can be
shown that $\Cal{H}_k=\overline{\operatorname{span}}\{k(\cdot,x)\dvt x\in\Cal{X}\}$
where the closure is taken w.r.t. the RKHS norm
(see Berlinet and Thomas-Agnan \cite{Berlinet-04}, Chapter~1, Theorem~3), which means the kernel
function, $k$ generates the RKHS. Since $\langle
k(\cdot,x),k(\cdot,y)\rangle_{\Cal{H}_k}=k(x,y), \forall x,y\in\Cal
{X}$, it
is easy to show that every reproducing kernel (r.k.), $k$ is symmetric
and positive definite. More interestingly, the converse is also true,
that is, the Moore--Aronszajn theorem (Aronszajn \cite{Aronszajn-50})
states that for
every positive definite kernel, $k$, there exists a unique RKHS, $\Cal{H}_k$
with $k$ as the r.k. Since $k$ is a reproducing kernel if and only if
it is
positive definite, usually it might be simpler to verify for the positive
definiteness of $k$ rather than explicitly constructing $\Cal{H}_k$ and
verifying whether $k$ satisfies the properties in
Definition~\ref{fin:rkhs}. An
important characterization for positive definiteness on $\bb{R}^d$ (more
generally on locally compact Abelian groups) is given by
Bochner's theorem (Wendland \cite{Wendland-05}, Theorem~6.6): a bounded continuous
translation invariant kernel
$k(x,y)=\psi(x-y)$ on $\bb{R}^d$ is positive definite if and only if
$\psi$ is
the Fourier transform of
a nonnegative finite Borel measure, $\Upsilon$, that is,
%
\begin{equation}
k(x,y)=\psi(x-y)=\int \mathrm{e}^{-\sqrt{-1}(x-y)^T\omega} \,\mathrm{d}\Upsilon(\omega), \qquad x,y\in
\bb{R}^d. \label{Eq:Bochner}
\end{equation}
In addition, if $\psi\in L^1(\bb{R}^d)$, then the
corresponding RKHS is given by Wendland (see \cite{Wendland-05}, Theorem~10.12),
\[
\Cal{H}_k= \biggl\{f\in L^2\bigl(\bb{R}^d
\bigr)\cap C\bigl(\bb{R}^d\bigr) \dvt \int \frac{|\widehat{f}(\omega)|^2}{\widehat{\psi}(\omega)} \,\mathrm{d}\omega<
\infty \biggr\},
\]
where $\widehat{f}$ and $\widehat{\psi}$ denote the Fourier transforms of
$f$ and $\psi$, respectively. Note that since $\psi\in L^1(\bb{R}^d)$,
we have
$\mathrm{d}\Upsilon(\omega)=(2\pi)^{-d/2}\widehat{\psi}(\omega) \,\mathrm{d}\omega$. Another
characterization of positive definiteness (which we will use later in our
results) is due to Sch\"{o}nberg for radially symmetric positive definite
kernels (Wendland \cite{Wendland-05}, Theorems 7.13 and 7.14): A function
$k(x,y)=\phi(\Vert
x-y\Vert^2_2), x,y\in\bb{R}^d$ is positive definite if and only if $\phi
$ is
the Laplace transform of a nonnegative finite Borel measure, $\nu$ on
$[0,\infty)$, that is,
%
\begin{equation}
k(x,y)=\phi\bigl(\Vert x-y\Vert^2_2\bigr)=\int
^\infty_0 \mathrm{e}^{-t\Vert
x-y\Vert^2_2} \,\mathrm{d}\nu(t),\qquad x,y\in
\bb{R}^d.\label{Eq:schoenberg}
\end{equation}
Note that the Bochner's characterization in (\ref{Eq:Bochner}) is also valid
for $k$ in (\ref{Eq:schoenberg}) as it is also translation invariant. Two
important examples of positive definite kernels and their corresponding RKHSs
that appear throughout the paper are: \emph{Gaussian kernel},
$k(x,y)=\exp(-\sigma\Vert
x-y\Vert^2_2), x,y\in\bb{R}^d, \sigma>0$, which induces the \emph{Gaussian
RKHS},
%
\begin{equation}
\Cal{H}_k= \biggl\{f\in L^2\bigl(\bb{R}^d
\bigr)\cap C\bigl(\bb{R}^d\bigr) \dvt \int \bigl|\widehat{f}(
\omega)\bigr|^2\mathrm{e}^{\Vert
\omega\Vert^2_2/4\sigma} \,\mathrm{d}\omega<\infty \biggr\}\label{Eq:gauss-rkhs}
\end{equation}
and the \emph{Mat\'{e}rn
kernel},
$k(x,y)=\frac{2^{1-\beta}}{\Gamma(\beta)}\Vert
x-y\Vert^{\beta-d/2}_2\frak{K}_{d/2-\beta}(\Vert
x-y\Vert_2), x,y\in\bb{R}^d, \beta>d/2$, which induces the Sobolev space,
$H^\beta_2$,
%
\begin{equation}
\Cal{H}_k=H^\beta_2= \biggl\{f\in
L^2\bigl(\bb{R}^d\bigr)\cap C\bigl(\bb{R}^d
\bigr) \dvt \int \bigl(1+\Vert \omega\Vert^2_2
\bigr)^\beta\bigl|\widehat{f}(\omega)\bigr|^2 \,\mathrm{d}\omega<\infty \biggr\}.
\label{Eq:sobolev}
\end{equation}
Here, $\Gamma$ is the Gamma
function and $\frak{K}_v$ is the modified Bessel
function of the third kind of order~$v$, where $v$ controls the
smoothness of
$k$. Note that $L^2(\bb{R}^d)$ is not an RKHS as it does not consist of
functions.

\subsection{Properties of \texorpdfstring{$\Vert\cdot\Vert_{\Cal{F}_H}$}{$||cdot||_{F_H}$}}
In the following, we present various properties which are not only
helpful to
intuitively understand $\Vert\cdot\Vert_{\Cal{F}_H}$ but are also
useful to
derive our main results in Section~\ref{Sec:mainresult}. First, in
Proposition~\ref{pro:embedding}, we provide an
alternate expression
for $\Vert\cdot\Vert_{\Cal{F}_H}$ to obtain a better interpretation, using
which we discuss the relation of
$\Vert\cdot\Vert_{\Cal{F}_H}$ to other classical distances on probabilities.
This alternate representation will be particularly helpful in studying the
convergence of $\bb{P}_n$ to $\bb{P}$ in $\Vert\cdot\Vert_{\Cal{F}_H}$
and also
in deriving our main result (Theorem~\ref{Thm:estim-to-p}) in
Section~\ref{Sec:mainresult}. Second, we discuss
the question of the metric property of
$\Vert\cdot\Vert_{\Cal{F}_H}$ -- it is easy to
verify that $\Vert\cdot\Vert_{\Cal{F}_H}$ is a pseudometric~-- and highlight
some of the results that we obtained
in our earlier works along with some examples in Example~\ref{Exm:charac}.
Third, we present a new result in Theorem~\ref{Thm:weak} about the topology
induced by $\Vert\cdot\Vert_{\Cal{F}_H}$ wherein we show that under
certain mild
conditions on $\Cal{K}$, $\Vert\cdot\Vert_{\Cal{F}_H}$ metrizes the
weak topology on probability measures. We also present some examples of
$\Cal{K}$ in Example~\ref{Exm:weak} that satisfy these conditions thereby
ensuring the metrization of weak topology by the corresponding metric,
$\Vert\cdot\Vert_{\Cal{F}_H}$. Finally, in Theorem~\ref
{Thm:consistency}, we
present an exponential concentration inequality for the tail
probabilities of
$\Vert\bb{P}_n-\bb{P}\Vert_{\Cal{F}_H}$ and show that for various
families of
$\Cal{K}$ (see Remark~\ref{Rem:supp}(i) and
Theorem~\ref{thmm:examples}), $\Vert\bb{P}_n-\bb{P}\Vert_{\Cal{F}_H}=\mathrm{O}_{\mathrm{a.s.}}
(n^{-1/2})$, which when
combined with Theorem~\ref{Thm:weak} provides a rate of convergence of
$n^{-1/2}$ for $\bb{P}_n$ converging to $\bb{P}$ in weak sense.

\textit{Alternate representation for}
$\Vert\cdot\Vert_{\Cal{F}_H}$: The following result (a similar
result is
proved in Sriperumbudur
\textit{et al.} \cite{Sriperumbudur-10a}, Theorem~1, where $\Cal{K}$ is chosen
to be a
singleton set but
we provide a proof in Section~\ref{subsec:embedding} for completeness) presents
an alternate
representation to $\Vert\cdot\Vert_{\Cal{F}_H}$. This representation is
particularly useful as it shows that $\Vert\cdot\Vert_{\Cal{F}_H}$ is
completely determined by the kernels, $k\in\Cal{K}$ and does not depend
on the
individual functions in the corresponding RKHSs.

\begin{pro}\label{pro:embedding}
Define
$\Scr{P}_\Cal{K}:= \{\bb{P}\in M^1_+(\Cal{X})\dvt \sup_{k\in\Cal{K}}\int
\sqrt{k(x,x)} \,\mathrm{d}\bb{P}(x)<\infty
 \}$ where every $k\in\Cal{K}$, $k\dvtx\Cal{X}\times\Cal
{X}\rightarrow\bb{R}$ is measurable.
Then for any
$\bb{P},\bb{Q}\in\Scr{P}_\Cal{K}$,
%
\begin{equation}
\Vert\bb{P}-\bb{Q}\Vert_{\Cal{F}_H} =\sup_ { k\in\Cal{ K } }
\fr{D}_k(\bb{P},\bb{Q}),\label{Eq:prop}\
\end{equation}
where
%
\begin{eqnarray}
\fr{D}_k(\bb{P},\bb{Q})&:=&\biggl\llVert \int k(\cdot,x) \,\mathrm{d}
\bb{P}(x)-\int k(\cdot,x) \,\mathrm{d}\bb{Q}(x)\biggr\rrVert _{\Cal{H}_k}\label{Eq:MMD}
\\
&=&\sqrt{\int\int k(x,y) \,\mathrm{d}(\bb{P}-\bb{Q}) (x) \,\mathrm{d}(\bb{P}-\bb{Q})
(y)},\label{Eq:MMD-1}
\end{eqnarray}
with $\int k(\cdot,x) \,\mathrm{d}\bb{P}(x)$ and $\int k(\cdot,x) \,\mathrm{d}\bb{Q}(x)$
being defined in Bochner sense (Diestel and Uhl \cite{Diestel-77}, Definition~1).
\end{pro}

Since $\sqrt{k(x,x)}=\Vert k(\cdot,x)\Vert_{\Cal{H}_k}$, it is easy to
verify from
Proposition~\ref{pro:embedding} that for any $\bb{P},\bb{Q}\in\Scr
{P}_\Cal{K}$, $\Vert\bb{P}-\bb{Q}\Vert_{\Cal{F}_H}<\infty$.
It also follows from (\ref{Eq:prop}) and (\ref{Eq:MMD})
that $\Vert\bb{P}-\bb{Q}\Vert_{\Cal{F}_H}$ can
be
interpreted as the supremum distance between the
embeddings $\bb{P}\mapsto\int
k(\cdot,x) \,\mathrm{d}\bb{P}(x)$ and $\bb{Q}\mapsto\int
k(\cdot,x) \,\mathrm{d}\bb{Q}(x)$, indexed by $k\in\Cal{K}$.
%
Choosing
$k(\cdot,x)$ as
%
\begin{equation}
\frac{1}{(2\pi)^{d/2}}\mathrm{e}^{-\sqrt{-1}\langle
\cdot,x\rangle_2}, \qquad\mathrm{e}^{\langle\cdot,x\rangle_2} \quad\mbox{and}\quad
\frac{1}{(4\pi)^{d/2}}\mathrm{e}^{-\Vert
\cdot-x\Vert^2_2/4},\qquad x\in\bb{R}^d,\label{Eq:choice}
\end{equation}
the
embedding $\Phi\dvtx M^1_+(\Cal{X})\rightarrow\Cal{H}_k$, $\bb{P}\mapsto\int
k(\cdot,x) \,\mathrm{d}\bb{P}(x)$ reduces to the
characteristic function, moment generating function (if it exists)
and Weierstrass transform of $\bb{P}$, respectively. In this sense, $\Phi
$ can be
seen as a generalization of these notions (which are all defined on
$\bb{R}^d$) to an arbitrary topological space $\Cal{X}$ (in fact, it
holds for
any arbitrary measurable space).

\textit{When is} $\Vert\cdot\Vert_{\Cal{F}_H}$ \textit{a metric on} $\Scr{P}_\Cal{K}$?
 While $\Vert\cdot\Vert_{\Cal{F}_H}$ is a pseudo-metric on
$\Scr{P}_\Cal{K}$, it
is in general not
a metric as $\Vert\bb{P}-\bb{Q}\Vert_{\Cal{F}_H}=0\not\Longrightarrow
\bb{P}=\bb{Q}$ as shown by
the choice $\Cal{K}=\{k\}$ where $k(x,y)=\langle
x,y\rangle_2, x,y\in\bb{R}^d$. For this choice, it
is easy to
check that $\Vert\bb{P}-\bb{Q}\Vert_{\Cal{F}_H}$ is the Euclidean distance
between the means of $\bb{P}$
and $\bb{Q}$ and, therefore, is not a metric on $\{\bb{P}\in
M^1_+(\bb{R}^d)\dvt \int\Vert x\Vert \,\mathrm{d}\bb{P}(x)<\infty\}$ (and hence on
$M^1_+(\bb{R}^d)$).
The
question of when is $\fr{D}_k$ a metric on $M^1_+(\Cal{X})$ is
addressed in
Fukumizu \textit{et al.} \cite{Fukumizu-08a,Fukumizu-08b}, Gretton \textit{et al.} \cite
{Gretton-06} and Sriperumbudur
\textit{et al.} \cite{Sriperumbudur-10a}. By defining any
kernel for which $\fr{D}_k$ is a metric as the \emph{characteristic
kernel}, it
is easy to see that if any $k\in\Cal{K}$ is characteristic, then
$\Vert\cdot\Vert_{\Cal{F}_H}$ is a
metric on $\Scr{P}_\Cal{K}$. Sriperumbudur
\textit{et al.} (\cite{Sriperumbudur-10a}, Theorem~7) showed that $k$
is characteristic if and
only if
%
\begin{equation}
\int\int k(x,y) \,\mathrm{d}\mu(x) \,\mathrm{d}\mu(y)>0\qquad \forall\mu\in M_b(\Cal{X})
\setminus\{0\} \mbox{ with } \mu(\Cal{X})=0.\label{Eq:charac}
\end{equation}
Combining this with the Bochner
characterization for positive definiteness (see (\ref{Eq:Bochner})),
Sriperumbudur
\textit{et al.} (\cite{Sriperumbudur-10a}, Corollary~4)
showed that
\[
\frak{D}_k(\bb{P},\bb{Q})=\llVert \phi_\bb{P}-
\phi_\bb{Q}\rrVert _{L^2(\bb{R}^d,\Upsilon)}, \qquad\bb{P},\bb {Q}\in
M^1_+\bigl(\bb{R}^d\bigr),
\]
using which $k$ is shown to be characteristic if and only if
$\operatorname{supp}(\Upsilon)=\bb{R}^d$ (Sriperumbudur
\textit{et al.} \cite{Sriperumbudur-10a}, Theorem~9)~-- Fukumizu \textit{et al.} \cite{Fukumizu-08b} generalized
this result to locally
compact Abelian groups, compact non-Abelian groups and the semigroup
$\bb{R}^d_+$. Here, $\phi_\bb{P}$ and $\phi_\bb{Q}$ represent the
characteristic functions of $\bb{P}$ and $\bb{Q}$, respectively. Another
interesting characterization for the characteristic property of $k$ is obtained
by Fukumizu \textit{et al.} \cite{Fukumizu-08a,Fukumizu-08b} and Gretton \textit{et al.}
\cite{Gretton-06}, which relates it to the richness
of $\Cal{H}_k$ in the sense of approximating certain classes of
functions by
functions in $\Cal{H}_k$. We refer the reader to Sriperumbudur,
Fukumizu and
Lanckriet \cite{Sriperumbudur-11a} for
more
details on the relation
between the characteristic property of $k$ and the richness of $\Cal{H}_k$.

\begin{exm}\label{Exm:charac}
The following are some examples of $\Cal{K}$ for which
$\Vert\cdot\Vert_{\Cal{F}_H}$ is a
metric on
$M^1_+(\bb{R}^d)$:
%
\begin{enumerate}
\item Gaussian: $\Cal{K}= \{\mathrm{e}^{-\sigma\Vert
x-y\Vert^2_2}, x,y\in\bb{R}^d \dvt \sigma\in(0,\infty) \}$;
\item Laplacian: $\Cal{K}= \{\mathrm{e}^{-\sigma\Vert
x-y\Vert_1}, x,y\in\bb{R}^d \dvt \sigma\in(0,\infty) \}$;
\item Mat\'{e}rn:
\[
\Cal{K}= \biggl\{\frac{2({c}/{2})^{\beta-{d}/{2}}}{\Gamma(\beta
-{d}/{2})} \Vert x-y\Vert^{\beta-{d}/{2}}_2
\frak{K}_{{d}/{2}-\beta} \bigl(c\Vert x-y\Vert_2 \bigr), x,y\in
\bb{R}^d, \beta>\frac{d}{2}\dvt c\in(0,\infty ) \biggr\},
\]
where $\frak{K}_v$ is the modified Bessel function of
the third kind of order $v$;
\item Inverse multiquadrics:
$\Cal{K}= \{ (1+\llVert \frac{x-y}{c}\rrVert ^2_2
)^{-\beta},
x,y\in\bb{R}^d, \beta>0 \dvt c\in(0,\infty)  \}$;
\item Splines:
\[
\Cal{K}= \Biggl\{\prod^d_{j=1} \biggl(1-
\frac{|x_j-y_j|}{c_j} \biggr)\mathds{1}_{
 \{|x_j-y_j|\le
c_j \}}, x,y\in\bb{R}^d
\dvt c_j\in (0,\infty), \forall j=1,\ldots,d \Biggr\};
\]
\item Radial basis functions:
\[
\Cal{K}= \biggl\{\int_{(0,\infty)}\mathrm{e}^{-\sigma\Vert
x-y\Vert^2_2} \,\mathrm{d}\Lambda(
\sigma), x,y\in\bb{R}^d \dvt \Lambda\in M^1_+\bigl((0,
\infty)\bigr) \biggr\}.
\]
\end{enumerate}
In all these examples, it is easy to check that every $k\in\Cal{K}$ is bounded
and characteristic (as $\operatorname{supp}(\Upsilon)=\bb{R}^d$ or in turn satisfies
(\ref{Eq:charac})) and, therefore, $\Vert\cdot\Vert_{\Cal{F}_H}$ is a
metric on
$M^1_+(\bb{R}^d)$.
\end{exm}

\textit{Topology induced by} $\Vert\cdot\Vert_{\Cal{F}_H}$:
Sriperumbudur
\textit{et al.} \cite{Sriperumbudur-10a} showed that for any bounded kernel $k$,
\[
\fr{D}_k(\bb{P},\bb{Q})\le\sup_{x\in\Cal{X}}
\sqrt{k(x,x)} \operatorname{TV}(\bb{P},\bb{Q}),
\]
where $\operatorname{TV}$ is the
total variation distance.
This means
there can be two
distinct $\bb{P}$ and $\bb{Q}$ which need not be distinguished by
$\fr{D}_k$ but are distinguished in total variation, that is, $\fr
{D}_k$ induces a
topology that is weaker (or coarser) than the strong topology on
$M^1_+(\Cal{X})$. Therefore, it is of interest to understand the topology
induced by $\Vert\cdot\Vert_{\Cal{F}_H}$. The following result shows
that under
additional conditions
on $\Cal{K}$, $\Vert\cdot\Vert_{\Cal{F}_H}$ metrizes the weak-topology on
$M^1_+(\Cal{X})$. A special case of this result is already proved in
Sriperumbudur
\textit{et al.} (\cite{Sriperumbudur-10a}, Theorem~23),
for $\fr{D}_k$ when $\Cal{X}$ is compact.

\begin{thmm}\label{Thm:weak}
Let $\Cal{X}$ be a Polish space that is locally compact Hausdorff. Suppose
$(\bb{P}_{(n)})_{n\in\bb{N}}\subset M^1_+(\Cal{X})$ and $\bb{P}\in
M^1_+(\Cal{X})$.

\textup{(a)} If there
exists a $k\in\Cal{K}$ such that $k(\cdot,x)\in C_0(\Cal{X})$ for all
$x\in
\Cal{X}$ and
%
\begin{equation}
\int\int k(x,y) \,\mathrm{d}\mu(x) \,\mathrm{d}\mu(y)>0\qquad\forall\mu\in M_b(\Cal{X})
\setminus\{0\}.\label{Eq:dense}
\end{equation}
Then
\[
\llVert \bb{P}_{(n)}-\bb{P}\rrVert _{\Cal{F}_H}\rightarrow 0
\quad\Longrightarrow\quad \bb{P}_{(n)}\leadsto \bb{P} \qquad\mbox{as } n\rightarrow
\infty.
\]
\textup{(b)} If $\Cal{K}$ is uniformly bounded, that is, $\sup_{k\in\Cal{K},x\in
\Cal{X}}k(x,x)<\infty$ and satisfies the following property $(\mathrm{P})$:
\[
\forall x\in \Cal{X}, \forall\epsilon>0, \exists \mbox{ open } U_{x,\epsilon}
\subset \Cal{X} \mbox{ such that } \bigl\Vert k(\cdot,x)-k(\cdot,y)\bigr\Vert_{\Cal{H}_k}<
\epsilon, \forall k\in\Cal{K}, \forall y\in U_{x,\epsilon}.
\]
Then
\[
\bb{P}_{(n)}\leadsto\bb{P}\quad \Longrightarrow \quad \llVert
\bb{P}_{(n)}-\bb{P}\rrVert _{\Cal{F}_H}\rightarrow 0 \qquad\mbox{as } n
\rightarrow\infty.
\]
\end{thmm}

\begin{pf}
(a) Define $\Cal{H}_\ast$ to be the RKHS associated with the
reproducing kernel $k_\ast$. Suppose $k_\ast\in\Cal{K}$
satisfies $k_*(\cdot,x)\in C_0(\Cal{X}), \forall x\in\Cal{X}$. By
Steinwart and Christmann \cite{Steinwart-08}, Lemma~4.28
(see Sriperumbudur, Fukumizu and
Lanckriet \cite{Sriperumbudur-11c}, Theorem~5), it follows
that the inclusion
$\mathrm{id}\dvtx\Cal{H}_{\ast}\rightarrow C_0(\Cal{X})$ is well-defined and
continuous. In addition, since $k_\ast\in\Cal{K}$ satisfies
(\ref{Eq:dense}), it follows from Sriperumbudur, Fukumizu and
Lanckriet (\cite{Sriperumbudur-11a}, Proposition~4), that
$\Cal{H}_{*}$ is
dense in $C_0(\Cal{X})$ w.r.t. the uniform norm. We would like to mention
that this denseness result simply follows from the Hahn--Banach theorem
(Rudin \cite{Rudin-91}, Theorem~3.5 and the remark following Theorem~3.5),
which says that $\Cal{H}_{\ast}$ is dense in $C_0(\Cal{X})$ if and only if
\[
\Cal{H}^\perp_{\ast}:= \biggl\{\mu\in M_b(
\Cal{X})\dvt \int f \,\mathrm{d}\mu=0, \forall f\in\Cal{H}_{\ast} \biggr\}=\{0\}.
\]
It is easy to
check that $\Cal{H}^\perp_{\ast}=\{0\}$ if and only if
\[
\mu\mapsto\int k(\cdot,x) \,\mathrm{d}\mu(x), \qquad\mu\in M_b(\Cal{X})
\]
is injective, which is then
equivalent to (\ref{Eq:dense}). Since $\Cal{H}_\ast$ is dense in
$C_0(\Cal{X})$
in the uniform norm, for any $f\in C_0(\Cal{X})$ and every $\epsilon
>0$, there
exists a $g\in\Cal{H}_{*}$ such that $\Vert f-g\Vert_\infty\le
\epsilon$.
Therefore,
\begin{eqnarray*}
\llvert \bb{P}_{(n)}f-\bb{P}f\rrvert &=&\bigl\llvert
\bb{P}_{(n)}(f-g)+\bb{P} (g-f)+(\bb{ P } _ { (n) } g-\bb{ P } g)
\bigr\rrvert
\\
&\le& \bb{P}_{(n)}|f-g|+\bb{P}|f-g|+\llvert
\bb{P}_{(n)}g-\bb{P}g\rrvert
\\
&\le& 2\epsilon+\llvert \bb{P}_{(n)}g-\bb{P}g\rrvert
\\
&\le& 2\epsilon+\Vert g\Vert_{\Cal{H}_{*}}\fr{D}_{k_*}(
\bb{P}_{(n)},\bb{P})
\le2\epsilon+\Vert g
\Vert_{\Cal{H}_{*}}\llVert \bb{P}_{(n)}-\bb{P}\rrVert _{\Cal{F}_H}
.
\end{eqnarray*}
Since $\epsilon>0$ is arbitrary, $\Vert g\Vert_{\Cal{H}_\ast}<\infty$
and $\Vert\bb{P}_{(n)}-\bb{P}\Vert_{\Cal{F}_H}\rightarrow0$ as
$n\rightarrow
\infty$, we
have $\bb{P}_{(n)}f\rightarrow\bb{P}f$ for all $f\in C_0(\Cal{X})$ as
$n\rightarrow
\infty$, which means $\bb{P}_{(n)}$ converges to $\bb{P}$ vaguely.
Since vague
convergence and weak convergence are equivalent on
the set of Radon probability measures (Berg, Christensen and
Ressel \cite{Berg-84},
page 51), which is same
as $M^1_+(\Cal{X})$ since $\Cal{X}$ is Polish, the
result
follows.

(b) Since $\Cal{K}$ is uniformly bounded, it is easy to see that
\[
\sup_{f\in\Cal{F}_H,x\in
\Cal{X}}\bigl|f(x)\bigr|=\sup_{k\in\Cal{K},x\in\Cal{X}}\sup
_{\Vert
f\Vert_{\Cal{H}_k\le1}}\bigl|\bigl\langle f,k(\cdot,x)\bigr\rangle_{\Cal{H}_k}\bigr|=
\sup_{k\in\Cal{K},x\in
\Cal{X}}\sqrt{k(x,x)}<\infty,
\]
which means $\Cal{F}_H$ is uniformly bounded. Now, for a given $x\in\Cal
{X}$ and
$\epsilon>0$, pick some $y\in U_{x,\epsilon}$ such that $\Vert
k(\cdot,x)-k(\cdot,y)\Vert_{\Cal{H}_k}<\epsilon$ for all $k\in\Cal{K}$. This
means, for any $f\in\Cal{F}_H$,
\[
\bigl|f(x)-f(y)\bigr|\le\Vert f\Vert_{\Cal{H}_k}\bigl\Vert k(\cdot,x)-k(\cdot,y)
\bigr\Vert_{\Cal{H}_k}< \epsilon,
\]
which
implies $\Cal{F}_H$ is equicontinuous on $\Cal{X}$. The result
therefore follows
from Dudley \cite{Dudley-02}, Corollary~11.3.4, which shows that if
$\bb{P}_{(n)}\leadsto\bb{P}$ then $\bb{P}_{(n)}$ converges to
$\bb{P}$ in $\Vert\cdot\Vert_{\Cal{F}_H}$ as $n\rightarrow\infty$.
\end{pf}

Comparing (\ref{Eq:charac}) and (\ref{Eq:dense}), it is clear that one requires
a stronger condition for weak convergence than for $\Vert\cdot\Vert
_{\Cal{F}_H}$
being just a metric.
However, these conditions can be shown to be equivalent for bounded continuous
translation invariant kernels on $\bb{R}^d$, that is, kernels of the
type in
(\ref{Eq:Bochner}). This is because if $k$ satisfies (\ref
{Eq:Bochner}), then
\[
\int\int k(x,y) \,\mathrm{d}\mu(x) \,\mathrm{d}\mu(y)=\Vert \widehat{\mu}\Vert^2_{L^2(\bb{R}^d,\Upsilon)}
\]
and, therefore, (\ref{Eq:dense})
holds
if and only if $\operatorname{supp}(\Upsilon)=\bb{R}^d$, which is indeed the
characterization for $k$ being characteristic. Here, $\widehat{\mu}$ represents
the
Fourier transform of $\mu$ defined as $\widehat{\mu}(\omega)=\int
\mathrm{e}^{\sqrt{-1}\omega^Tx} \,\mathrm{d}\mu(x), \omega\in\bb{R}^d$. Therefore, for $\Cal
{K}$ in
Example~\ref{Exm:charac},
convergence in $\Vert\cdot\Vert_{\Cal{F}_H}$ implies weak convergence on
$M^1_+(\bb{R}^d)$.
However, for the converse to hold, $\Cal{K}$ has to satisfy $(\mathrm{P})$ in
Theorem~\ref{Thm:weak}, which is captured in the following example.

\begin{exm}\label{Exm:weak}
The following families of kernels satisfy the conditions (\ref
{Eq:dense}), $(\mathrm{P})$
and the uniform boundedness condition of Theorem~\ref{Thm:weak} so that
$\Vert\cdot\Vert_{\Cal{F}_H}$ metrizes the weak topology on $M^1_+(\bb{R}^d)$.
\begin{enumerate}
\item
$\Cal{K}= \{\mathrm{e}^{-\sigma\Vert
x-y\Vert^2_2}, x,y\in\bb{R}^d \dvt \sigma\in(0,a],
a<\infty \}$;
\item
$\Cal{K}= \{\mathrm{e}^{-\sigma\Vert
x-y\Vert_1}, x,y\in\bb{R}^d \dvt \sigma\in(0,a],
a<\infty \}$;
\item
$\Cal{K}= \{\frac{2 ({c}/{2} )^{\beta-
{d}/{2}}}{\Gamma(\beta-{d}/{2})} \Vert
x-y\Vert^{\beta-{d}/{2}}_2\frak{K}_{{d}/{2}-\beta} (c\Vert
x-y\Vert_2 ), x,y\in\bb{R}^d, \beta>\frac{d}{2} \dvt c\in
(0,a], a<\infty \}$;
\item
$\Cal{K}= \{ (1+\llVert \frac{x-y}{c}\rrVert ^2_2
)^{-\beta},
x,y\in\bb{R}^d, \beta>0 \dvt c\in[a,\infty), a>0  \}$;
\item
$\Cal{K}= \{\prod^d_{i=1} (1+\frac{|x_i-y_i|^2}{c^2_i}
 )^{-1}, x,y\in\bb{R}^d \dvt c_i\in
[a_i,\infty), a_i>0, \forall i=1,\ldots,d \}$;
%
\item$\Cal{K}$ in Theorem~\ref{thmm:examples}(b) and (c).
\end{enumerate}
\end{exm}

\textit{Rate of convergence of} $\bb{P}_n$ \textit{to} $\bb{P}$ in $\Vert
\cdot\Vert_{\Cal{F}_H}$: The following result, which is
proved in Section~\ref{subsec:consistency}, presents an
exponential inequality for the tail
probability of $\Vert\bb{P}_n-\bb{P}\Vert_{\Cal{F}_H}$, which in combination
with Theorem~\ref{Thm:weak} provides a convergence rate for the weak convergence
of $\bb{P}_n$ to $\bb{P}$.

\begin{thmm}\label{Thm:consistency}
Let $X_1,\ldots,X_n$ be random samples drawn i.i.d. from $\bb{P}$
defined on
a measurable space $\Cal{X}$. Assume there exists $\nu>0$ such
that $\sup_{k\in\Cal{K},x\in\Cal{X}}k(x,x)\le\nu$. Then for every $\tau
>0$, with
probability at least $1-2\mathrm{e}^{-\tau}$ over the choice of
$(X_i)^n_{i=1}\sim\bb{P}^n$, 
%
\begin{equation}
\Vert\bb{P}_n-\bb{P}\Vert_{\Cal{F}_H}\le 4\sqrt{2}\sqrt{
\inf_{\alpha>0} \biggl\{\alpha+\frac{4\mathrm{e}}{
n}\int
^{2\nu}_{\alpha} \log2 \Cal{ N } (\Cal{ K }, \rho ,
\epsilon) \,\mathrm{d}\epsilon \biggr\}}+\frac{3\sqrt{2\nu}(\sqrt{
2}+\sqrt{ \tau
} ) } { \sqrt{ n } }, \label{Eq:empirical-consistent}
\end{equation}
where for any $k_1,k_2\in\Cal{K}$,
%
\begin{equation}
\rho(k_1,k_2):=\sqrt{\frac{2}{n^2}\sum
^n_{i<j} \bigl(k_1(X_i,
X_j)-k_2(X_i,X_j)
\bigr)^2 }.\label{Eq:rho}
\end{equation}
In particular, if
there exists finite positive
constants $A$
and $\beta$ (that are not dependent on $n$) such that
%
\begin{equation}
\log\Cal{N}(\Cal{K},\rho,\epsilon)\le A \biggl(\frac{2\nu}{\epsilon}
\biggr)^\beta, \qquad 0<\epsilon<2\nu \label{Eq:entropynumber}
\end{equation}
then there exists constants $(D_i)^{5}_{i=1}$ (dependent
only on $A$, $\beta$, $\nu$, $\tau$ and not on $n$) such that
%
\begin{equation}
\bb{P}^n \bigl( \bigl\{(X_1,\ldots,X_n)\in
\Cal{X}^n\dvt \Vert \bb{P}_n-\bb{P}\Vert_{\Cal{F}_H}>
\lambda(A,\beta,\nu, \tau) \bigr\} \bigr)\le 2\mathrm{e}^{-\tau},\label{Eq:concentration-empirical}
\end{equation}
where
%
\begin{equation}
\label{Eq:rates-empirical}\lambda(A,\beta,\nu,\tau)\le \cases{ %
\displaystyle\frac{D_1}{
 \sqrt{ n } }, &\quad $0<\beta<1,$
\vspace*{2pt}\cr
\displaystyle D_2\sqrt{\frac{\log
n}{n}}+\frac{D_3}{\sqrt{n}},&\quad $\beta=1,$
\vspace*{2pt}\cr
\displaystyle \frac{D_4}{n^{
1/2\beta} } +\frac{
D_5} {
\sqrt{ n } }, &\quad $\beta>1$.}
\end{equation}
%
\end{thmm}

\begin{rem}\label{Rem:supp}
(i) If $\Cal{K}$ is a VC-subgraph, by van der Vaart and Wellner
(\cite{Vaart-96}, Theorem~2.6.7),
there exists finite constants
$B$
and $\alpha$ (that are not dependent on $n$) such that
$\Cal{N}(\Cal{K},\rho,\epsilon)\le
B (2\nu/\epsilon )^\alpha, 0<\epsilon<2\nu$, which implies there
exists $A$ and $0<\beta<1$ such that (\ref{Eq:entropynumber}) holds.
By Theorem~\ref{Thm:consistency}, this implies $\Vert
\bb{P}_n-\bb{P}\Vert_{\Cal{F}_H}=\mathrm{O}_\bb{P}(n^{-1/2})$, and hence by
the Borell--Cantelli lemma, $\Vert
\bb{P}_n-\bb{P}\Vert_{\Cal{F}_H}\stackrel{\mathrm{a.s.}}{\rightarrow}0$ as
$n\rightarrow
\infty$. Therefore, it is clear that if $\Cal{K}$ is a uniformly
bounded VC-subgraph, then
{\renewcommand{\theequation}{Q}\label{eqQ}
\begin{equation}
\bb{P}_n\leadsto\bb{P} \qquad\mbox{a.s. as } n\rightarrow\infty
\end{equation}}
\hspace*{-2pt}with a rate of convergence of
$n^{-1/2}$. Ying and Campbell (\cite{Ying-10}, Lemma~2)~-- also see
Proposition~\ref{pro:vc}~-- showed
that the Gaussian kernel family in Example~\ref{Exm:charac} is a
VC-subgraph (in fact, using the proof idea in Lemma~2 of Ying and
Campbell \cite{Ying-10} it can
be easily shown that Laplacian and inverse multiquadric
families are also VC-subgraphs) and, therefore, these kernel
classes ensure
\eqref{eqQ} with a convergence rate of $n^{-1/2}$. Instead of directly
showing the
radial basis function (RBF) class in Example~\ref{Exm:charac} to be a
VC-subgraph, Ying and Campbell (\cite{Ying-10}, see the proof of
Corollary~1), bounded
the expected suprema of the
Rademacher chaos process of degree 2, that is,
%
\setcounter{equation}{15}
\begin{equation}
U_n\bigl(\Cal{K};(X_i)^n_{i=1}
\bigr):=\bb{E}_\varepsilon\sup_{k\in
\Cal{K}} \Biggl\llvert \sum
^n_
{i<j}\varepsilon_i
\varepsilon_j k(X_i,X_j)\Biggr\rrvert
\label{Eq:chaos}
\end{equation}
indexed by the RBF class,
$\Cal{K}$,
by that of the Gaussian class (see (\ref{Eq:exm2-U}) in
the proof of Theorem~\ref{thmm:examples}(a)) and since the Gaussian
class is
a VC-subgraph, we obtain
$U_n(\Cal{K};(X_i)^n_{i=1})=\mathrm{O}_\bb{P}(n)$ for the RBF class. We also
show in
Theorem~\ref{thmm:examples}(d) that $U_n(\Cal{K};(X_i)^n_{i=1})=\mathrm{O}_\bb
{P}(n)$ for
the Mat\'{e}rn kernel family in Example~\ref{Exm:charac}. Using these
bounds in
(\ref{Eq:Mc-4}) and following through the proof of Theorem~\ref{Thm:consistency}
yields that \eqref{eqQ} holds with a convergence rate of $n^{-1/2}$. Note
that this
rate
of convergence is faster than the rate of $n^{-1/d}, d\ge3$ that is
obtained with $\Vert\cdot\Vert_{\Cal{F}_{\mathrm{BL}}}$
(Sriperumbudur \textit{et al.}
\cite{Sriperumbudur-12}, Corollary~3.5). Here, $(\varepsilon_i)^n_{i=1}$ denote
i.i.d. Rademacher random variables.

(ii) We would like to mention that Theorem~\ref{Thm:consistency}
is a
variation
on Theorem~7 in Sriperumbudur
\textit{et al.} \cite{Sriperumbudur-09c} where $U_n(\Cal{K};(X_i)^n_{i=1})$ is
bounded by the
entropy integral in de la Pe{\~{n}}a and
Gin{\'{e}} (\cite{delaPena-99}, Corollary~5.1.8), with the lower limit of
the
integral being zero unlike in Theorem~\ref{Thm:consistency}. This
generalization (see Mendelson \cite{Mendelson-02}, Srebro, Sridharan and
Tewari \cite{Srebro-10} for a similar result to bound
the expected
suprema of empirical processes) allows to handle the polynomial growth of
entropy number for $\beta\ge1$ compared to Sriperumbudur
\textit{et al.} (\cite{Sriperumbudur-09c}, Theorem~7).
Also, compared to Sriperumbudur
\textit{et al.} (\cite{Sriperumbudur-09c}, Theorem~7), we provide explicit
constants in Theorem~\ref{Thm:consistency}.%

\end{rem}

\section{Main results}\label{Sec:mainresult}
In this section, we present our main results of demonstrating the
optimality of
the kernel density estimator in $\Vert\cdot\Vert_{\Cal{F}_H}$ through an
exponential concentration inequality in Section~\ref{subsec:expinequality} and a
uniform central limit theorem in Section~\ref{subsec:uclt-main}.

\subsection{An exponential concentration inequality for
\texorpdfstring{$\Vert\bb{P}_n\ast K_h-\bb{P}\Vert_{\Cal{F}_H}$}{$||P_n\ast K_h-P||_{F_H}$}}\label{subsec:expinequality}

In this section, we present an exponential inequality for
the weak convergence of kernel density estimator on $\bb{R}^d$ using which
we show
the optimality of the kernel density estimator in both strong and weak
topologies. This is carried out using the
ideas in Section~\ref{Sec:rkhs}, in particular through bounding the tail
probability of $\Vert\bb{P}_n\ast K_h-\bb{P}\Vert_{\Cal{F}_H}$, where
$\bb{P}_n\ast K_h$ is the kernel density estimator. Since $\Vert
\bb{P}_n\ast K_h-\bb{P}\Vert_{\Cal{F}_H}\le
\Vert\bb{P}_n\ast K_h-\bb{P}_n\Vert_{\Cal{F}_H}+\Vert
\bb{P}_n-\bb{P}\Vert_{\Cal{F}_H}$, the result follows from
Theorem~\ref{Thm:consistency} and bounding the tail probability of
$\Vert
\bb{P}_n\ast
K_h-\bb{P}_n\Vert_{\Cal{F}_H}$, again through an application of McDiarmid's
inequality, which is captured in Theorem~\ref{Thm:estim-to-p}.

\begin{thmm}\label{Thm:estim-to-p}
Let $\bb{P}$ have a density $p\in W^s_1(\Cal{X})$, $s\in\bb{N}$ with
$(X_i)^n_{i=1}$ being samples drawn
i.i.d. from $\bb{P}$ defined on an open subset $\Cal{X}$ of $\bb{R}^d$.
Assume $\Cal{K}$
satisfies the following:
\begin{longlist}[(iii)]
\item[(i)] Every $k\in\Cal{K}$ is translation invariant, that is,
$k(x,y)=\psi(x-y), x,y\in\Cal{X}$, where $\psi$ is a positive definite
function on $\Cal{X}$;
\item[(ii)] For every
$k\in\Cal{K}$, $\partial^{\alpha,\alpha}k\dvtx\Cal{X}\times\Cal{
X}\rightarrow\bb{R}$ exists and is
continuous for all multi-indexes $\alpha\in\bb{N}^d_0$ with $|\alpha
|\le m$,
$m\in\bb{N}$, where
$\partial^{\alpha,\alpha}:=\partial^{\alpha_1}_1\cdots\partial^{\alpha_d}
_d\partial^{\alpha_1}_{1+d}\cdots\partial^{\alpha_d}_{2d}$;
\item[(iii)] $\exists\nu>0$ such that $\sup_{k\in\Cal{K}, x\in
\Cal{X}}k(x,x)\le\nu<\infty$;
\item[(iv)] For
$|\alpha|=m\wedge r$, $\exists\nu_\alpha>0$ such
that $\sup_{k^\prime\in\Cal{K}_\alpha,x\in
\Cal{X}}k^\prime(x,x)\le\nu_\alpha<\infty$ where
$\Cal{K}_\alpha:=\{\partial^{\alpha,\alpha}k \dvt k\in\Cal{K}\}$,
\end{longlist}
where $1\le r\le m+s$, $r\in\bb{N}$ is the order of the smoothing kernel
$K$. Then for every $\tau>0$, with
probability at least $1-2\mathrm{e}^{-\tau}$ over the choice of $(X_i)^n_{i=1}$,
there exists finite constants
$(A_i)^2_{i=1}$ and $(B_i)^2_{i=1}$ (dependent only on $m$, $r$, $s$, $p$,
$K$, $\tau$, 
$\nu$, $\nu_\alpha$ and not on $n$) such that
%
\begin{eqnarray}
\Vert K_h\ast\bb{P}_n-\bb{P}_n
\Vert_{\Cal{F}_H}\le4\sqrt{2}h^{m\wedge
r}\sum
_{|\alpha|=m\wedge r}\Theta(\alpha)\sqrt{\Cal{ T } (\Cal { K}
_\alpha,\rho_\alpha,\nu_\alpha)}+\frac{
A_1h^{m\wedge r}}{\sqrt{n}}+A_2h^{r}\label{Eq:final-talagrand}
\end{eqnarray}
and
%
\begin{eqnarray}
\label{Eq:final-talagrand-1} \Vert K_h\ast\bb{P}_n-\bb{P}
\Vert_{\Cal{F}_H}&\le& 4\sqrt{2}h^{m\wedge
r}\sum
_{|\alpha|=m\wedge r}\Theta(\alpha)\sqrt{\Cal{ T } (\Cal { K }
_\alpha,\rho_\alpha,\nu_\alpha)}+4\sqrt{2}\sqrt{
\Cal{T}(\Cal{K},\rho,\nu ) }
\nonumber
\\[-8pt]
\\[-8pt]
\nonumber
& &{} +\frac{B_1h^{m\wedge r}}{\sqrt{n}}+\frac{B_2}{\sqrt{n}}+A_2h^{r},
\end{eqnarray}
where $m\wedge r:=\min(m,r)$,
\begin{eqnarray*}
\Cal{T}(\Cal{K}_\alpha,\rho_\alpha,\nu_\alpha)&:=&\inf
_{\delta>0} \biggl\{\delta+\frac{4\mathrm{e}
} { n } \int
^ {
2\nu_\alpha}_\delta\log2\Cal{N}(\Cal{K}_\alpha,
\rho_\alpha,\epsilon) \,\mathrm{d}\epsilon \biggr\},\\
\Theta(\alpha)&=&\int
\frac{\prod^d_{i=1}|t_i|^{\alpha_i}}{\prod^d_{i=1}\alpha_i!}\bigl|K(t)\bigr| \,\mathrm{d}t,
\end{eqnarray*}
$\rho$ is defined as in (\ref{Eq:rho}) and for any $k_1,k_2\in\Cal
{K}_\alpha$,
\[
\rho_\alpha(k_1,k_2)=\sqrt{
\frac{2}{n^2}\sum^n_{i<j}
\bigl(k_1(X_i, X_j)-k_2(X_i,X_j)
\bigr)^2 }.
\]
In addition, suppose there exists finite constants $C_\alpha$, $C_\Cal{K}$,
$\omega_\alpha$ and $\omega_\Cal{K}$ (that are not dependent on~$n$)
such that
%
\begin{equation}
\log\Cal{N}(\Cal{K}_\alpha,\rho_\alpha,\epsilon)\le
C_\alpha \biggl(\frac{2\nu_\alpha}{\epsilon} \biggr)^{\omega_\alpha},\qquad 0<\epsilon<2
\nu_\alpha, \mbox{ for } |\alpha|=m\wedge r\label{Eq:entropy-alpha-kde}
\end{equation}
and
%
\begin{equation}
\log\Cal{N}(\Cal{K}, \rho,\epsilon)\le C_\Cal{K} \biggl(
\frac{2\nu}{\epsilon} \biggr)^{\omega_\Cal{K}},\qquad 0<\epsilon<2\nu.\label{Eq:entropy-full-kde}
\end{equation}
Define $\omega_\star:=\max\{\omega_\alpha\dvt |\alpha|=m\wedge r\}$. If
%
\begin{equation}
\sqrt{ (\log n )^{\mathds{1}_{\{\omega_\star=1\}}}} n^{{((\omega_\star\vee1)-1)}/{(2\omega_\star)}}h^{m\wedge r}\rightarrow0,\qquad
\sqrt{n}h^{r}\rightarrow 0 \mbox{ as } h\rightarrow 0, n\rightarrow
\infty,\label{Eq:condition-h-kde}
\end{equation}
then
%
\begin{equation}
\Vert\bb{P}_n\ast K_h-\bb{P}_n
\Vert_{\Cal{F}_H}=\mathrm{o}_{\mathrm{a.s.}}\bigl(n^{-1/2}\bigr)\label{Eq:small-oh-kde}
\end{equation}
and, therefore,
%
\begin{equation}
\Vert\bb{P}_n\ast K_h-\bb{P}\Vert_{\Cal{F}_H}=\mathrm{O}_{\mathrm{a.s.}}
\bigl(\sqrt{(\log n)^{\mathds{1}_{\{\omega_\Cal{K}=1\}}}}
n^{-{(\omega_\Cal{K}\wedge
1)}/{(2\omega_\Cal{K})}} \bigr).\label{Eq:rates-kde}
\end{equation}
\end{thmm}

\begin{rem}\label{Rem:main}
(i) Theorem~\ref{Thm:estim-to-p} shows that the kernel density estimator
with bandwidth, $h$ converging to zero sufficiently fast as given by conditions
in (\ref{Eq:condition-h-kde}) is within $\Vert\cdot\Vert_{\Cal
{F}_H}$-ball of
size $\mathrm{o}_{\mathrm{a.s.}}(n^{-1/2})$ around $\bb{P}_n$ and behaves like $\bb{P}_n$
in the
sense that it converges to $\bb{P}$ in $\Vert\cdot\Vert_{\Cal{F}_H}$ at a
dimension independent rate of $n^{-1/2}$ (see
Theorem~\ref{Thm:consistency}) as long as $\Cal{K}$ is not too big,
which is
captured by $\omega_\Cal{K}<1$ in (\ref{Eq:entropy-full-kde}). In
addition, if
$\Cal{K}$ satisfies the conditions in Theorem~\ref{Thm:weak}, then the kernel
density estimator converges weakly to $\bb{P}$ a.s. at the rate of $n^{-1/2}$.
Since we are interested in the optimality of $\bb{P}_n\ast K_h$ in both strong
and weak topologies, it is interest to understand
whether the asymptotic behavior in (\ref{Eq:small-oh-kde}) holds for
$h^\ast\simeq
n^{-1/(2s+d)}$ where $h^\ast$ is the optimal bandwidth (of the kernel
density estimator) for the estimation of $p$ in $L^1$ norm. It is easy
to verify
that if
%
\begin{equation}
r>s+\frac{d}{2} \quad\mbox{and}\quad m>\frac{(2s+d)(\omega_\star-1)}{2\omega_\star}\vee\frac{ d } {
2}\label{Eq:m-choice}
\end{equation}
then $h^\ast$ satisfies (\ref{Eq:condition-h-kde}) and,
therefore, $\Vert
K_{h^\ast}\ast\bb{P}_n-\bb{P}_n\Vert_{\Cal{F}_H}=\mathrm{o}_ { \mathrm{a.s.} } (n^ { -1/2
})$ so that $\Vert K_{h^\ast}
\ast\bb{P}_n-\bb{P}\Vert_{\Cal{F}_H}=\mathrm{O}_{\mathrm{a.s.}}(n^{-1/2})$ if $\Cal{K}$
is not
too big. This means for an appropriate choice of $\Cal{K}$ (i.e.,
$\omega_\Cal{K}<1$),
the kernel density estimator $K_h\ast\bb{P}_n$ with $h=h^\ast$ is
optimal in
both weak (induced by $\Vert\cdot\Vert_{\Cal{F}_H}$) and strong topologies
unlike
$\bb{P}_n$ which is only an optimal estimator of $\bb{P}$ in the weak
topology. In
Theorem~\ref{thmm:examples}, we present examples of $\Cal{K}$ for which the
kernel density estimator is optimal in both strong and weak topologies (induced
by $\Vert\cdot\Vert_{\Cal{F}_H}$). Under the conditions in
(\ref{Eq:m-choice}), it can be shown that $h^{\ast\ast}\simeq(n/\log
n)^{-1/(2s+d)}$, which is the optimal bandwidth for the estimation of
$p$ in
sup-norm, also satisfies
(\ref{Eq:condition-h-kde}) and, therefore, (\ref{Eq:small-oh-kde}) and
(\ref{Eq:rates-kde}) hold for $h=h^{\ast\ast}$.

(ii) The
condition on $r$ in (\ref{Eq:m-choice}) coincides with the one obtained for
$\{\mathds{1}_{(-\infty,t]}\dvt t\in\bb{R}\}$ in Bickel and Ritov \cite
{Bickel-03} and bounded
variation and Lipschitz classes with $d=1$ in Gin{\'{e}} and Nickl \cite{Gine-08a},
see Remarks 7 and
8. This condition shows that for the kernel density
estimator with
bandwidth $h^\ast$ to be optimal in the weak topology (assuming $\omega
_\ast\le
1$, $\omega_\Cal{K}<1$ and $\Cal{K}$ satisfying the conditions
in Theorem~\ref{Thm:weak}), the order of the kernel has to be chosen
higher by
$\frac{d}{2}$ than the usual (the usual being estimating $p$ using the
kernel density estimator in $L^1$-norm). An interesting aspect of the second
condition in (\ref{Eq:m-choice}) is that the smoothness of kernels in
$\Cal{K}$
should increase with either $d$ or the size of $\Cal{K}_\alpha$ for
$\bb{P}_n\ast K_h$ with $h=h^\ast$ or $h=h^{\ast\ast}$ to lie in
$\Vert\cdot\Vert_{\Cal{F}_H}$-ball of
size $\mathrm{o}_{\mathrm{a.s.}}(n^{-1/2})$ around $\bb{P}_n$. If $\Cal{K}_\alpha$ is large,
that is, $\omega_\star>1$, then the choice of $m$ depends on the
smoothness $s$ of
$p$ and therefore $s$ has to be known a priori to pick $k$
appropriately. Also,
since the smoothness of kernels in
$\Cal{K}$ should grow with $d$ for
(\ref{Eq:small-oh-kde}) to hold, it implies that the rate in
(\ref{Eq:rates-kde}) holds under
weaker metrics on the space of probabilities. On the other hand, it is
interesting to note that as long $\Cal{K}$ satisfies the conditions in
Theorem~\ref{Thm:weak}, each of these weaker metrics metrize the weak
topology.

(iii) If $\Cal{K}$ is singleton, then it is easy to verify that the
first terms in (\ref{Eq:final-talagrand}) and (\ref{Eq:final-talagrand-1})
are of order $h^{m\wedge r}/\sqrt{n}$~-- use the idea in
Remark~\ref{rem:pn-p-remark}(ii) for (\ref{Eq:tempo-proof})~-- and
the second
term in (\ref{Eq:final-talagrand-1}) is of order $n^{-1/2}$ (see
(\ref{Eq:Mc-5})). Therefore, the claims of Theorem~\ref{Thm:estim-to-p}
hold as
if $\omega_\star\le1$ and $\omega_\Cal{K}\le1$.
\end{rem}

\begin{pf}
Note that
%
\begin{equation}
\Vert K_h\ast\bb{P}_n-\bb{P}_n
\Vert_{\Cal{F}_H}\le \bigl\Vert K_h\ast(\bb{P}_n-
\bb{P})-(\bb{P}_n-\bb{P})\bigr\Vert_{\Cal{F}_H}+\Vert K_h
\ast\bb{P}- \bb{ P }\Vert_{\Cal{F}_H}.\label{Eq:split}
\end{equation}

(a) \textit{Bounding} $\Vert
K_h\ast(\bb{P}_n-\bb{P})-(\bb{P}_n-\bb{P})\Vert_{\Cal{F}_H}$:

By defining $\Cal{B}_0:=\Vert
K_h\ast(\bb{P}_n-\bb{P})-(\bb{P}_n-\bb{P})\Vert_{\Cal{F}_H}$, we have
\begin{eqnarray*}
\Cal{B}_0&=&\sup_{f\in
\Cal{F}_H} \biggl\llvert \int
f(x) \,\mathrm{d}\bigl(K_h\ast(\bb{P}_n-\bb{P})\bigr) (x)-\int
f(x) \,\mathrm{d}(\bb{P}_n-\bb{P}) (x)\biggr\rrvert
\\
&=&\sup_{f\in\Cal{F}_H}\biggl\llvert \int(f\ast
K_h-f) \,\mathrm{d}(\bb{P}_n-\bb{P})\biggr\rrvert
=\Vert\bb{P}_n-\bb{P}\Vert_\Cal{G},
\end{eqnarray*}
where $\Cal{G}:=\{f\ast K_h-f\dvt f\in\Cal{F}_H\}$. We now
obtain a bound on $\Vert\bb{P}_n-\bb{P}\Vert_\Cal{G}$ through an application
of McDiarmid's inequality. To this end, consider
%
\begin{equation}
\Vert g\Vert_\infty\le\Vert f\ast K_h-f
\Vert_\infty=\sup_{x\in
\Cal{X}}\biggl\llvert \int
\bigl(f(x+ht)-f(x) \bigr)K(t) \,\mathrm{d}t\biggr\rrvert .\label{Eq:g}
\end{equation}
Since every $k\in\Cal{K}$ is $m$-times differentiable, by Steinwart and
Christmann \cite{Steinwart-08}, Corollary~4.36, every $f\in\Cal{F}_H$ is $m$-times continuously
differentiable and for any $k\in\Cal{K}$, $f\in\Cal{H}_k$,
%
\begin{equation}
\bigl\llvert \partial^\alpha f(x)\bigr\rrvert \le \Vert f
\Vert_{\Cal{H}_k}\sqrt{\partial^{\alpha,\alpha}k(x,x)},\qquad x\in\Cal{X}
\label{Eq:diff-inequality}
\end{equation}
for $\alpha\in\bb{N}^d_0$ with $|\alpha|\le m$. Therefore, Taylor series
expansion of
$f(x+th)$ around $x$ gives
%
\begin{eqnarray}
\label{Eq:taylor}f(x+th)-f(x)&=&\sum_{0<|\alpha|\le
(m\wedge r)-1}h^{|\alpha|}
\Lambda_\alpha(t) \partial^{\alpha}f(x)
\nonumber
\\[-8pt]
\\[-8pt]
\nonumber
& &\hspace*{55pt}{} +h^{m\wedge r}\sum_{|\alpha|=m\wedge r}
\Lambda_\alpha(t)\partial^\alpha f(x+hD_\theta t),
\end{eqnarray}
where
\[
\Lambda_\alpha(t):=\frac{\prod^d_{i=1}t^{\alpha_i}_i}{\prod^d_{i=1}\alpha_i!},
\]
$D_\theta=\operatorname{diag}(\theta_1,\ldots,\theta_d)$ and $0<\theta_i<1$ for
all $i=1,\ldots,d$. Using (\ref{Eq:taylor}) in (\ref{Eq:g}) along with the
regularity of $K$, we have
%
\begin{eqnarray}
\label{Eq:g1} \Vert g\Vert_\infty&\le& \Vert f\ast
K_h-f\Vert_\infty\le h^{m\wedge r}\sup
_{x\in
\Cal{X}}\biggl\llvert \sum_{|\alpha|=m\wedge r}\int
\Lambda_\alpha(t) K(t) \partial^\alpha f(x+hD_\theta
t) \,\mathrm{d}t\biggr\rrvert
\nonumber
\\[-8pt]
\\[-8pt]
\nonumber
&\le&h^{m\wedge r}\sup_{x\in\Cal{X}}\sum
_{|\alpha|=m\wedge r}\int \Lambda_\alpha\bigl(|t|\bigr) \bigl|K(t)\bigr| \bigl\llvert
\partial^\alpha f(x+hD_\theta t)\bigr\rrvert \,\mathrm{d}t,
\end{eqnarray}
where $|t|_i:=|t_i|, \forall i=1,\ldots,d$. Using
(\ref{Eq:diff-inequality}) in (\ref{Eq:g1}), for any $g\in\Cal{G}$, we get
%
\begin{eqnarray}
\label{Eq:uniformbound} \Vert g\Vert_\infty&\le& \Vert f\ast
K_h-f\Vert_\infty
\nonumber
\\
&\le& h^{m\wedge r}\sup_{k\in\Cal{K},x\in
\Cal{X}}\sum
_{|\alpha|=m\wedge r}\int \Lambda_\alpha\bigl(|t|\bigr) \bigl|K(t)\bigr| \sqrt{
\partial^{\alpha,\alpha}k(x+hD_\theta t,x+hD_\theta t)} \,\mathrm{d}t
\nonumber
\\[-8pt]
\\[-8pt]
\nonumber
&\stackrel{\mathrm{(i)}} {\le} &h^{m\wedge
r}\sum_{|\alpha|=m\wedge r}
\sqrt{\sup_{k\in\Cal{K}, x\in
\Cal{X}}\partial^{\alpha,\alpha}k(x, x) } \int
\Lambda_\alpha\bigl(|t|\bigr) \bigl|K(t)\bigr| \,\mathrm{d}t
\\
&=& L_{m,r}h^{m\wedge r},\nonumber
\end{eqnarray}
where
\[
L_{m,r}:=\sum_{|\alpha|=m\wedge r}\sqrt{
\nu_\alpha} \Theta(\alpha)<\infty.
\]
%
Now, let us consider
%
\begin{eqnarray}
\label{Eq:symm} \bb{E}\Vert \bb{P}_n-\bb{P}\Vert_\Cal{G}
\stackrel{(\star)} {\le}\frac{2}{n} \bb {E}\sup_{
g\in\Cal{G}}
\Biggl\llvert \sum^n_{ i=1 }
\varepsilon_i g(X_i)\Biggr\rrvert =\frac{2}{n}
\bb{E}\sup_{f\in\Cal{F}_H} \Biggl\llvert \sum
^n_ { i=1
} \varepsilon_i(f\ast
K_h-f) (X_i)\Biggr\rrvert ,
\end{eqnarray}
where we have invoked the symmetrization inequality (van der Vaart and
Wellner \cite{Vaart-96}, Lemma~2.3.1) in $(\star)$ with
$(\varepsilon_i)^n_{i=1}$ being the Rademacher random variables. By McDiarmid's
inequality, for any $\tau>0$, with probability at least
$1-\mathrm{e}^{-\tau}$,
%
\begin{eqnarray}
\label{eq:mcdiarmid} \Vert\bb{P}_n-\bb{P}\Vert_\Cal{G}&\le&
\bb{E}\Vert \bb{P}_n-\bb{P}\Vert_\Cal{G}+\Vert g
\Vert_\infty\sqrt{\frac{2\tau}{n}}
\nonumber
\\[-8pt]
\\[-8pt]
\nonumber
&\le&\frac{2}{n} \bb{E}\sup_{f\in\Cal{F}_H} \Biggl\llvert
\sum^n_ { i=1
} \varepsilon_i(f
\ast K_h-f) (X_i)\Biggr\rrvert +L_{m,r}h^{m\wedge
r}
\sqrt{\frac{2\tau}{n}},
\end{eqnarray}
where (\ref{Eq:uniformbound}) and (\ref{Eq:symm}) are used in
(\ref{eq:mcdiarmid}). Define
\[
R_n(\Cal{F}_H):=\bb{E}_\varepsilon\sup
_{f\in\Cal{F}_H} \Biggl\llvert \frac{1}{n}\sum
^n_{i=1
} \varepsilon_i(f\ast
K_h-f) (X_i)\Biggr\rrvert ,
\]
where $\bb{E}_\varepsilon$ denotes
the
expectation
w.r.t. $(\varepsilon_i)^n_{i=1}$ conditioned on $(X_i)^n_{i=1}$. Applying
McDiarmid's inequality to $R_n(\Cal{F}_H)$, we have for any $\tau>0$, with
probability at least
$1-\mathrm{e}^{-\tau}$,
%
\begin{eqnarray}\label{Eq:empiricalrad}
&&\bb{E}\sup_{f\in\Cal{F}_H}\Biggl\llvert \frac{1}{n}\sum
^n_{i=1
} \varepsilon_i(f\ast
K_h-f) (X_i)\Biggr\rrvert
\nonumber
\\[-8pt]
\\[-8pt]
\nonumber
&&\quad \le R_n(
\Cal{F}_H)+L_{m,r}h^{m\wedge
r}\sqrt{
\frac{2\tau}{n}}.
\end{eqnarray}
%
Bounding
$R_n(\Cal{F}_H)$ yields
\begin{eqnarray*}
&&R_n(\Cal{F}_H)\\
&&\quad=\bb{E}_\varepsilon\sup
_{f\in\Cal{F}_H} \Biggl\llvert \frac{1}{n}\sum
^n_{i=1} \varepsilon_i\int
\bigl(f(X_i+th)-f(X_i)\bigr)K(t) \,\mathrm{d}t\Biggr\rrvert
\\
&&\quad\stackrel{(\ref{Eq:taylor})} {=} \frac{h^{m\wedge r}}{n}\bb{E} _\varepsilon
\sup_ {
f\in\Cal{ F
}_H} \Biggl\llvert \sum_{|\alpha|=m\wedge r}
\int \Lambda_\alpha(t) K(t)\sum^n_{j=1}
\varepsilon_j\partial^\alpha f(X_j+hD_\theta
t) \,\mathrm{d}t\Biggr\rrvert
\\
&&\quad=\frac{h^{m\wedge r}}{n}\bb{E}_\varepsilon\sup_{f\in\Cal{F}_H}
\Biggl\llvert \sum_{ |\alpha|=m\wedge r } \int \Lambda_\alpha(t)
K(t) \Biggl\langle f,\sum^n_{j=1}
\varepsilon_j\partial^\alpha k(\cdot,X_j+hD_\theta
t) \Biggr\rangle_{\Cal{H}_k} \,\mathrm{d}t\Biggr\rrvert
\\
&&\quad\le \frac{h^{m\wedge
r}}{n}\bb{E}_\varepsilon\sup
_{k\in\Cal{K}}\sum_{|\alpha|=m\wedge r}\int
\Lambda_\alpha\bigl(|t|\bigr) \bigl|K(t)\bigr|\Biggl\llVert \sum
^n_{j=1}\varepsilon_j
\partial^\alpha k(\cdot, X_j+hD_\theta t)\Biggr
\rrVert _{\Cal{H}_k} \,\mathrm{d}t
\\
&&\quad= \frac{h^{m\wedge
r}}{n}\bb{E}_\varepsilon\sup_{k\in\Cal{K}}
\sum_{|\alpha|=m\wedge r}\int \Lambda_\alpha\bigl(|t|\bigr) \bigl|K(t)\bigr|
\sqrt{\sum^n_{i,j=1}
\varepsilon_i\varepsilon_j\partial^{\alpha
,\alpha}
k(X_i+hD_\theta t,X_j+hD_\theta t)}
\,\mathrm{d}t.
\end{eqnarray*}
Since $k$ is translation invariant, we have
%
\begin{eqnarray}
R_n(\Cal{F}_H)&=& \frac{h^{m\wedge
r}}{n}
\bb{E}_\varepsilon\sup_{k\in\Cal{K}}\sum
_{|\alpha|=m\wedge r} \Theta(\alpha)\sqrt{ \sum
^n_ { i, j=1 } \varepsilon_i
\varepsilon_j\partial^{\alpha,\alpha} k(X_i,X_j)}\label{Eq:tempo-proof}
\\
&\le& \frac{\sqrt{2}h^{m\wedge r}}{n}\sum_{|\alpha|=m\wedge r}\Theta(\alpha
)\sqrt{ \bb{E}_\varepsilon\sup_{k\in\Cal{K}}\Biggl\llvert
\sum^n_ {i<j} \varepsilon_i
\varepsilon_j\partial^{\alpha,\alpha} k(X_i,X_j)
\Biggr\rrvert }+\frac{h^{m\wedge r}L_{m,r}}{\sqrt{n}}
\nonumber
\\
&=& \frac{\sqrt{2}h^{m\wedge r}}{n}\sum_{|\alpha|=m\wedge r}\Theta(\alpha )
\sqrt{ \bb{E}_\varepsilon\sup_{k^\prime\in\Cal{K}_\alpha}\Biggl\llvert
\sum^n_ {i<j} \varepsilon_i
\varepsilon_jk^\prime(X_i,X_j)\Biggr
\rrvert }+\frac{h^{m\wedge
r}L_{m,r}}{\sqrt{n}} \label{Eq:bound-gauss-exm}
\\
&\stackrel{(\dagger)} {\le} & 2\sqrt{2}h^{m\wedge
r}
\sum_{|\alpha|=m\wedge
r}\Theta(\alpha)\sqrt{\Cal{T}(
\Cal{K}_\alpha,\rho_\alpha,\nu_\alpha)}
+
\frac{3L_{m,r}h^{m\wedge
r}}{\sqrt{n}},\label{Eq:expectation}
\end{eqnarray}
where we used Lemma~\ref{lem:chaining} in
$(\dagger)$ with $\theta=\frac{3}{4}$. Combining (\ref{eq:mcdiarmid}),
(\ref{Eq:empiricalrad}) and (\ref{Eq:expectation}), we have that for any
$\tau>0$, with probability at least $1-2\mathrm{e}^{-\tau}$, 
%
\begin{eqnarray}
\Cal{B}_0&\le&4\sqrt{2}h^{
m\wedge
r}
\sum_{|\alpha|=m\wedge
r}\Theta(\alpha)\sqrt{\Cal{T}(
\Cal{K}_\alpha,\rho_\alpha,\nu_\alpha )}+
\frac{
A_1h^ { m\wedge
r}}{\sqrt{n}}, 
\label{Eq:diff-talagrand}
\end{eqnarray}
where
$A_1:=(6+\sqrt{18\tau})L_{m,r}$.

(b) \textit{Bounding}
$\Vert K_h\ast\bb{P}-\bb{P}\Vert_{\Cal{F}_H}$:

Defining $\Cal{B}_1:=\Vert\bb{P}\ast
K_h-\bb{P}\Vert_{\Cal{F}_H}$, we have
%
\begin{eqnarray}
\label{Eq:approx} 
\Cal{B}_1&=&\sup
_{f\in\Cal{F}_H}\biggl\llvert \int f(x) \,\mathrm{d}(\bb{P}\ast K_h)
(x)-\int f(x) \,\mathrm{d}\bb{P}(x)\biggr\rrvert
\nonumber
\\
&=&\sup_{f\in\Cal{F}_H}\biggl\llvert \frac{1}{h^d}\int\int
f(x)K \biggl(\frac{x-y}{h} \biggr) \,\mathrm{d}\bb{P}(y) \,\mathrm{d}x-\int f(x) \,\mathrm{d}\bb{P}(x)\biggr
\rrvert
\nonumber
\\[-8pt]
\\[-8pt]
\nonumber
&=&\sup_{f\in\Cal{F}_H}\biggl\llvert \int \biggl(\int
\bigl(f(x+th)-f(x)\bigr) \,\mathrm{d}\bb{P}(x) \biggr)K(t) \,\mathrm{d}t\biggr\rrvert
\nonumber
\\
&=&\sup_{f\in\Cal{F}_H}\biggl\llvert \int
\bigl((\tilde{f}\ast p) (ht)-(\tilde{f}\ast p) (0) \bigr)K(t) \,\mathrm{d}t\biggr\rrvert,\nonumber
\end{eqnarray}
where $\tilde{f}(x)=f(-x)$. Since $p\in L^1(\bb{R}^d)$ and $\partial
^\alpha f$
is bounded for all $|\alpha|\le m$, by
Folland \cite{Folland-99}, Proposition~8.10, we have $\partial^\alpha
(f\ast
p)=(\partial^\alpha f)\ast p$ for $|\alpha|\le m$. In addition, since
$\partial^\alpha f$ is continuous for all $|\alpha|\le m$ and
$\partial^\beta p\in L^1(\bb{R}^d)$ for $|\beta|\le s$, by extension
of Gin{\'{e}} and Nickl \cite{Gine-08a}, Lemma~5(b), to $\bb{R}^d$, we have
$\partial^{\alpha+\beta}(f\ast
p)=\partial^\beta((\partial^\alpha
f)\ast p)=(\partial^\alpha f)\ast(\partial^\beta p)$ for $|\alpha|\le
m, |\beta|\le s$, which means for all $f\in\Cal{F}_H$, $f\ast p$ is
$m+s$-differentiable. Therefore, using the Taylor series expansion of
$(\tilde{f}\ast p)(ht)$ around zero (as in (\ref{Eq:taylor})) along
with the
regularity of $K$ in (\ref{Eq:approx}), we have
%
\begin{eqnarray}\label{Eq:temp}
\Cal{B}_1&=&\sup_{f\in\Cal{F}_H}
\biggl\llvert h^{r}\sum_{|\alpha|+|\beta
|=r} \int
\Lambda_{\alpha+\beta}(t) K(t) \bigl(\partial^\alpha\tilde{ f } \ast
\partial^\beta p\bigr) (hD_\theta t) \,\mathrm{d}t\biggr\rrvert
\nonumber
\\
&\le&h^{r}\sup_{f\in\Cal{F}_H}\sum
_{
{|\alpha|+|\beta|=r} }\int \Lambda_{\alpha+\beta}\bigl(|t|\bigr) \bigl|K(t)\bigr|\bigl|\bigl(
\partial^\alpha\tilde{ f } \ast \partial^\beta p\bigr)
(hD_\theta t)\bigr| \,\mathrm{d}t
\nonumber\\
&\le&h^{r}\sup_{f\in\Cal{F}_H}\sum
_{
{|\alpha|+|\beta|=r}}\int \Lambda_{\alpha+\beta}\bigl(|t|\bigr)
 \bigl|K(t)\bigr|\bigl(\bigl|
\partial^\alpha\tilde{f}\bigr|\ast\bigl |\partial^\beta p\bigr|\bigr)
(hD_\theta t) \,\mathrm{d}t.
\end{eqnarray}
Since
%
\begin{eqnarray}
\label{Eq:temp-1} \bigl(\bigl|\partial^\alpha\tilde{f}\bigr|\ast\bigl |
\partial^\beta p\bigr|\bigr) (hD_\theta t)&=&\int\bigl|
\partial^\alpha f (x-hD_\theta t)\bigr| \bigl|\partial^\beta
p(x)\bigr| \,\mathrm{d}x
\nonumber
\\
&\stackrel{(\ref{Eq:diff-inequality})} {\le}  & \int\sqrt{\partial^{\alpha,\alpha}
k(x-hD_\theta t,x-hD_\theta t)} \bigl|\partial^\beta p(x)\bigr|
\,\mathrm{d}x
\nonumber\\
&\stackrel{\mathrm{(i)}} {\le}  & \bigl\Vert \partial^\beta p
\bigr\Vert_{L^1(\bb{R}^d)}\sqrt{\sup_{k\in\Cal{K},x\in
\Cal{X}}\partial^{\alpha,\alpha}
k(x,x)},
\end{eqnarray}
using (\ref{Eq:temp-1}) in (\ref{Eq:temp}), we obtain
%
\begin{equation}
\Cal{B}_1\le A_2
h^{r},\label{Eq:approx-final} 
\end{equation}
where
$A_2:=\sum_{|\alpha|+|\beta|=r}\Theta(\alpha+\beta)\sqrt{\nu_{\alpha} }
\Vert
\partial^\beta
p\Vert_{L^1(\bb{R}^d)}$
and $(\alpha+\beta)_i=\alpha_i+\beta_i, \forall i=1,\ldots, d$. Using
(\ref{Eq:diff-talagrand}) and (\ref{Eq:approx-final}) in (\ref
{Eq:split}), we
obtain the result in (\ref{Eq:final-talagrand}). Since
%
\begin{equation}
\Vert K_h\ast\bb{P}_n-\bb{P}\Vert_{\Cal{F}_H}\le
\Vert K_h\ast \bb{P}_n-\bb{P}_n
\Vert_{\Cal{F}_H}+\Vert \bb{P}_n-\bb{P}\Vert_{\Cal{F}_H},\label{Eq:triangle-1}
\end{equation}
the result in (\ref{Eq:final-talagrand-1}) follows from
Theorem~\ref{Thm:consistency} and (\ref{Eq:final-talagrand}). Under the entropy
number conditions in~(\ref{Eq:entropy-alpha-kde}), it is easy to check (see
(\ref{Eq:rates-empirical})) that
\[
\sum_{|\alpha|=m\wedge r}\Theta(\alpha)\sqrt{\Cal{T}(
\Cal{K}_{\alpha
},\rho_{
\alpha}, \nu_ { \alpha})}=\mathrm{O} \bigl(
\sqrt{(\log n)^{\mathds{1}_{\{\omega_\star=1\}}}} n^{-{(\omega_\star\wedge1)}/{(2\omega_\star)}} \bigr)
\]
and, therefore, (\ref{Eq:small-oh-kde}) holds if $h$ satisfies
(\ref{Eq:condition-h-kde}). Using (\ref{Eq:small-oh-kde}) and
(\ref{Eq:concentration-empirical}) in (\ref{Eq:triangle-1}), the result in
(\ref{Eq:rates-kde}) follows under the assumption that $\Cal{K}$ satisfies
(\ref{Eq:entropy-full-kde}).
\end{pf}

\begin{rem}
Since every $k\in\Cal{K}$ is translation invariant, an alternate proof
can be
provided by using the representation for $\fr{D}_k$ (following
(\ref{Eq:charac})) in
Proposition~\ref{pro:embedding}: $\Vert
\bb{P}-\bb{Q}\Vert_{\Cal{F}_H}=\sup_\Upsilon\Vert\phi_\bb{P}-\phi_\bb{Q}
\Vert_{L^2 (\bb{R}^d,\Upsilon)}$, where the supremum is taken over all finite
nonnegative Borel measures on $\bb{R}^d$. In this case, conditions on the
derivatives of $k\in\Cal{K}$ translate into moment requirements for
$\Upsilon$.
However, the
current proof is more transparent as it clearly shows why the translation
invariance of $k$ is needed; see (\ref{Eq:uniformbound}) and
(\ref{Eq:tempo-proof}).
\end{rem}

In the following result (proved in Section~\ref{subsec:pro-examples}),
we present some families of $\Cal{K}$ that
ensure the claims of Theorems \ref{Thm:consistency} and \ref{Thm:estim-to-p}.

\begin{thmm}\label{thmm:examples}
Suppose the assumptions on $\bb{P}$ and $K$ in Theorem~\ref{Thm:estim-to-p}
hold and let $0<a<\infty$. Then for the following classes of kernels,
\begin{longlist}
\item[(a)]
\[
\Cal{K}= \bigl\{k(x,y)=\psi_\sigma(x-y), x,y\in\bb{R}^d
\dvt \sigma\in\Sigma \bigr\},
\]
where $\psi_\sigma(x)=\mathrm{e}^{-\sigma\Vert
x\Vert^2_2}$ and $\Sigma:=(0,a]$;
\item[(b)]
\[
\Cal{K}= \biggl\{k(x,y)=\int^\infty_0
\psi_\sigma(x-y) \,\mathrm{d}\Lambda(\sigma), x,y\in\bb{R}^d \dvt
\Lambda\in\Cal{M}_A \biggr\},
\]
where
\[
\Cal{M}_A:= \biggl\{\Lambda\in M^1_+\bigl((0,\infty)
\bigr)\dvt \int^\infty_0\sigma^r \,\mathrm{d}
\Lambda(\sigma)\le A<\infty \biggr\}
\]
for some fixed $A>0$;
\item[(c)]
\[
\Cal{K}= \biggl\{k(x,y)=\int_{(0,\infty)^d} \mathrm{e}^{-(x-y)^T\Delta
(x-y)} \,\mathrm{d}\Lambda
(\Delta), x,y\in\bb{R}^d \dvt \Lambda\in\Cal{Q}_A
\biggr\},
\]
where $\Delta:=\operatorname{diag}(\sigma_1,\ldots,\sigma_d)$,
\[
\Cal{Q}_A:= \Biggl\{\Lambda\in M^1_+\bigl((0,
\infty)^d\bigr) \Big| \Lambda=\bigotimes^d_{i=1}
\Lambda_i \dvt \Lambda _i\in \Cal{M}_{A_i},
i=1,\ldots,d \Biggr\}
\]
and
\[
\Cal{M}_{A_i}:= \biggl\{\Lambda_i\in M^1_+
\bigl((0,\infty)\bigr) \dvt \sup_{j\in\{1,\ldots,d\}}\int^\infty_0
\sigma^{\alpha_j} \,\mathrm{d}\Lambda_i(\sigma)\le A_i<\infty
\biggr\}
\]
for some
fixed constant
$A:=(A_1,\ldots,A_d)\in(0,\infty)^d$ with $\sum^d_{i=1}\alpha_i=r$
and $\alpha_i\ge0, \forall i=1,\ldots, d$;
\item[(d)]
\[
\Cal{K}= \biggl\{k(x,y)=A\frac{\llVert
x-y\rrVert ^{\beta-{d}/{2}}_2}{c^{{d}/{2}-\beta}}\frak {K}_{{d}/{2
} -\beta} \bigl(c\Vert x-y
\Vert_2 \bigr), x,y\in\bb{R}^d, \beta>m+\frac{d}{2}
\dvt c\in\Sigma \biggr\},
\]
where $A:=\frac{2^{{d}/{2}+1-\beta}}{\Gamma(\beta-{d}/{2})}$
and $m\in\bb{N}$,
\end{longlist}
$\Vert
\bb{P}_n-\bb{P}\Vert_{\Cal{F}_H}=\mathrm{O}_{\mathrm{a.s.}}(n^{-1/2})$, $\Vert\bb
{P}_n\ast
K_h-\bb{P}_n\Vert_{\Cal{F}_H}=\mathrm{o}_{\mathrm{a.s.}}(n^{-1/2})$ and $\Vert
\bb{P}_n\ast
K_h-\bb{P}\Vert_{\Cal{F}_H}=\mathrm{O}_{\mathrm{a.s.}}(n^{-1/2})$ for any $h$ satisfying
$\sqrt{n}h^{r}\rightarrow0$ as $h\rightarrow0$ and $n\rightarrow\infty$,
which is particularly satisfied by $h=h^\ast$ and
$h=h^{\ast\ast}$ if $r>s+\frac{d}{2}$, where $h^\ast$ and $h^{\ast\ast
}$ are defined in
Remark~\ref{Rem:main}\textup{(i)}.
%
\end{thmm}

\begin{rem}\label{rem:gauss-exm} (i) The Gaussian RKHS in
(\ref{Eq:gauss-rkhs}) has the
property that $\Cal{H}_{\sigma}\subset\Cal{H}_\tau$ if
$0<\sigma<\tau<\infty$, where $\Cal{H}_\sigma$ is the Gaussian RKHS
induced by
$\psi_\sigma$. This follows since for any $f\in\Cal{H}_\sigma$,
%
\begin{eqnarray}
\label{Eq:subset}\Vert f\Vert^2_{\Cal{H}_\tau}&:=&(4\pi
\tau)^{{d}/{2}}\int \bigl|\widehat{f}(\omega)\bigr|^2\mathrm{e}^{{\Vert
\omega\Vert^2_2}/{(4\tau)}} \,\mathrm{d}
\omega
\nonumber
\\[-8pt]
\\[-8pt]
\nonumber
& =&(4\pi\tau)^{{d}/{2}}\int \bigl|\widehat{f}(\omega)\bigr|^2\mathrm{e}^{{\Vert
\omega\Vert^2_2}/{(4\sigma)}}\mathrm{e}^{\Vert\omega\Vert^2_2 ({1}/{(4\tau)}-
{1} /{ (4\sigma)}  )}
\,\mathrm{d}\omega\le \biggl(\frac{\tau}{\sigma} \biggr)^{{d}/{2}}\Vert f
\Vert^2_{\Cal{H}_\sigma},
\end{eqnarray}
which implies\vspace*{1pt} for any $k\in\Cal{K}$, $\Cal{H}_k\subset\Cal{H}_a$,
where the
definition of $\Vert\cdot\Vert^2_{\Cal{H}_\tau}$ for any $\tau>0$ in
the first
line of (\ref{Eq:subset}) is obtained from Wendland \cite{Wendland-05}, Theorem~10.12.
From
(\ref{Eq:subset}), it follows that
\[
\biggl\{\Vert f\Vert_{\Cal{H}_\sigma}\le \biggl(\frac{\sigma}{a}
\biggr)^{d/4} \dvt f\in\Cal{H}_\sigma, \sigma\in (0,a] \biggr\}
\subset\bigl\{f\in\Cal{H}_a\dvt \Vert f\Vert_{\Cal{H}_a}\le1\bigr\}\subset
\Cal{F}_H
\]
and
%
\begin{equation}
\Cal{F} _H\subset \bigcup_{\sigma\in(0,a]} \biggl
\{f\in\Cal{H}_a\dvt \Vert f\Vert_{\Cal
{H}_a}\le \biggl(
\frac{a}{\sigma} \biggr)^{d/4} \biggr\}=\Cal{H}_a,\label{Eq:fh-subset}
\end{equation}
where
$\Cal{F}_H:= \{\Vert
f\Vert_{\Cal{H}_\sigma}\le1 \dvt f\in\Cal{H}_\sigma, \sigma\in(0,a] \}$.

(ii) The
kernel classes in (b) and (c) above are generalizations
of
the
Gaussian family in (a). This can be seen by choosing
$\Cal{M}_A=\{\delta_{\sigma}\dvt \sigma\in\Sigma\}$ in (b) where $A=a^r$ and
$\Cal{M}_{A_i}=\{\delta_\sigma: \sigma\in\Sigma\}$ in (c) with
$A_i=a^r$ for all $i=1,\ldots, d$.
%
By choosing
\[
\Cal{M}_A= \biggl\{\mathrm{d}\Lambda(\sigma)=\frac{c^{2\beta}}{\Gamma(\beta)}
\sigma^{
\beta-1 } \mathrm{e}^ { -\sigma c ^2} \,\mathrm{d}\sigma, \beta>0 \dvt c\in [a,\infty), a>0
\biggr\}
\]
in (b), the inverse multiquadrics kernel
family,
%
\begin{equation}
\Cal{K}= \biggl\{k(x,y)= \biggl(1+\biggl\llVert \frac{x-y}{c} \biggr\rrVert
^2_2 \biggr)^ {
-\beta}, x,y\in\bb{R}^d,
\beta>0 \dvt c\in[a,\infty), a>0 \biggr\}\label{Eq:multiquadrics}
\end{equation}
mentioned in Example~\ref{Exm:weak}
is obtained, where $A=a^{-2r}\frac{\Gamma(r+\beta)}{\Gamma(\beta)}$ with
$\Gamma$ being the
Gamma function.
%
Similarly, choosing
\[
\Cal{M}_{A_i}= \bigl\{\mathrm{d}\Lambda_i(\sigma)=c^2_i\mathrm{e}^{-\sigma
c^2_i}
\,\mathrm{d}\sigma: c_i\in[a_i,\infty), a_i>0 \bigr\}
\]
yields the family
\[
\Cal{K}= \Biggl\{\prod^d_{i=1} \biggl(1+
\frac{|x_i-y_i|^2}{c^2_i} \biggr)^{-1}, x,y\in\bb{R}^d \dvt
c_i\in [a_i,\infty), a_i>0, \forall i=1,
\ldots,d \Biggr\}
\]
in
Example~\ref{Exm:weak} where $A_i=\sup_{j\in
\{1,\ldots,d\}}\alpha_j! a^{-2\alpha_j}_i$. It is easy to verify that these
classes of
kernels metrize the weak topology on $M^1_+(\bb{R}^d)$.

(iii)
Suppose there exists $B>0$ and
$\delta>0$
such that $\inf_{\Lambda\in\Cal{M}_A}\int^\infty_0
\mathrm{e}^{-\delta\sigma^2} \,\mathrm{d}\Lambda(\sigma)\ge B$ (similarly, there exists
$B_i>0$ and $\delta_i>0$
such that $\inf_{\Lambda_i\in\Cal{M}_{A_i}}\int^\infty_0
\mathrm{e}^{-\delta_i\sigma^2} \,\mathrm{d}\Lambda_i(\sigma)\ge B_i, i=1,\ldots,d$), where
$\Cal{M}_A$ and $(\Cal{M}_{A_i})^d_{i=1}$ are defined in (b) and
(c) of
Theorem~\ref{thmm:examples}. Then it is easy to show (see
Section~\ref{sec:proof-weak} for a proof) that $\Vert\cdot\Vert_{\Cal{F}_H}$
metrizes the weak topology on $M^1_+(\bb{R}^d)$, which when combined
with the
result in Theorem~\ref{thmm:examples} yields that for $\Cal{K}$ in (a)--(c),
\[
\bb{P}_n\leadsto\bb{P} \quad\mbox{and}\quad \bb{P}_n\ast
K_h\leadsto\bb{P} \qquad\mbox{a.s.}
\]
at the rate of $n^{-1/2}$.

(iv) It is clear from (\ref{Eq:sobolev}) that any
$k$ in the Mat\'{e}rn family in Example~\ref{Exm:charac}~-- the family in
Theorem~\ref{thmm:examples}(d) is a special case of this~-- induces
an RKHS
which is a Sobolev space,
\[
H_c:=H^{\beta,c}_2= \biggl\{f\in L^2
\bigl(\bb{R}^d\bigr)\cap C\bigl(\bb{R}^d\bigr) \dvt \int
\bigl(c^2+\Vert \omega\Vert^2_2
\bigr)^\beta\bigl|\widehat{f}(\omega)\bigr|^2 \,\mathrm{d}\omega<\infty \biggr\},
\]
where $\beta>d/2$ and $c>0$. Similar to the Gaussian kernel family, it
can be
shown that $H_c\subset H_\alpha$ for $0<c<\alpha<\infty$ since for any
$f\in H_c$,
\[
\Vert f\Vert^2_{H_\alpha}:=\frac{2^{1-\beta}}{A\alpha^{2\beta-d}\Gamma(\beta
)}\int \bigl(
\alpha^2+\Vert \omega\Vert^2_2
\bigr)^\beta\bigl|\widehat{f}(\omega)\bigr|^2 \,\mathrm{d}\omega\le \biggl(
\frac{\alpha}{c} \biggr)^{d}\Vert f\Vert^2_{H_c},
\]
where the definition of $\Vert\cdot\Vert_{H_\alpha}$ follows from
Wendland \cite{Wendland-05}, Theorems
6.13 and 10.12. Therefore, we have
\[
\biggl\{\Vert f\Vert_{H_c}\le \biggl(\frac{c}{a}
\biggr)^{d/2} \dvt f\in H_c, c\in (0,a] \biggr\}\subset\bigl\{f
\in H_a\dvt \Vert f\Vert_{H_a}\le1\bigr\}\subset
\Cal{F}_H
\]
and
%
\begin{equation}
\Cal{F} _H\subset \bigcup_{c\in(0,a]} \biggl
\{f\in H_a\dvt \Vert f\Vert_{H_a}\le \biggl(
\frac{a}{c} \biggr)^{d/2} \biggr\}=H_a,\label{Eq:fh-subset-1}
\end{equation}
where
$\Cal{F}_H:= \{\Vert
f\Vert_{H_c}\le1 \dvt f\in H_c, c\in(0,a] \}$. Unlike in Example~\ref{Exm:charac}, $\Cal{K}$
in Theorem~\ref{thmm:examples}(d) requires $\beta>m+\frac{d}{2}$.
This is to
ensure that every $k\in\Cal{K}$ is $m$-times continuously
differentiable as
required in Theorem~\ref{Thm:estim-to-p}, which is guaranteed by the
Sobolev embedding theorem (Folland \cite{Folland-99}, Theorem~9.17) if
$\beta>m+\frac{d}{2}$. Also, since $\Cal{K}$ metrizes the weak topology (holds
for the Gaussian family as well) on
$M^1_+(\bb{R}^d)$ (see Example~\ref{Exm:weak}), we obtain that
\[
\bb{P}_n\leadsto\bb{P} \quad\mbox{and}\quad \bb{P}_n\ast
K_h\leadsto\bb{P} \qquad\mbox{a.s.}
\]
at the rate of $n^{-1/2}$.
\end{rem}

%
\subsection{Uniform central limit theorem}\label{subsec:uclt-main}
So far, we have presented
exponential concentration inequalities for $\Vert
\bb{P}_n-\bb{P}\Vert_{\Cal{F}_H}$ and $\Vert\bb{P}_n\ast K_h
-\bb{P}\Vert_{\Cal{F}_H}$ in Theorems \ref{Thm:consistency} and
\ref{Thm:estim-to-p}, respectively, and showed that $\sqrt{n} \Vert
\bb{P}_n-\bb{P}\Vert_{\Cal{F}_H}=\mathrm{O}_{\mathrm{a.s.}}{(1)}$ and $\sqrt{n} \Vert
\bb{P}_n\ast K_h-\bb{P}\Vert_{\Cal{F}_H}=\mathrm{O}_{\mathrm{a.s.}}{(1)}$ for families of
$\Cal{K}$ in Theorem~\ref{thmm:examples}. It is therefore easy to note that
if $\Cal{F}_H$ is $\bb{P}$-Donsker, then
$\sqrt{n} (\bb{P}_n-\bb{P} )\leadsto_{\ell^\infty(\Cal{F}_H)}
\bb{G}_\bb{P}$
and so
$\sqrt{n} (\bb{P}_n\ast K_h-\bb{P} )\leadsto_{\ell^\infty(\Cal{F}_H)}
\bb{G}_\bb{P}$ (as $\sqrt{n}\Vert\bb{P}_n\ast
K_h-\bb{P}_n\Vert_{\Cal{F}_H}=\mathrm{o}_{\mathrm{a.s.}}(1)$) for any $h$ satisfying
$\sqrt{n}h^{r}\rightarrow0$ as $h\rightarrow0$ and $n\rightarrow
\infty$.
Here, $\bb{G}_\bb{P}$ denotes the $\bb{P}$-Brownian bridge indexed by
$\Cal{F}_H$. However, unlike Theorems~\ref{Thm:consistency} and~\ref{Thm:estim-to-p} which hold for a general $\Cal{F}_H$, it is not
easy to
verify the $\bb{P}$-Donsker property of $\Cal{F}_H$ for any general~$\Cal{K}$.
In particular, it is not
easy to check whether there exists a pseudometric on $\Cal{F}_H$ such that
$\Cal{F}_H$ is totally bounded (w.r.t. that pseudometric) and $\Cal{F}_H$
satisfies the asymptotic equicontinuity
condition (see Dudley \cite{Dudley-99}, Theorem~3.7.2) or $\Cal{F}_H$
satisfies the uniform entropy condition
(see van der Vaart and Wellner \cite{Vaart-96}, Theorem~2.5.2) as
obtaining estimates on the
$L^2(\bb{P}_n)$
covering number of $\Cal{F}_H$ does not appear to be straightforward.
On the
other hand, in the following result (which is proved in Section~\ref{subsec:thmm-singleton}), we show that $\Cal{F}_H$ is
$\bb{P}$-Donsker for the classes considered in Theorem~\ref{thmm:examples}
(with a slight restriction to the parameter space) and,
therefore, UCLT in $\ell^\infty(\Cal{F}_H)$ holds. A similar result
holds for any
general $\Cal{K}$
(other than the ones in Theorem~\ref{thmm:examples}), if $\Cal{K}$ is
singleton consisting of a bounded continuous kernel.

\begin{thmm}\label{pro:singleton}
Suppose the assumptions on $\bb{P}$ and $K$ in Theorem~\ref{Thm:estim-to-p}
hold and let $0<a<b<\infty$. Define $\Sigma:=[a,b]$. Then for the
following classes of kernels,
\begin{longlist}[(a)]
\item[(a)]
\[
\Cal{K}= \bigl\{k(x,y)=\mathrm{e}^{-\sigma\Vert
x-y\Vert^2_2}, x,y\in\bb{R}^d \dvt \sigma
\in\Sigma \bigr\};
\]
\item[(b)]
\[
\Cal{K}= \biggl\{k(x,y)= \biggl(1+\biggl\llVert \frac{x-y}{c} \biggr\rrVert
^2_2 \biggr)^{-\beta}, x,y\in\bb{R}^d,
\beta>0 \dvt c\in \Sigma \biggr\};
\]
\item[(c)]
\[
\Cal{K}= \biggl\{k(x,y)=A\frac{\llVert
x-y\rrVert ^{\beta-{d}/{2}}_2}{c^{{d}/{2}-\beta}}\frak {K}_{{d}/{2
} -\beta} \bigl(c\Vert x-y
\Vert_2 \bigr), x,y\in\bb{R}^d, \beta>m+\frac{d}{2}
\dvt c\in\Sigma \biggr\},
\]
where $A:=\frac{2^{{d}/{2}+1-\beta}}{\Gamma(\beta-{d}/{2})}$
and $m\in\bb{N}$;
\item[(d)] $\Cal{K}=\{k\}$ where $k$ satisfies the conditions in
Theorem~\ref{Thm:estim-to-p},
\end{longlist}
$\sqrt{n} (\bb{P}_n-\bb{P})\leadsto_{\ell^\infty(\Cal{F}_H)}
\bb{G}_\bb{P}$ and for any $h$ satisfying
$\sqrt{n}h^{r}\rightarrow0$ as $h\rightarrow0$ and $n\rightarrow
\infty$,
we have
\[
\sqrt{n} (\bb{P}_n\ast K_h-\bb{P})
\leadsto_{\ell^\infty(\Cal{F}_H)} \bb{G}_\bb{P},
\]
which particularly holds for $h^\ast$ and $h^{\ast\ast}$ if $r>s+\frac
{d}{2}$, where $h^\ast$
and $h^{\ast\ast}$ are defined in Remark~\ref{Rem:main}\textup{(i)}.
\end{thmm}

Theorem~7 in Gin{\'{e}} and Nickl \cite{Gine-08a} shows the above
result for Mat\'{e}rn kernels
(i.e., $\Cal{K}$ in (c) with $c=1$ and $d=1$), but here we generalize
it to a
wide
class of kernels. Theorem~\ref{pro:singleton}(d) shows that all the kernels
(with the parameter fixed a priori, for example, $\sigma$ in the
Gaussian kernel) we
have encountered so far -- such as in
Examples \ref{Exm:charac} and \ref{Exm:weak}~-- satisfy the conditions in
Theorem~\ref{pro:singleton} and, therefore, yield a UCLT. Note that the kernel
classes, $\Cal{K}$ in Theorem~\ref{pro:singleton} are slightly constrained
compared to those in Theorem~\ref{thmm:examples}
and Remark~\ref{rem:gauss-exm}(ii). This restriction in the kernel
class is
required as the proof of $\Cal{F}_H$ being $\bb{P}$-Donsker (which in
combination with Slutsky's lemma and Theorem~\ref{thmm:examples} yields the
desired result in Theorem~\ref{pro:singleton}) critically hinges on the
inclusion result shown in (\ref{Eq:fh-subset}) and (\ref
{Eq:fh-subset-1}); also
see (\ref{Eq:fh-subset2}) for such an inclusion result for $\Cal{K}$ in
Theorem~\ref{pro:singleton}(b). However, this technique is not
feasible for the
kernel classes, (b) and (c) in Theorem~\ref{thmm:examples} to be
shown as
$\bb{P}$-Donsker, while we reiterate that for any general $\Cal{K}$, it is
usually difficult to check for the Donsker property of $\Cal{F}_H$.

Combining Theorems \ref{Thm:weak}, \ref{thmm:examples} and
\ref{pro:singleton}, we obtain that the kernel density estimator with
bandwidth $h^*$ is an optimal estimator of $p$ in both strong and weak
topologies unlike $\bb{P}_n$, which estimates $\bb{P}$ optimally only
in the
weak topology. While this optimality result holds in $d=1$ when using
$\Vert
\cdot\Vert_{\Cal{F}_{\mathrm{BL}}}$ as the loss to measure the optimality of
$\bb{P}_n\ast K_h$ in the weak sense, the result does not hold for
$d\ge2$ as
discussed before. In addition, for $d=1$, the UCLT for
$\sqrt{n}(\bb{P}_n-\bb{P})$ and $\sqrt{n}(\bb{P}_n\ast K_h-\bb{P})$ in
$\ell^\infty(\Cal{F}_{\mathrm{BL}})$ holds only under a certain moment condition on
$\bb{P}$, that is, $\int|x|^{1+\gamma} \,\mathrm{d}\bb{P}(x)<\infty$ for some
$\gamma>0$ (see Gin{\'{e}} and Zinn \cite{Gine-86}, Theorem~2) while no
such condition on $\bb{P}$
is required
to obtain the UCLT for the above processes in
$\ell^\infty(\Cal{F}_H)$ though both $\Vert\cdot\Vert_{\Cal{F}_{\mathrm{BL}}}$
and $\Vert\cdot\Vert_{\Cal{F}_H}$ metrize the weak topology on
$M^1_+(\bb{R}^d)$.
\section{Discussion}\label{Sec:adaptive}
So far we have shown that the kernel density estimator on $\bb{R}^d$
with an
appropriate choice of bandwidth is an optimal estimator of $\bb{P}$ in
$\Vert\cdot\Vert_{\Cal{F}_H}$, that is, in
weak topology, similar to $\bb{P}_n$. In Section~\ref{subsec:alternate}, we
present a similar result for an alternate metric
$\Vert\cdot\Vert_{\Cal{K}_\Cal{X}}$ (defined below) that is topologically
equivalent to $\Vert\cdot\Vert_{\Cal{F}_H}$, that is, metrizes the weak
topology on
$M^1_+(\Cal{X})$ where $\Cal{X}$ is a
topological space and $\Cal{K}_\Cal{X}\subset\Cal{F}_H$, showing that
$\Cal{F}_H$ is not the only class that guarantees the optimality of kernel
density estimator in weak and strong topologies. While a
result similar to these is shown in $\Vert\cdot\Vert_{\Cal{F}_{\mathrm{BL}}}$
for $d=1$
in Gin{\'{e}} and Nickl \cite{Gine-08b}, there is a significant
computational advantage associated
with $\Cal{F}_H$ over $\Cal{K}_\Cal{X}$ and $\Cal{F}_{\mathrm{BL}}$ in the context
of constructing adaptive estimators that are optimal in both strong and weak
topologies, which we discuss in Section~\ref{subsec:adaptive}.
\subsection{Optimality in
\texorpdfstring{$\Vert\cdot\Vert_{\Cal{K}_\Cal{X}}$}{$||cdot||_{K_X}$}}\label{subsec:alternate}
In this section, we consider an alternate metric, $\Vert
\cdot\Vert_{\Cal{K}_\Cal{X}}$, which we show in Proposition~\ref
{pro:bound} (proved in Section~\ref{subsec:probound}) to
be topologically equivalent to
$\Vert\cdot\Vert_{\Cal{F}_H}$ if $\Cal{K}$
is uniformly bounded, where
\[
\Cal{K}_\Cal{X}:= \bigl\{k(\cdot,x) \dvt k\in\Cal{K}, x\in\Cal{X}
\bigr\}
\]
and $\Cal{X}$ is a topological space. Note that
$\Cal{K}_\Cal{X}\subset\Cal{F}_H$ if $k(x,x)\le1, \forall x\in\Cal{X},
k\in\Cal{K}$, which means a reduced subset of $\Cal{F}_H$ is sufficient to
metrize the weak topology on $M^1_+(\Cal{X})$.

\begin{pro}\label{pro:bound}
Suppose $\nu:=\sup_{k\in\Cal{K},x\in\Cal{X}}k(x,x)<\infty$. Then for any
$\bb{P},\bb{Q}\in M^1_+(\Cal{X})$
%
\begin{equation}
\nu^{-1/2}\Vert\bb{P}-\bb{Q}\Vert_{\Cal{K}_\Cal{X}}\le\Vert \bb{P}-\bb{Q}
\Vert_{\Cal{F}_H}\le\sqrt{2 \Vert\bb{P}-\bb{Q} \Vert_{\Cal{K}_\Cal{X}}},\label{Eq:bound-distance}
\end{equation}
where
\[
\Vert \bb{P}-\bb{Q}\Vert_{\Cal{K}_\Cal{X}}=\sup_{k\in\Cal{K}} \biggl
\llVert \int k(\cdot,x) \,\mathrm{d}\bb{P}(x)-\int k(\cdot,x) \,\mathrm{d}\bb{Q}(x)\biggr\rrVert
_\infty.
\]
In addition if $\Cal{K}$ satisfies the assumptions in Theorem~\ref{Thm:weak},
then for any sequence $(\bb{P}_{(n)})_{n\in\bb{N}}\subset M^1_+(\Cal
{X})$ and
$\bb{P}\in M^1_+(\Cal{X})$,
%
\begin{equation}
\Vert \bb{P}_{(n)}-\bb{P}\Vert_{\Cal{K}_\Cal{X}}\rightarrow0
\quad\Longleftrightarrow\quad \Vert\bb{P}_{(n)}-\bb{P}\Vert_{\Cal{F}_H}
\rightarrow0 \quad\Longleftrightarrow\quad \bb{P}_{(n)}\leadsto\bb{P} \qquad\mbox{as } n
\rightarrow\infty.\label{Eq:chain-pro}
\end{equation}
\end{pro}

From (\ref{Eq:bound-distance}), it simply follows that
\[
\sqrt{n} \Vert \bb{P}_n-\bb{P}\Vert_{\Cal{K}_\Cal{X}}=\mathrm{O}_{\mathrm{a.s.}}(1),\qquad
\sqrt{n} \Vert \bb{P}_n\ast K_h-\bb{P}_n
\Vert_{\Cal{K}_\Cal{X}}=\mathrm{o}_{\mathrm{a.s.}}(1)
\]
and
\[
\sqrt{n} \Vert \bb{P}_n\ast K_h-\bb{P}
\Vert_{\Cal{K}_\Cal{X}}=\mathrm{O}_{\mathrm{a.s.}}(1)
\]
for any $\Cal{K}$
in Theorem~\ref{Thm:estim-to-p} with $\omega_\ast< 1$
and $\omega_\Cal{K}< 1$ (and, therefore, for any $\Cal{K}$
in Theorem~\ref{thmm:examples}) with $h$ satisfying
$\sqrt{n}h^{r}\rightarrow0$ as $h\rightarrow0$ and $n\rightarrow
\infty$. Therefore, if $\Cal{K}_\Cal{X}$ is $\bb{P}$-Donsker, then for
any $h$
satisfying these conditions, we obtain
\[
\sqrt{n} (\bb{P}_n\ast K_h-\bb{P})
\leadsto_{\ell^\infty(\Cal{K}_\Cal{X})} \bb{G}_\bb{P}.
\]
The following result (proved in Section~\ref{subsec:uclt}) shows that
$\Cal{K}_\Cal{X}$ is
a
universal Donsker class (i.e., $\bb{P}$-Donsker for all probability measures
$\bb{P}$ on $\bb{R}^d$) for $\Cal{K}$ considered in
Theorem~\ref{thmm:examples} and therefore we obtain UCLT for
$\sqrt{n} (\bb{P}_n-\bb{P})$ and
$\sqrt{n} (\bb{P}_n\ast K_h-\bb{P})$ in $\ell^\infty(\Cal{K}_\Cal{X})$.

\begin{thmm}\label{thmm:uclt}
Suppose the assumptions on $\bb{P}$ and $K$ in Theorem~\ref{Thm:estim-to-p}
hold. Define $\Cal{K}_\Cal{X}:=\{k(\cdot,x)\dvt k\in\Cal{K}, x\in\Cal{X}\}
$. Then
for $\Cal{K}$ in Theorem~\ref{thmm:examples}, $\Cal{K}_\Cal{X}$ is a
universal Donsker class and
\[
\sqrt{n} (\bb{P}_n-\bb{P})\leadsto_{\ell^\infty(\Cal{K}_\Cal{X})}
\bb{G}_\bb{P} \quad\mbox{and}\quad \sqrt{n} (\bb{P}_n\ast
K_h-\bb{P})\leadsto_{\ell^\infty(\Cal{K}_\Cal{X})} \bb{G}_\bb{P},
\]
for $h$ satisfying
$\sqrt{n}h^{r}\rightarrow0$ as $h\rightarrow0$ and $n\rightarrow
\infty$,
which particularly holds for $h^\ast$ and $h^{\ast\ast}$ if $r>s+\frac
{d}{2}$, where $h^\ast$
and $h^{\ast\ast}$ are defined in Remark~\ref{Rem:main}\textup{(i)}.
\end{thmm}

Combining Theorem~\ref{Thm:weak} and Proposition~\ref{pro:bound}, along
with Theorems \ref{thmm:examples} and
\ref{thmm:uclt}, it is clear that the kernel density estimator with
bandwidth $h^*$ is an optimal estimator of $p$ in both strong and weak
topologies (induced by $\Vert\cdot\Vert_{\Cal{K}_\Cal{X}}$). While this result
matches with the one obtained for
$\Vert\cdot\Vert_{\Cal{F}_H}$, by comparing Theorems \ref
{pro:singleton} and
\ref{thmm:uclt}, we note that the convergence in
$\ell^\infty(\Cal{K}_\Cal{X})$ does not require the restriction in the parameter
space as imposed in kernel classes for
convergence in $\ell^\infty(\Cal{F}_H)$ in Theorem~\ref{pro:singleton}.
However, we show in the following section that $\Vert\cdot\Vert_{\Cal
{F}_H}$ is
computationally easy to deal with than $\Vert\cdot\Vert_{\Cal{K}_\Cal{X}}$.

\subsection{Adaptive estimation and computation}
\label{subsec:adaptive}
Let us return to the fact that there exists estimators that are
$\mathrm{o}_{\bb{P}}(n^{-1/2})$ from $\bb{P}_n$ in $\Vert\cdot\Vert_\Cal{F}$ (for
suitable choice of $\Cal{F}$) and behave
statistically similar to $\bb{P}_n$. While we
showed this fact through Theorems \ref{Thm:estim-to-p} and \ref{thmm:examples}
for the kernel density estimator with $\Cal{F}=\Cal{F}_H$ (and
Proposition~\ref{pro:bound} for
$\Cal{F}=\Cal{K}_\Cal{X}$), Gin{\'{e}} and Nickl \cite
{Gine-08a,Gine-09a,Gine-08b} showed the same
result
with $\Cal{F}$ being functions of bounded variation,
$\{\mathds{1}_{(-\infty,t]}\dvt t\in\bb{R}\}$, H\"{o}lder, Lipschitz and Sobolev
classes on $\bb{R}$. Similar result is shown for wavelet density estimators
and spline projection estimators in
$\Cal{F}=\{\mathds{1}_{(-\infty,t]}\dvt t\in\bb{R}\}$ (Gin{\'{e}} and Nickl
\cite{Gine-09b,Gine-10}) and
maximum likelihood estimators in $\Cal{F}_{\mathrm{BL}}$ (Nickl \cite
{Nickl-07}). While
$\bb{P}_n$ is simple and elegant to
use in practice, these other estimators that are
$\mathrm{o}_{\bb{P}}(n^{-1/2})$ from $\bb{P}_n$ have been shown to improve upon
it in the
following
aspect: without any assumption on $\bb{P}$, it is possible to construct
adaptive estimators that estimate $\bb{P}$ efficiently in $\Cal{F}$ and
at the
same time estimate the density of $\bb{P}$ (if it exists without a priori
assuming its existence) at the best possible convergence rate in some relevant
loss over prescribed class of densities. Concretely, Gin{\'{e}} and
Nickl \cite{Gine-09a,Gine-10}
proved the above
behavior for kernel density estimator, wavelet density estimator and spline
projection estimator on $\bb{R}$ for
$\Cal{F}=\{\mathds{1}_{(-\infty,t]}\dvt t\in\bb{R}\}$ and sup-norm loss
over the H\"{o}lder balls.
By choosing
$\Cal{F}=\Cal{F}_{\mathrm{BL}}$ (with $d=1$), Gin{\'{e}} and Nickl \cite
{Gine-08b} showed that the kernel
density
estimator adaptively estimates $\bb{P}$ in weak topology and
at the same
time estimates the density of $\bb{P}$ in strong topology at the best possible
convergence rate over Sobolev balls.

The construction of these adaptive estimators involves applying Lepski's
method (Lepski, Mammen and Spokoiny \cite{Lepski-97}) to kernel density
estimators (in fact to any of the
other estimators we discussed above) that are within a $\Vert
\cdot\Vert_\Cal{F}$-ball of size smaller than $n^{-1/2}$ around
$\bb{P}_n$ and then using the exponential inequality of the type in
Theorem~\ref{Thm:estim-to-p} to control the probability of the event that
$\sqrt{n}\Vert\bb{P}_n\ast K_h-\bb{P}_n\Vert_{\Cal{F}}$ is ``too
large'' (see
Gin{\'{e}} and Nickl \cite{Gine-08b}, Theorem~1, \cite{Gine-09a}, Theorem~2, and
\cite{Gine-10}, Theorem~3, for the optimality of the adaptive estimator in
both $\Vert\cdot\Vert_\Cal{F}$ and some relevant loss over prescribed
class of
densities). Using Theorem~\ref{Thm:estim-to-p}, it is quite
straightforward in
principle to construct an adaptive estimator that is optimal
in both strong and weak topologies along the lines of Gin{\'{e}} and
Nickl (\cite{Gine-08b}, Theorem~1), by incorporating two minor changes in the proof of Theorem~1 in
\cite{Gine-08b}: the first change is to apply Theorem~\ref
{Thm:estim-to-p} in
the place of Lemma~1 and extend Lemma~2 in Gin{\'{e}} and Nickl \cite
{Gine-08b} from
$\bb{R}$ to $\bb{R}^d$. Informally, the procedure involves computing the
bandwidth $\tilde{h}_n$ as
%
\begin{eqnarray}
\tilde{h}_n&=&\max \biggl\{h\in\eu{H}\dvt \bigl\Vert
\bb{P}_n\ast( K_h-K_g)
\bigr\Vert_{L^1}\le \sqrt{\frac{A}{ng^d}}, \forall g<h, g\in\eu{H}
\nonumber
\\[-8pt]
\\[-8pt]
\nonumber
& & \hspace*{24pt}\mbox{and } \Vert\bb{P}_n\ast K_h-
\bb{P}_n\Vert_{\Cal{F}}\le\frac{n^{-{1}/{2}}}{\log n} \biggr\},
\label{Eq:lepski}
\end{eqnarray}
where
$\eu{H}:= \{h_k=\rho^{-k}\dvt k\in\bb{N}\cup\{0\}, \rho^{-k}>(\log
n)^2/n \}$ and $\rho>1$ is arbitrary. Here, $A$ depends on some moment
conditions on $\bb{P}\in\Cal{P}(\gamma,L)$, specifically through $\gamma
$ and
$L$, where
\[
\Cal{P}(\gamma,L)= \biggl\{\bb{P}\in M^1_+\bigl(\bb{R}^d
\bigr) \dvt \int\bigl(1+\Vert x\Vert^2_2
\bigr)^{\gamma} \,\mathrm{d}\bb{P}(x)\le L \biggr\}
\]
for some $L<\infty$ and $\gamma>\frac{d}{2}$. Along the lines of
Theorem~1 in
Gin{\'{e}} and Nickl \cite{Gine-08b}, the following result can be
obtained (we state here without a
proof) that shows the kernel density estimator with a purely data-driven
bandwidth, $\tilde{h}_n$ to be optimal in both strong and weak topologies.

\begin{thmm}\label{thmm:adaptive}
Let $(X_i)^n_{i=1}$ be random samples drawn i.i.d. from a probability
measure $\bb{P}\in\Cal{P}(\gamma,L)$ for some $L<\infty$ and
$\gamma>\frac{d}{2}$. Suppose $K$ is of order $r$ satisfying $r>T+\frac{d}{2}$,
$T\in\bb{N}\cup\{0\}$ such that $\int_{\bb{R}^d}(1+\Vert
x\Vert^2_2)^{\gamma}K^2(x) \,\mathrm{d}x<\infty$ where $p\in W^s_1(\bb{R}^d)$ for
some $0<s\le T$. If $\Cal{F}_H$ is $\bb{P}$-Donsker (satisfied by $\Cal
{K}$ in
Theorem~\ref{pro:singleton}), then
\[
\Vert\bb{P}_n\ast K_{\tilde{h}_n}-\bb{P}\Vert_{\Cal{F}_H}=\mathrm{O}_{\bb{P}}
\bigl(n^{-1/2}\bigr) \quad\mbox{and}\quad \sqrt{ n } (\bb{ P } _n\ast
K_{\tilde{h}_n}-\bb{P})\leadsto_{\ell^\infty(\Cal{F}_H)}\bb{G}_\bb{P}.
\]
Similarly, for $\Cal{K}$ in Theorem~\ref{thmm:examples}, we have
\[
\Vert\bb{P}_n\ast K_{\tilde{h}_n}-\bb{P}\Vert_{\Cal{K}_\Cal{X}}=\mathrm{O}_{\bb{P}}
\bigl(n^{-1/2}\bigr) \quad\mbox{and}\quad \sqrt{ n } (\bb{ P } _n\ast
K_{\tilde{h}_n}-\bb{P})\leadsto_{\ell^\infty(\Cal{K}_\Cal{X})}\bb{G}_\bb{P}.
\]
In addition, for any $0<s\le T$,
\[
\Vert \bb{P}_n\ast K_{\tilde{h}_n}-p\Vert_{L^1}=\mathrm{O}_{\bb{P}}
\bigl(n^{-{s}/{(2s+d)}} \bigr).
\]
\end{thmm}

We now discuss some computational aspects of the estimator in
(\ref{Eq:lepski}), which requires computing $\Vert\bb{P}_n\ast
K_h-\bb{P}_n\ast K_g\Vert_{L^1}$ and $\Vert\bb{P}_n\ast
K_h-\bb{P}_n\Vert_{\Cal{F}}$. While computing $\Vert\bb{P}_n\ast
K_h-\bb{P}_n\ast K_g\Vert_{L^1}$ is usually not straightforward, the
computation of $\Vert\bb{P}_n\ast
K_h-\bb{P}_n\Vert_{\Cal{F}}$ can be simple depending on the choice of
$\Cal{F}$. In the following, we show that $\Cal{F}=\Cal{F}_H$ yields a simple
maximization problem over a subset of $(0,\infty)$ depending on the
choice of
$\Cal{K}$ and $K$, in contrast to an infinite dimensional optimization problem
that would
arise if $\Cal{F}=\Cal{F}_{\mathrm{BL}}$ and optimization over $\bb{R}^{d}\times
(0,\infty)$ if $\Cal{F}=\Cal{K}_\Cal{X}$ therefore demonstrating the
computational
advantage of working with $\Cal{F}_H$ over $\Cal{F}_{\mathrm{BL}}$ and $\Cal
{K}_\Cal{X}$.

Consider $\Vert\bb{P}_n\ast
K_h-\bb{P}_n\Vert_{\Cal{F}_H}$, which from (\ref{Eq:MMD}) and (\ref{Eq:MMD-1})
yields
\begin{eqnarray*}
&&\Vert\bb{P}_n\ast K_h-\bb{P}_n
\Vert_{\Cal{F}_H}\\
&&\quad=\sup_{k\in\Cal{K}}\Biggl\llVert
\frac{1}{n} \sum^n_{i=1}\int
K_h(X_i-x)k(\cdot,x) \,\mathrm{d}x-\frac{1}{n}\sum
^n_{i=1}k(\cdot, X_i)\Biggr\rrVert
_{\Cal{H}_k}
\\
&&\quad=\frac{1}{n}\sup_{k\in\Cal{K}}\sqrt{\sum
^n_{i,j=1}\eu{A}(X_i,
X_j)+k(X_i, X_j) -2\int K_h(x)k(X_i-x,X_j)
\,\mathrm{d}x},
\end{eqnarray*}
which in turn reduces to
%
\begin{equation}
\Vert\bb{P}_n\ast K_h-\bb{P}_n
\Vert_{\Cal{F}_H}= \sqrt{\frac{1}{n^2}\sup_{k\in\Cal{K}}
\sum^n_{i,j=1} (K_h\ast
K_h\ast\psi+\psi-2K_h\ast\psi) (X_i-X_j)}\label{Eq:compute}
\end{equation}
when $k$ is translation invariant, that is,
$k(x,y)=\psi(x-y), x,y\in\bb{R}^d$, where
\[
\eu{A}(X_i,X_j):=\int\int K_h(x)K_h(y)k(X_i-x,X_j-y)
\,\mathrm{d}x \,\mathrm{d}y.
\]
While computing
(\ref{Eq:compute}) is not easy in general, in the following we present two
examples where~(\ref{Eq:compute}) is easily computable for appropriate
choices of $\Cal{K}$ and $K$. Let $\Cal{K}$ be as in
Theorem~\ref{thmm:examples}(a) (i.e., $\psi(x):=\psi_\sigma
(x)=\mathrm{e}^{-\sigma\Vert
x\Vert^2_2}, x\in\bb{R}^d, \sigma\in\Sigma$) and $K=\pi^{-d/2}\psi_1$. Then
%
\begin{equation}
\Vert\bb{P}_n\ast K_h-\bb{P}_n
\Vert_{\Cal{F}_H}=\frac{1}{n}\sqrt{\sup_{\sigma\in\Sigma}
\Cal{A}_\psi(\sigma)},\label{Eq:compute-fh}
\end{equation}
where
\[
\Cal{A}_\psi(\sigma):= \sum^n_{i,j=1}
\biggl(\frac{\psi_{{\sigma}/{
(2\sigma h^2+1)}}(X_i-X_j)}{(2\pi)^d(2\sigma
h^2+1)^{d/2}}-\frac{2\psi_{{\sigma}/{(\sigma h^2+1)}}(
X_i-X_j)}{(2\pi)^{d/2}(\sigma
h^2+1)^{d/2}}+\psi_{\sigma}(X_i-X_j)
\biggr).
\]
Also choosing $\Cal{K}$ to be as in Remark~\ref{rem:gauss-exm}, that is,
\[
\psi(x):=\phi_\alpha(x)=\prod^d_{i=1}
\frac{\alpha^2}{\alpha^2+x^2_i},\qquad x\in\bb { R } ^d, \alpha\in[c,\infty), c>0,
\]
which is a
special case of $\Cal{K}$ in Theorem~\ref{thmm:examples}(c)
and $K=\pi^{-d}\phi_1$ in (\ref{Eq:compute}) yields
\[
\Vert\bb{P}_n\ast K_h-\bb{P}_n
\Vert_{\Cal{F}_H}=\frac{1}{n}\sqrt{\sup_{\alpha\in
[c,\infty)}
\Cal{A}_\phi(\alpha)},
\]
where
\[
\Cal{A}_\phi(\alpha):=\sum^n_{i, j=1 }
\biggl(\frac{\phi_{\alpha+2h}(X_i-X_j)}{\alpha^{-d}2^d(\alpha
+2h)^d}-\frac{2\phi_
{ \alpha+h}
(X_i-X_j)}{2^{d/2}\alpha^{-d}(\alpha+h)^d}+\phi_\alpha(X_i-X_j)
\biggr).
\]
In both these examples (where $\Vert\cdot\Vert_{\Cal{F}_H}$ metrizes
the weak
topology on $M^1_+(\bb{R}^d)$), it is clear that one can compute
$\Vert\bb{P}_n\ast
K_h-\bb{P}_n\Vert_{\Cal{F}_H}$ easily by solving a maximization problem
over a
subset of $(0,\infty)$, which can be carried out using standard
gradient ascent
methods. For the choice of $K$ in both these examples, it is easy to see
that $K$ is of order $2$ and
therefore Theorem~\ref{thmm:examples} holds if $s<2-\frac{d}{2}$.

On the other hand, note that
\[
\Vert\bb{P}_n\ast K_h-\bb{P}_n
\Vert_{\Cal{F}_{\mathrm{BL}}}=\frac{1}{n}\sup_{f\in\Cal{F}_{\mathrm{BL}}}\Biggl\llvert
\sum^n_
{i=1}(K_h\ast f-f)
(X_i)\Biggr\rrvert
\]
is not easily computable in practice. Also for $\Cal{F}=\Cal{K}_\Cal
{X}$, we
have
\[
\Vert\bb{P}_n\ast K_h-\bb{P}_n
\Vert_{\Cal{K}_\Cal{X}}=\sup_{k\in\Cal{K},y\in\Cal{X}} \Biggl\llvert
\frac{1}{n} \sum^n_{i=1}\int
K_h(X_i-x)k(y,x) \,\mathrm{d}x-\frac{1}{n}\sum
^n_{i=1}k(y,X_i)\Biggr\rrvert ,
\]
which reduces to
\[
\Vert\bb{P}_n\ast K_h-\bb{P}_n
\Vert_{\Cal{K}_\Cal{X}}=\sup_{k\in\Cal{K},y\in\Cal{X}} \Biggl\llvert
\frac{1}{n} \sum^n_{i=1}(K_h
\ast\psi-\psi) (y-X_i)\Biggr\rrvert
\]
when $k(x,y)=\psi(x-y)$. For the choice of $K$ and $\Cal{K}$ as above (i.e.,
$\psi_\sigma$ and $\phi_\alpha$), it is easy to verify that the
computation of
$\Vert\bb{P}_n\ast
K_h-\bb{P}_n\Vert_{\Cal{K}_\Cal{X}}$ involves solving an optimization problem
over $\bb{R}^d\times(0,\infty)$ which is more involved than solving the
one in
(\ref{Eq:compute-fh}) that is obtained by working with
$\Cal{F}_H$.

In addition to the above application of adaptive estimation, there
are various statistical applications where the choice of $\Cal{F}_H$
can be
computationally useful (over $\Cal{F}_{\mathrm{BL}}$ and $\Cal{K}_\Cal{X}$), the
examples of which include the two-sample and independence testing. As an
example, in two-sample testing, $\Vert\bb{P}_n\ast K_h-\bb{Q}_m\ast
K_g\Vert_{\Cal{F}_H}$ can be used as a statistic to test for
$\bb{P}=\bb{Q}$ vs. $\bb{P}\ne\bb{Q}$ based on $n$ and $m$ numbers of random
samples drawn i.i.d. from $\bb{P}$ and $\bb{Q}$
respectively, assuming these distributions to have densities w.r.t. the Lebesgue
measure. Based on the above discussion, it is easy to verify that
$\Vert\bb{P}_n\ast K_h-\bb{Q}_m\ast K_g\Vert_{\Cal{F}}$ is simpler to compute
when $\Cal{F}=\Cal{F}_H$ compared to the other choices of $\Cal{F}$
such as
$\Cal{K}_\Cal{X}$ and $\Cal{F}_{\mathrm{BL}}$. Similarly, computationally
efficient test
statistics can be obtained for
nonparametric independence tests through $\Vert\cdot\Vert_{\Cal{F}_H}$.

\section{Proofs}\label{Sec:proofs}
In this section, we present the missing proofs of results in
Sections~\ref{Sec:rkhs} and \ref{Sec:mainresult}.

\subsection{Proof of Proposition \texorpdfstring{\protect\ref{pro:embedding}}{3.1}}\label
{subsec:embedding}
For any $f\in\Cal{H}_k$ and $\bb{P}\in\Scr{P}_\Cal{K}$, we have
\[
\int f(x) \,\mathrm{d}\bb{P}(x)=\int\bigl\langle f,k(\cdot,x)\bigr\rangle_{\Cal{H}_k} \,\mathrm{d}
\bb{P}(x)= \biggl\langle f,\int k(\cdot,x) \,\mathrm{d}\bb{P}(x) \biggr\rangle_{\Cal{H}_k},
\]
where the last equality
follows from
the assumption that $k$ is Bochner-integrable, that is,
\[
\int\bigl\Vert k(\cdot,x)\bigr\Vert_{\Cal{H}_k} \,\mathrm{d}\bb{P}(x)=\int \sqrt{k(x,x)} \,\mathrm{d}
\bb{P}(x)<\infty.
\]
Therefore, for any $\bb{P},\bb{Q}\in\Scr{P}_{\Cal{K}}$,
\begin{eqnarray*}
\sup_{k\in\Cal{K}}\sup_{\Vert f\Vert_{\Cal{H}_k}\le
1}\biggl\llvert \int f(x)
\,\mathrm{d}(\bb{P}-\bb{Q}) (x)\biggr\rrvert &=&\sup_{k\in\Cal{K}}\sup
_{\Vert
f\Vert_{\Cal{H}_k}\le1} \biggl\langle f,\int k(\cdot,x) \,\mathrm{d}(\bb{P}-\bb{Q}) (x)
\biggr\rangle_{\Cal{H}_k}
\\
&=&\sup_{k\in\Cal{K}}\biggl\llVert \int k(\cdot,x) \,\mathrm{d}(\bb{P}-
\bb{Q}) (x)\biggr\rrVert _{\Cal{H}_k},
\end{eqnarray*}
where the inner supremum is attained at $f=\frac{\int
k(\cdot,x) \,\mathrm{d}(\bb{P}-\bb{Q})(x)}{\llVert \int
k(\cdot,x) \,\mathrm{d}(\bb{P}-\bb{Q})(x)\rrVert _{\Cal{H}_k}}$. Because of the
Bochner-integrability of
$k$,
\[
\biggl\langle\int k(\cdot,x) \,\mathrm{d}\bb{P}(x),\int k(\cdot,y) \,\mathrm{d}\bb{Q}(y) \biggr
\rangle_{\Cal{H}_k}=\int\int k(x,y) \,\mathrm{d}\bb{P}(x) \,\mathrm{d}\bb{Q}(y)
\]
and (\ref{Eq:MMD-1}) follows.

\subsection{Proof of Theorem \texorpdfstring{\protect\ref{Thm:consistency}}{3.3}}\label{subsec:consistency}
Since $\sup_{k\in\Cal{K},x\in\Cal{X}}k(x,x)\le\nu$ and
$\bb{P}_n,\bb{P}\in\Scr{P}_\Cal{K}$, by Proposition~\ref
{pro:embedding}, we have
\[
\Vert\bb{P}_n-\bb{P}\Vert_{\Cal{F}_H}=\sup_{k\in\Cal{K}}
\biggl\llVert \int k(\cdot,x) \,\mathrm{d}(\bb{P}_n-\bb{P}) (x)\biggr\rrVert
_{\Cal{H}_k}.
\]
It is easy to check
that $\sup_{k\in\Cal{K}}\llVert \int
k(\cdot,x) \,\mathrm{d}(\bb{P}_n-\bb{P})(x)\rrVert _{\Cal{H}_k}$ satisfies the bounded
difference property and therefore, by McDiarmid's inequality, for every
$\tau>0$, with probability at least $1-\mathrm{e}^{-\tau}$,
%
\begin{eqnarray}
\label{Eq:Mc-1}\sup_{k\in\Cal{K}}\biggl\llVert \int k(\cdot,x) \,\mathrm{d}(
\bb{P}_n-\bb{P}) (x)\biggr\rrVert _{\Cal{H}_k}&\le & \bb{E}
\sup_{k\in\Cal{K}}\biggl\llVert \int k(\cdot,x) \,\mathrm{d}(\bb{P}_n-
\bb{P}) (x)\biggr\rrVert _{\Cal{H}_k}+\sqrt{\frac{2\nu
\tau}{n
} }
\nonumber
\\[-8pt]
\\[-8pt]
\nonumber
&\stackrel{(\ast)} {\le} & 2 \bb{E}\bb{E}_\varepsilon\sup
_{k\in\Cal{K}}\Biggl\llVert \frac{1}{n}\sum
^n_{i=1} \varepsilon_i k(
\cdot,X_i)\Biggr\rrVert _{\Cal{H}_k}+\sqrt{\frac{2\nu\tau}{n}},
\end{eqnarray}
where $(\varepsilon_i)^n_{i=1}$ represent i.i.d. Rademacher
variables, $\bb{E}_\varepsilon$ represents the expectation
w.r.t. $(\varepsilon_i)^n_{i=1}$ conditioned on $(X_i)^n_{i=1}$, and
$(\ast)$ is obtained by symmetrizing $\bb{E}\sup_{k\in\Cal{K}}\frak
{D}_k(\bb{P}_n,\bb{P})$
(see van der Vaart and Wellner \cite{Vaart-96}, Lemma~2.3.1). Since
$\bb{E}_\varepsilon\sup_{k\in\Cal{K}}\llVert \frac{1}{n}\sum^n_{i=1}
\varepsilon_i k(\cdot,X_i)\rrVert _{\Cal{H}_k}$ satisfies the bounded
difference property, another application of McDiarmid's inequality yields
that, for every $\tau>0$, with probability at least $1-\mathrm{e}^{-\tau}$,
%
\begin{eqnarray}\label{Eq:Mc-2}
&&\bb{E}\bb{E}_\varepsilon\sup_{k\in\Cal{K}}\Biggl\llVert
\frac{1}{n}\sum^n_{i=1}
\varepsilon_i k(\cdot,X_i)\Biggr\rrVert _{\Cal{H}_k}
\nonumber
\\[-8pt]
\\[-8pt]
\nonumber
&&\quad\le \bb{E}_\varepsilon\sup_{k\in\Cal{K}}\Biggl\llVert
\frac{1}{n}\sum^n_{i=1}
\varepsilon_i k(\cdot,X_i)\Biggr\rrVert
_{\Cal{H}_k}+\sqrt{\frac{2\nu\tau}{n}}
\end{eqnarray}
and, therefore, combining (\ref{Eq:Mc-1}) and (\ref{Eq:Mc-2}) yields
that for
every $\tau>0$, with probability at least $1-2\mathrm{e}^{-\tau}$,
%
\begin{eqnarray}\label{Eq:Mc-3}
&&\sup_{k\in\Cal{K}}\biggl\llVert \int k(\cdot,x) \,\mathrm{d}(
\bb{P}_n-\bb{P}) (x)\biggr\rrVert _{\Cal{H}_k}
\nonumber
\\[-8pt]
\\[-8pt]
\nonumber
&&\quad\le 2
\bb{E}_\varepsilon\sup_{k\in\Cal{K}}\Biggl\llVert \frac{1}{n}
\sum^n_{i=1} \varepsilon_i k(
\cdot,X_i)\Biggr\rrVert _{\Cal{H}_k}+\sqrt{\frac{18\nu\tau}{n}}.
\end{eqnarray}
Note that
%
\begin{eqnarray}
\label{Eq:Mc-4} \bb{E}_\varepsilon\sup_{k\in\Cal{K}}\Biggl\llVert
\frac{1}{n}\sum^n_{i=1}
\varepsilon_i k(\cdot,X_i)\Biggr\rrVert
_{\Cal{H}_k}&\le& \frac{1}{n}\sqrt{\bb{E}_\varepsilon
\sup_{k\in\Cal{K}}\sum^n_{i,j=1}
\varepsilon_i\varepsilon_j k(X_i,X_j)}
\nonumber
\\[-8pt]
\\[-8pt]
\nonumber
&\le& \frac{\sqrt{2}}{n}\sqrt{U_n\bigl(
\Cal{K};(X_i)^n_{i=1}\bigr)} +
\frac{\sqrt{\nu}}{\sqrt{n}},
\end{eqnarray}
where
\[
U_n\bigl(\Cal{K};(X_i)^n_{i=1}
\bigr):=\bb{E}_\varepsilon\sup_{k\in\Cal{K}}\Biggl\llvert \sum
^n_
{i<j}\varepsilon_i
\varepsilon_j k(X_i,X_j)\Biggr\rrvert
\]
is the expected suprema of
the Rademacher chaos process of degree 2, indexed by $\Cal{K}$. The
proof until
this point already appeared in Sriperumbudur
\textit{et al.} (\cite{Sriperumbudur-09c}, see the proof of
Theorem~7), but we have presented here for completeness.

The
result in
(\ref{Eq:empirical-consistent}) therefore follows by using
(\ref{Eq:Mc-4}) in (\ref{Eq:Mc-3}) and bounding
$U_n(\Cal{K};(X_i)^n_{i=1})$ through Lemma~\ref{lem:chaining} with
$\theta=\frac{3}{4}$. Using
(\ref{Eq:entropynumber}) in (\ref{Eq:empirical-consistent}) and solving for
$\alpha$ yields (\ref{Eq:concentration-empirical}) and~(\ref{Eq:rates-empirical}).

\begin{rem}\label{rem:pn-p-remark}
(i) Note that instead of using McDiarmid's inequality
in the above proof, one can
directly obtain a version of (\ref{Eq:Mc-3}) by applying Talagrand's inequality
through Theorem~2.1 in Bartlett, Bousquet and
Mendelson \cite{Bartlett-05}, albeit with worse constants and
similar dependency on $n$.

(ii) If $\Cal{K}$ is singleton, that is, $\Cal{K}=\{k\}$, then l.h.s. of
(\ref{Eq:Mc-4}) can be bounded as
\begin{eqnarray*}
\bb{E}_\varepsilon\sup_{k\in\Cal{K}}\Biggl\llVert
\frac
{1}{n}\sum^n_{i=1}
\varepsilon_i k(\cdot,X_i)\Biggr\rrVert
_{\Cal{H}_k}&\le& \frac{1}{n}\sqrt{\bb{E}_\varepsilon
\sum^n_{i,j=1} \varepsilon_i
\varepsilon_j k(X_i,X_j)}
\\
&\le& \frac{1}{n}\sqrt{\bb{E}_\varepsilon\sum
^n_{i\ne j} \varepsilon_i
\varepsilon_j k(X_i,X_j)}+\frac{\sqrt{\nu}}{\sqrt{n}},
\end{eqnarray*}
and, therefore,
%
\begin{equation}
\bb{E}_\varepsilon\sup_{k\in\Cal{K}}\Biggl\llVert
\frac{1}{n}\sum^n_{i=1}
\varepsilon_i k(\cdot,X_i)\Biggr\rrVert _{\Cal{H}_k}
\le\frac{\sqrt{\nu}}{\sqrt{n}}. \label{Eq:Mc-5}
\end{equation}
\end{rem}

\subsection{Proof of Theorem \texorpdfstring{\protect\ref{thmm:examples}}{4.2}}\label
{subsec:pro-examples}
The proof involves showing that the kernels in (a)--(d) satisfy the
conditions (i)--(iv) in Theorem~\ref{Thm:estim-to-p}, thereby ensuring
that (\ref{Eq:final-talagrand}) and (\ref{Eq:final-talagrand-1}) hold. However,
instead of bounding $\Cal{T}$ through bounds on the covering numbers of
$\Cal{K}$,
we directly bound the
expected suprema of the Rademacher
chaos process indexed by $\Cal{K}$ and $\Cal{K}_\alpha$,
that is, $U_n(\Cal{K},(X_i)^n_{i=1})$ and
$U_n(\Cal{K}_\alpha,(X_i)^n_{i=1})$ which are defined in
(\ref{Eq:chaos})~-- note that the terms involving $\Cal{T}$ in
(\ref{Eq:final-talagrand}) and (\ref{Eq:final-talagrand-1}) are in fact bounds
on $U_n(\Cal{K},(X_i)^n_{i=1})$ and $U_n(\Cal{K}_\alpha,(X_i)^n_{i=1})$
-- and
show that $U_n(\Cal{K},(X_i)^n_{i=1})=\mathrm{O}_\bb{P}(n)$ and
$U_n(\Cal{K}_\alpha,(X_i)^n_{i=1})=\mathrm{O}_\bb{P}(n)$. Using these results in
(\ref{Eq:Mc-4}) and (\ref{Eq:bound-gauss-exm}) and following the proofs of
Theorems \ref{Thm:consistency} and \ref{Thm:estim-to-p}, we have $\Vert
\bb{P}_n-\bb{P}\Vert_{\Cal{F}_H}=\mathrm{O}_{\mathrm{a.s.}}(n^{-1/2})$,
\[
\Vert K_h\ast\bb{P}_n-\bb{P}_n
\Vert_{\Cal{F}_H}\le \frac{E_1h^r}{\sqrt{n}}+A_2h^{r}
\]
and
\[
\Vert K_h\ast\bb{P}_n-\bb{P}\Vert_{\Cal{F}_H}\le
\frac{F_1h^r}{\sqrt{n}}+A_2h^{r}+\frac{F_2}{\sqrt{n}},
\]
%
where $E_1$ and $(F_i)^2_{i=1}$ are constants that do not depend on
$n$ (we do not provide the explicit constants here but can be easily
worked out
by following the proofs of Theorems \ref{Thm:consistency}
and \ref{Thm:estim-to-p}). Therefore, the result
follows.

In the
following, we show that for $\Cal{K}$ in
(a)--(d), (iv) in Theorem~\ref{Thm:estim-to-p}
holds (note that (i)--(iii) in Theorem~\ref{Thm:estim-to-p} hold trivially
because of the choice of
$\Cal{K}$) along with $U_n(\Cal{K},(X_i)^n_{i=1})=\mathrm{O}_\bb{P}(n)$ and
$U_n(\Cal{K}_\alpha,(X_i)^n_{i=1})=\mathrm{O}_\bb{P}(n)$. In order to obtain
bounds on $U_n(\Cal{K},(X_i)^n_{i=1})$ and $U_n(\Cal{K}_\alpha,(X_i)^n_{i=1})$,
we need an intermediate result (see Proposition~\ref{pro:vc} below)~--
also of
independent interest~-- which is based on the notion of
\emph{pseudo-dimension} (Anthony and Bartlett \cite{Anthony-99}, Definition~11.1) of a function
class $\Cal{F}$. It has to be noted that the pseudo-dimension of $\Cal{F}$
matches with the VC-index of a VC-subgraph class, $\Cal{F}$
(Anthony and Bartlett \cite{Anthony-99}, Chapter~11, page~153).

\begin{fin}[(Pseudo-dimension)]
Let $\Cal{F}$ be a set of real valued functions on $\Cal{X}$ and
suppose that
$S=\{z_1,\ldots,z_n\}\subset\Cal{X}$. Then $S$ is pseudo-shattered
by
$\Cal{F}$ if there are real numbers
$r_1,\ldots,r_n$ such that
for any $b\in\{-1,1\}^n$ there is a function $f_b\in\Cal{F}$ with
$\operatorname{sign}(f_b(z_i)-r_i)=b_i$ for $i=1,\ldots, n$. The
pseudo-dimension or
VC-index of $\Cal{F}$, $\mathit{VC}(\Cal{F})$ is the maximum cardinality of $S$
that is
pseudo-shattered by $\Cal{F}$.
\end{fin}

\begin{pro}\label{pro:vc}
Let
\[
\Cal{F}= \Biggl\{f_\sigma(x,y)=\sigma^\theta\prod
^d_{i=1} \bigl(\sigma (x_i-y_i
)^2 \bigr)^{\delta_i}\mathrm{e}^{-\sigma(x_{i}-y_{
i})^2}, x,y\in\bb{R}^d
\dvt \sigma\in(0,\infty) \Biggr\},
\]
where $x:=(x_1,\ldots,x_d)\in\bb{R}^d$, $y:=(y_1,\ldots,y_d)\in\bb
{R}^d$, $\theta\ge0$ and
$\delta_i>0$ for
any $i\in\{1,\ldots,d\}$. Then $\mathit{VC}(\Cal{F})\le2$. If
$\theta=\delta_1=\cdots=\delta_d=0$, then $\mathit{VC}(\Cal{F})=1$.
\end{pro}

\begin{pf}
Suppose $\mathit{VC}(\Cal{F})> 2$. Then there exists a set
$S=\{(x_i,y_i)\in\bb{R}^d\times\bb{R}^d\dvt i\in\{1,2,3\}\}$ which is
pseudo-shattered by $\Cal{F}$, where $x_i=(x_{i1},\ldots,x_{id})\in\bb{R}^d$
and $y_i=(y_{i1},\ldots,y_{id})\in\bb{R}^d$. This implies
there exists $(r_1,r_2,r_3)\in\bb{R}^3$
such that for any $b\in\{-1,1\}^3$ there is a function $f_\sigma\in\Cal{F}$
with $\operatorname{sign}(f_\sigma(x_i,y_i)-r_i)=b_i$ for $i=1,2,3$. Without
loss of generality, let us assume
the points in $S$ satisfy
{\renewcommand{\theequation}{$\mathrm{P}_{213}$}
\begin{equation}\label{eqp123}
\Vert x_2-y_2\Vert_2
\le\Vert x_1-y_1\Vert_2 \le\Vert
x_3-y_3\Vert_2.
\end{equation}}
We now consider two cases.

\textit{Case} 1: $\Vert x_2-y_2\Vert_2 <\Vert x_1-y_1\Vert_2
< \Vert
x_3-y_3\Vert_2$:
 Let $b=(b_1,b_2,b_3)=(-1,1,1)$. Then there exists
$\sigma_1\in(0,\infty)$ such that the following hold:
\[
f_{\sigma_1}(x_1,y_1)<r_1,\qquad
f_{\sigma_1}(x_2,y_2)\ge r_2,\qquad
f_{\sigma_1}(x_3, y_3)\ge r_3.
\]
Similarly, for $b=(1,-1,-1)$, there exists $\sigma_2\in(0,\infty)$
such that
the
following hold:
\[
f_{\sigma_2}(x_1,y_1)\ge r_1,\qquad
f_{\sigma_2}(x_2,y_2)<r_2,\qquad
f_{\sigma_2}(x_3, y_3)<r_3.
\]
This implies
$f_{\sigma_2}(x_1,y_1)>f_{\sigma_1}(x_1,y_1), f_{\sigma_1}(x_2,y_2)>f_{
\sigma_2}(x_2,y_2), f_{\sigma_1}(x_3,y_3)>f_{\sigma_2}(x_3,y_3)$, that is,
\begin{eqnarray*}
\sigma^{\theta+\sum_i\delta_i}_2\prod^d_{i=1}
(x_{1i}-y_{1i} )^{2\delta_i}\mathrm{e}^{-\sigma_2\Vert
x_1 -y_1\Vert^2_2}& >&
\sigma^{\theta+\sum_i\delta_i}_1\prod^d_{i=1}
(x_{1i}-y_{1i} )^{2\delta_i}\mathrm{e}^{-\sigma_1\Vert
x_1 -y_1\Vert^2_2},\\
\sigma^{\theta+\sum_i\delta_i}_2\prod^d_{i=1}
(x_{2i}-y_{2i} )^{2\delta_i}\mathrm{e}^{-\sigma_2\Vert
x_2 -y_2\Vert^2_2} &<&
\sigma^{\theta+\sum_i\delta_i}_1\prod^d_{i=1}
(x_{2i}-y_{2i} )^{2\delta_i}\mathrm{e}^{-\sigma_1\Vert
x_2 -y_2\Vert^2_2},
\\
\sigma^{\theta+\sum_i\delta_i}_2\prod^d_{i=1}
(x_{3i}-y_{3i} )^{2\delta_i}\mathrm{e}^{-\sigma_2\Vert
x_3 -y_3\Vert^2_2}&<&
\sigma^{\theta+\sum_i\delta_i}_1\prod^d_{i=1}
(x_{3i}-y_{3i} )^{2\delta_i}\mathrm{e}^{-\sigma_1\Vert
x_3 -y_3\Vert^2_2}.
\end{eqnarray*}
It is clear that $x_{ji}-y_{ji}\ne0$ for any $i\in\{1,\ldots,d\}$ and
all $j\in\{1,2,3\}$ (otherwise leads to a contradiction). This implies
\begin{eqnarray*}
\mathrm{e}^{-(\sigma_1-\sigma_2)\Vert
x_2 -y_2\Vert^2_2}> \biggl(\frac{\sigma_2}{\sigma_1}
\biggr)^{\theta+\sum_i\delta_i}&>&\mathrm{e}^{
-(\sigma_1-\sigma_2)\Vert
x_1 -y_1\Vert^2_2},
\\
\mathrm{e}^{-(\sigma_1-\sigma_2)\Vert
x_3 -y_3\Vert^2_2}> \biggl(\frac{\sigma_2}{\sigma_1} \biggr)^{\theta+\sum_i\delta_i}&>&\mathrm{e}^{
-(\sigma_1-\sigma_2)\Vert
x_1 -y_1\Vert^2_2}
\end{eqnarray*}
and, therefore,
%
\setcounter{equation}{5}
\begin{equation}
(\sigma_1-\sigma_2) \bigl(\Vert x_2
-y_2\Vert^2_2-\Vert x_1
-y_1\Vert^2_2 \bigr)<0\label{Eq:vc-11}
\end{equation}
and
\[
(\sigma_1-\sigma_2) \bigl(\Vert x_3
-y_3\Vert^2_2-\Vert x_1
-y_1\Vert^2_2 \bigr)<0,
\]
which by our assumption $\Vert x_2-y_2\Vert_2 <\Vert x_1-y_1\Vert_2 <
\Vert
x_3-y_3\Vert_2$ implies $\sigma_1>\sigma_2$ and $\sigma_1<\sigma_2$ leading
to a
contradiction. Therefore, no 3-point set $S$ satisfying $\Vert
x_2-y_2\Vert_2 <\Vert x_1-y_1\Vert_2 < \Vert
x_3-y_3\Vert_2$ is pseudo-shattered by $\Cal{F}$.

\textit{Case} 2: \textit{At least one equality in} \eqref{eqp123} \textit{holds}:
Suppose $\Vert x_2-y_2\Vert_2 =\Vert x_1-y_1\Vert_2 < \Vert
x_3-y_3\Vert_2$. Then (\ref{Eq:vc-11}) yields a contradiction.
Similarly, a contradiction
arises if $\Vert x_2-y_2\Vert_2 <\Vert x_1-y_1\Vert_2 = \Vert
x_3-y_3\Vert_2$ or $\Vert x_2-y_2\Vert_2 =\Vert x_1-y_1\Vert_2 = \Vert
x_3-y_3\Vert_2$.

 Since every 3-point set $S$ satisfies \eqref{eqp123}, from cases
1 and 2, it follows that no 3-point set $S$ is pseudo-shattered by $\Cal
{F}$, which implies $\mathit{VC}(\Cal{F})\le2$.

If $\theta=\delta_i=0$ for all $i\in\{1,\ldots,d\}$, then
$\Cal{F}=\{f_\sigma(x,y)=\mathrm{e}^{-\sigma\Vert x-y\Vert^2_2}\dvt \sigma\in
(0,\infty)\}$.
Using
the same technique as above (also see the proof of Lemma~2 in Ying and
Campbell \cite{Ying-10}),
it can be shown that no two-point is shattered by
$\Cal{F}$ and, therefore, $\mathit{VC}(\Cal{F})=1$.
\end{pf}

\begin{pf*}{Proof of Theorem~\ref{thmm:examples}}
(a)
Consider
$\Cal{K}_\alpha:=\{\partial^{\alpha,\alpha}
\psi_\sigma(\cdot-\cdot)\dvt \sigma\in\Sigma\}$ for $|\alpha|=r$. It can be shown
that
\[
\partial^{\alpha,\alpha}\psi_\sigma(x-y)=\prod
^d_{i=1}(-1)^{\alpha
_i}\sigma^{
\alpha_i}
H_ { 2\alpha_i}\bigl(\sqrt{\sigma}(x_i-y_i)
\bigr)\mathrm{e}^{-\sigma(x_i-y_i)^2},
\]
where $x=(x_1,\ldots,x_d)\in\bb{R}^d$, $y=(y_1,\ldots,y_d)\in\bb{R}^d$ and
$H_l$ denotes the Hermite polynomial of degree $l$. By expanding
$H_{2\alpha_i}$ we obtain
%
\begin{eqnarray}
\label{Eq:partial-psi-1} \partial^{\alpha,\alpha}\psi_\sigma(x-y)&=&\sigma
^r\prod^d_{
i=1 } \sum
^ { \alpha_i } _ { j=0 } \eta_{ij} \bigl(
\sigma(x_i-y_i)^2 \bigr)^{j}\mathrm{e}^{-\sigma
(x_i-y_i)^2}
\nonumber
\\[-8pt]
\\[-8pt]
\nonumber
&=&\sum^{\alpha_1}_{j_1=0}\cdots\sum
^{\alpha_d}_{j_d=0}\prod^d_{i=1}
\eta_{
ij_i } \sigma^{\alpha_i+j_i} (x_i-y_i
)^{2j_i}\mathrm{e}^{-\sigma(x_i-y_i)^2},
\end{eqnarray}
where $\eta_{ij_i}$ are finite constants and $\eta_{i0}>0$ for all
$i=1,\ldots,d$. Therefore,
%
\begin{eqnarray}\label{Eq:derivative-1}
\sup_{\sigma\in\Sigma,x,y\in\bb{R}^d}\partial^{\alpha,\alpha} \psi_\sigma(x-y)&
\le&\sup_{\sigma\in\Sigma} \sigma^r \Biggl(\sum
^{\alpha}_{j=0}\prod^d_
{i=1}|
\eta_{ij_i}|j^{j_i}_i\mathrm{e}^{-j_i} \Biggr)
\nonumber
\\[-8pt]
\\[-8pt]
\nonumber
&=&a^r \Biggl(\sum^{
\alpha}_{j=0}
\prod^d_
{i=1}|\eta_{ij_i}|j^{j_i}_i\mathrm{e}^{-j_i}
\Biggr)<\infty, 
\end{eqnarray}
which implies (iv) in
Theorem~\ref{Thm:estim-to-p} is satisfied, where
$\sum^\alpha_{j=0}:=\sum^{\alpha_1}_{j_1=0}\cdots\sum^{\alpha
_d}_{j_d=0}$. Defining $\Cal{B}_2:=U_n(\Cal{K}_\alpha;(X_i)^n_{i=1})$,
we have
%
\begin{eqnarray}
\label{Eq:rademacher-gauss} \Cal{B}_2 &:=&\bb{E}_\varepsilon\sup
_{k^\prime\in
\Cal{K}_\alpha}\Biggl\llvert \sum^n_{i<j}
\varepsilon_i\varepsilon_jk^\prime(X_i,
X_j)\Biggr\rrvert
\nonumber
\\
&=&\bb{E}_\varepsilon\sup_{\sigma\in\Sigma
} \Biggl\llvert \sum
^n_ { p<
q}\varepsilon_p
\varepsilon_q\partial^{\alpha,\alpha}\psi_\sigma(X_p,
X_q)\Biggr\rrvert
\nonumber
\\
&\stackrel{(\ref{Eq:partial-psi-1})} {=} &\bb{E}_\varepsilon\sup
_
{\sigma\in\Sigma}\Biggl\llvert \sum^n_{p<
q}
\varepsilon_p\varepsilon_q\sum
^{\alpha_1}_{j_1=0}\cdots\sum^{\alpha_d}_{j_d=0}
\prod^d_{i=1}\eta_{ij_i}
\sigma^{\alpha_i+j_i} (X_{pi}-X_{qi} )^{2j_i}\mathrm{e}^{-\sigma(X_{pi}-X_{qi}
)^2 }
\Biggr\rrvert
\nonumber
\\
&=&\bb{E}_\varepsilon\sup_
{\sigma\in\Sigma}\Biggl\llvert \sum
^\alpha_{j=0} \Biggl(\prod
^d_{i=1}\eta_{ij_i} \Biggr) \sum
^n_{p<
q}\varepsilon_p
\varepsilon_q \prod^d_{i=1}
\sigma^{\alpha_i+j_i} (X_{pi}-X_{qi} )^{2j_i}\mathrm{e}^{-\sigma(X_{pi}-X_{qi}
)^2 }
\Biggr\rrvert
\nonumber
\\
&\le&\sum^{\alpha}_{j=0}
\Biggl(\prod^d_{i=1}\llvert
\eta_{ij_i}\rrvert \Biggr) \bb{E}_\varepsilon \sup
_
{\sigma\in\Sigma}\Biggl\llvert \sum^n_{p<
q}
\varepsilon_p\varepsilon_q \prod
^d_{i=1} \sigma^{\alpha_i+j_i}
(X_{pi}-X_{qi} )^{2j_i}\mathrm{e}^{-\sigma(X_{pi}-X_{qi}
)^2 }\Biggr
\rrvert
\nonumber
\\
&=&\sum^{\alpha_1}_{j_1=0}\cdots\sum
^{\alpha_d}_{j_d=0} \Biggl(\prod
^d_{i=1}\llvert \eta_{ij_i}\rrvert \Biggr)
\bb{E}_\varepsilon \sup_{k_{j_1\cdots j_d}\in\Cal{K}^{j_1\cdots j_d}_\alpha}\Biggl\llvert \sum
^n_{p<
q}\varepsilon_p
\varepsilon_qk_{j_1\cdots
j_d}(X_p,X_q)
\Biggr\rrvert ,
\end{eqnarray}
where
\[
\Cal{K}^{j_1\cdots
j_d}_\alpha:= \Biggl\{k_{j_1\cdots j_d}(x,y)=\prod
^d_{i=1} \sigma^{\alpha_i+j_i}
(x_{i}-y_{i} )^{2j_i}\mathrm{e}^{-\sigma(x_{i}-y_{i}
)^2}, x,y\in
\bb{R}^d \dvt \sigma\in\Sigma \Biggr\}.
\]
Since
\[
\sup_{k_{j_1,\ldots,j_d}\in\Cal{K}^{j_1\cdots
j_d}_\alpha,x,y\in\bb{R}^d}k_{j_1\cdots j_d}(x,y)\le a^r\mathrm{e}^{-\sum^d_{i=1}j_i}
\prod^d_{i=1}j^{j_i}_i:=
\zeta_{j_1\cdots
j_d}<\infty,
\]
by Lemma~\ref{lem:chaining}, we have
%
\begin{eqnarray}
\label{Eq:bound-gauss-exm-1}&&\bb{E}_\varepsilon \sup_{k_{j_1\cdots j_d}\in\Cal{K}^{j_1\cdots j_d}_\alpha}\Biggl\llvert
\sum^n_{p<
q}\varepsilon_p
\varepsilon_qk_{j_1\cdots j_d}(X_p,X_q)
\Biggr\rrvert
\nonumber
\\[-8pt]
\\[-8pt]
\nonumber
&&\quad
\le 2n^2 \Cal{T} \biggl(\Cal{K}^{j_1\cdots
j_d}_{\alpha},
\rho_{j_1\cdots j_d},\frac{\zeta_{j_1\cdots
j_d}}{2} \biggr) +\frac{n\zeta_{j_1\cdots
j_d}}{\sqrt{2}},
\end{eqnarray}
where $\Cal{T}$ and $\rho_{j_1\cdots j_d}$ (same as $\rho_\alpha$ but defined
on $\Cal{K}^{j_1\cdots j_d}_\alpha$) are defined in the statement of
Theorem~\ref{Thm:estim-to-p}. Since every element of $\Cal{K}^{j_1\cdots
j_d}_\alpha$ is nonnegative and bounded above by
$\zeta_{j_1\cdots j_d}$, we obtain the diameter of $\Cal{K}^{j_1\cdots
j_d}_\alpha$ to be bounded above by $\zeta_{j_1\cdots j_d}$ and,
therefore, we
used $\zeta_{j_1\cdots j_d}/2$ as an argument for $\Cal{T}$ in
(\ref{Eq:bound-gauss-exm-1}). Proposition~\ref{pro:vc} shows that
$\Cal{K}^{j_1\cdots j_d}_\alpha$ is a VC-subgraph with VC-index,
$V:=\mathit{VC}(\Cal{K}^{j_1\cdots j_d}_\alpha)\le2$ for any $0\le
j_i\le\alpha_i, i=1,\ldots,d$, which by Theorem~2.6.7 in van der Vaart
and Wellner \cite{Vaart-96} implies
that
%
\begin{eqnarray}\label{Eq:cover}
&&\Cal{N}\bigl(\Cal{K}^{j_1\cdots j_d}_\alpha,\rho_{j_1\cdots
j_d},
\epsilon\bigr)
\nonumber
\\[-8pt]
\\[-8pt]
\nonumber
&&\quad\le C^\prime V (16\mathrm{e})^{V} \biggl(
\frac{\zeta_{j_1\cdots
j_d}}{\epsilon} \biggr)^{2(V-1)},\qquad 0<\epsilon<\zeta_{j_1\cdots
j_d}
\end{eqnarray}
for some universal constant, $C^\prime$ and, therefore,
%
\begin{equation}
\Cal{T} \biggl(\Cal{K}^{j_1\cdots
j_d}_{\alpha},\rho_{j_1\cdots j_d},
\frac{\zeta_{j_1\cdots
j_d}}{2} \biggr)\le\frac{C^{\prime\prime}_{j_1\cdots
j_d}}{n},\label{Eq:covering-exm}
\end{equation}
where
$C^{\prime\prime}_{j_1\cdots j_d}$ is a constant that depends on
$C^\prime$, $V$
and $\zeta_{j_1\cdots
j_d}$. Combining (\ref{Eq:bound-gauss-exm-1}) and (\ref
{Eq:covering-exm}) in
(\ref{Eq:rademacher-gauss}), we obtain
\begin{eqnarray*}
&&U_n\bigl(\Cal{K}_\alpha;(X_i)^n_{i=1}
\bigr)\\
&&\quad\le n \sum^{\alpha_1}_{j_1=0}\cdots\sum
^{\alpha_d}_{j_d=0} \Biggl(\prod
^d_{i=1}\llvert \eta_{ij_i}\rrvert \Biggr)
\biggl(2C^{\prime\prime}_{
j_1\cdots j_d}+\frac{\zeta_{j_1\cdots
j_d}}{\sqrt{2}}
\biggr)\\
&&\quad =\mathrm{O}_{\bb{P}}(n).
\end{eqnarray*}
Also, since $\Cal{K}$ is a VC-subgraph with $\mathit{VC}(\Cal{K})=1$, from
(\ref{Eq:cover}) we obtain $\Cal{N}(\Cal{K},\rho,\epsilon)$ is a constant
independent of $\epsilon$. Following the
analysis as above, it is easy to show that
$U_n(\Cal{K};(X_i)^n_{i=1})=\mathrm{O}_\bb{P}(n)$.

(b) Since
$\partial^{\alpha,\alpha}\int^\infty_0
\psi_\sigma(x-y) \,\mathrm{d}\Lambda(\sigma)=\int^\infty_0
\partial^{\alpha,\alpha}\psi_\sigma(x-y) \,\mathrm{d}\Lambda(\sigma)$ holds by
Theorem~2.27(b) in
Folland \cite{Folland-99}, define
\[
\Cal{K}_\alpha:= \biggl\{\int^\infty_0
\partial^{\alpha,\alpha}\psi_\sigma(x-y) \,\mathrm{d}\Lambda(\sigma), x,y\in
\bb{R}^d \dvt \Lambda\in \Cal{M}_A \biggr\}.
\]
Therefore,
\begin{eqnarray*}
\sup_{k^\prime\in\Cal{K}_\alpha,x,y\in
\bb{R}^d}k'(x,y)&=&\sup
_{\Lambda\in\Cal{M}_A,
x,y\in\bb{R}^d}\int^\infty_0
\partial^{\alpha,\alpha}\psi_\sigma(x-y) \,\mathrm{d}\Lambda(\sigma)
\\
&\stackrel{(\ref{Eq:derivative-1})} {\le} & \Biggl(\sum
^{
\alpha}_{j=0}\prod^d_
{i=1}|
\eta_{ij_i}|j^{j_i}_i\mathrm{e}^{-j_i} \Biggr)\sup
_{\Lambda\in\Cal{M}_A} \int^\infty_0
\sigma^r \,\mathrm{d}\Lambda(\sigma)
\\
&=&A \Biggl(\sum^{
\alpha}_{j=0}\prod
^d_
{i=1}|\eta_{ij_i}|j^{j_i}_i\mathrm{e}^{-j_i}
\Biggr)<\infty,
\end{eqnarray*}
and so $\Cal{K}$ satisfies (iv) in
Theorem~\ref{Thm:estim-to-p}. Now consider
%
\begin{eqnarray}
\label{Eq:exm2-U} U_n\bigl(\Cal{K};(X_i)^n_{i=1}
\bigr) 
&=&
\bb{E}_\varepsilon\sup_{\Lambda\in\Cal{M}_A}\Biggl\llvert \sum
^n_{p<q} \varepsilon_p
\varepsilon_q\int^\infty_0
\psi_\sigma(X_p-X_q) \,\mathrm{d}\Lambda(\sigma)\Biggr
\rrvert
\nonumber
\\[-8pt]
\\[-8pt]
\nonumber
&\le&\bb{E}_\varepsilon\sup_{\sigma\in(0,\infty)}\Biggl\llvert \sum
^n_{p<q} \varepsilon_p
\varepsilon_q \psi_\sigma(X_p-X_q)
\Biggr\rrvert .
\end{eqnarray}
By Proposition~\ref{pro:vc}, since $\{\psi_\sigma(x-y)\dvt \sigma\in
(0,\infty)\}$
is a VC-subgraph, carrying out the analysis (following (\ref
{Eq:cover})) in
(a), we obtain $U_n(\Cal{K};(X_i)^n_{i=1})=\mathrm{O}_{\bb{P}}(n)$. Also,
\begin{eqnarray*}
U_n\bigl(\Cal{K}_\alpha;(X_i)^n_{i=1}
\bigr)&:=&\bb{E}_\varepsilon\sup_{k^\prime\in
\Cal{K}_\alpha}\Biggl\llvert
\sum^n_{p<q}\varepsilon_p
\varepsilon_q k^\prime(X_p,X_q)
\Biggr\rrvert
\\
&=&\bb{E}_\varepsilon\sup_{\Lambda\in\Cal{M}_A}\Biggl\llvert \sum
^n_{p<q} \varepsilon_p
\varepsilon_q\int^\infty_0
\partial^{\alpha,\alpha}\psi_\sigma(X_p-X_q) \,\mathrm{d}
\Lambda(\sigma)\Biggr\rrvert
\\
&\le& \bb{E}_\varepsilon\sup_{\Lambda\in\Cal{M}_A}\int
^\infty_0\Biggl\llvert \sum
^n_{p<q} \varepsilon_p
\varepsilon_q \partial^{\alpha,\alpha}\psi_\sigma(X_p-X_q)
\Biggr\rrvert \,\mathrm{d}\Lambda(\sigma)
\\
&\le& A \bb{E}_\varepsilon\sup
_{\sigma\in(0,\infty)}\Biggl\llvert \sum^n_{p<q}
\varepsilon_p\varepsilon_q\sigma^{-r}
\partial^{\alpha,\alpha}\psi_\sigma(X_p-X_q)
\Biggr\rrvert 
=:A U_n\bigl(
\Cal{L};(X_i)^n_{i=1}\bigr),
\end{eqnarray*}
where
\[
\Cal{L}:= \bigl\{\sigma^{-r} \partial^{\alpha,\alpha}
\psi_\sigma(x-y), x,y\in\bb{R}^d \dvt \sigma\in (0,\infty)
\bigr\}.
\]
Replicating the analysis in (\ref{Eq:rademacher-gauss})
for
$U_n(\Cal{L};(X_i)^n_{i=1})$ in conjunction with Proposition~\ref
{pro:vc}, it is
easy to show that
$U_n(\Cal{L};(X_i)^n_{i=1})=\mathrm{O}_{\bb{P}}(n)$
and, therefore, $U_n(\Cal{K}_\alpha;(X_i)^n_{ i=1})=\mathrm{O}_{\bb
{P}}(n)$.

(c) It is easy to check that any $k\in\Cal{K}$ is of
the form
\[
k(x,y)=\prod^d_{i=1}\int
^\infty_0 \mathrm{e}^{-\sigma_i(x_i-y_i)^2} \,\mathrm{d}
\Lambda_i(\sigma_i).
\]
Therefore,
\[
\Cal{K}_\alpha= \Biggl\{\prod^d_{i=1}
\int^\infty_0\partial^{\alpha
_i,\alpha_i}
\psi_{\sigma_i}(x_i-y_i) \,\mathrm{d}\Lambda_i(
\sigma_i), x,y\in\bb{R}^d \dvt \Lambda_i\in
\Cal{M}_{A_i}, i=1,\ldots,d \Biggr\}
\]
and
\begin{eqnarray*}
\sup_{k^\prime\in\Cal{K}_\alpha,x,y\in\bb{R}^d}k^\prime(x,y)&=&\prod
^d_{i=1} \sup_{\Lambda_i\in\Cal{M}_{A_i},x_i,y_i\in\bb{R}}\int
^\infty_0\partial ^{\alpha_i
,\alpha_i}
\psi_{\sigma_i}(x_i-y_i) \,\mathrm{d}\Lambda_i(
\sigma_i)
\\
&=&\prod^d_{i=1} A_i
\sum^{\alpha_i}_{j=0}|\eta_{ij}|j^j\mathrm{e}^{-j}
<\infty,
\end{eqnarray*}
which implies $\Cal{K}$ satisfies (iv) in
Theorem~\ref{Thm:estim-to-p}. Now consider
%
\begin{eqnarray*}
\label{Eq:exm3-U} U_n\bigl(\Cal{K};(X_i)^n_{i=1}
\bigr)&=& 
\bb{E}_\varepsilon\sup_{\Lambda\in\Cal{Q}_A}\Biggl\llvert \sum
^n_{p<q} \varepsilon_p
\varepsilon_q\int \mathrm{e}^{-(X_p-X_q)^T\Delta(X_p-X_q)} \,\mathrm{d}\Lambda(\Delta)\Biggr\rrvert
\nonumber
\\
&\le&\bb{E}_\varepsilon\sup_{\operatorname{diag}(\Delta)\in
(0,\infty)^d}\Biggl\llvert \sum
^n_{p<q} \varepsilon_p
\varepsilon_q \mathrm{e}^{-(X_p-X_q)^T\Delta
(X_p-X_q)}\Biggr\rrvert
\\
&=:& U_n\bigl(\Cal{J};(X_i)^n_{i=1}
\bigr),\nonumber
\end{eqnarray*}
where
\[
\Cal{J}:= \Biggl\{\tilde{k}(x,y)=\mathrm{e}^{-(x-y)^T\Delta
(x-y)}=\prod
^d_{i=1}\mathrm{e}^{-\sigma_i(x_i-y_i)^2}, x,y\in
\bb{R}^d:  \operatorname{diag}(\Delta)\in (0,\infty)^d \Biggr\}.
\]
Define
\[
\Cal{J}_i:= \bigl\{\tilde{k}^i(x,y)=\mathrm{e}^{-\sigma_i(x_i-y_i)^2},
x_i,y_i\in \bb{R} \dvt \sigma_i\in(0,\infty)
\bigr\}.
\]
It is easy to check that for any
$\tilde{k}_1,\tilde{k}_2\in\Cal{J}$, $\rho(\tilde{k}_1,\tilde{k}_2)\le
\sqrt{d}
\sum^d_{i=1}\rho(\tilde{k}^i_1,\tilde{k}^i_2)$, where
$\tilde{k}^i_1,\tilde{k}^i_2\in\Cal{J}_i$ and
$\Cal{N}(\Cal{J},\rho,\epsilon)=\prod^d_{i=1}\Cal{N}(\Cal{J}_i,\rho,
d^{-3/2}\epsilon)$. By Proposition~\ref{pro:vc}, since $\Cal{J}_i$ is a
VC-subgraph for any $i=1,\ldots,d$, from the analysis in (a), we
obtain $\Cal{N}(\Cal{J}_i,\rho,
\epsilon)=\mathrm{O}(1)$ and, therefore,
\[
U_n\bigl(\Cal{K};(X_i)^n_{i=1}
\bigr)\le U_n\bigl(\Cal{J};(X_i)^n_{i=1}
\bigr)=\mathrm{O}_{\bb{P}}(n).
\]
Similarly,
\begin{eqnarray*}
&&U_n\bigl(\Cal{K}_\alpha;(X_i)^n_{i=1}
\bigr)\\
&&\quad:=\bb{E}_\varepsilon\sup_{k^\prime\in
\Cal{K}_\alpha}\Biggl\llvert
\sum^n_{p<q}\varepsilon_p
\varepsilon_q k^\prime(X_p,X_q)
\Biggr\rrvert
\\
&&\quad=\bb{E}_\varepsilon\sup_{\Lambda_i\in\Cal{M}_{A_i}, i\in
[d]}\Biggl\llvert \sum
^n_{p<q} \varepsilon_p
\varepsilon_q\prod^d_{i=1}
\int^\infty_0 \partial^{\alpha_i,\alpha_i}
\psi_{\sigma_i}(X_{pi}-X_{qi}) \,\mathrm{d}\Lambda_i(
\sigma_i)\Biggr\rrvert
\\
&&\quad=\bb{E}_\varepsilon\sup_{\Lambda_i\in\Cal{M}_{A_i}, i\in
[d]}\Biggl\llvert \int
^\infty_0\cdots\int^\infty_0
\sum^n_{p<q} \varepsilon_p
\varepsilon_q\prod^d_{i=1}
\partial^{\alpha_i,\alpha_i}\psi_{
\sigma_i}(X_{pi}-X_{qi})
\prod^d_{i=1} \,\mathrm{d}\Lambda_i(
\sigma_i)\Biggr\rrvert
\\
&&\quad\le \Biggl(\prod^d_{i=1}A_i
\Biggr)\bb{E}_\varepsilon\sup_{\operatorname{diag}(\Delta)\in(0,\infty)^d}\Biggl\llvert \sum
^n_{p<q} \varepsilon_p
\varepsilon_q\prod^d_{i=1}
\sigma^{-\alpha_i}_i\partial ^{\alpha_i,
\alpha_i }
\psi_{\sigma_i}(X_{pi}-X_{qi})\Biggr\rrvert
\\
&&\quad=: \Biggl(\prod^d_{i=1}A_i
\Biggr)U_n\bigl(\Cal {I};(X_i)^n_{i=1}
\bigr),
%
%
\end{eqnarray*}
where $[d]:=\{1,\ldots,d\}$ and
\begin{eqnarray*}
&&\Cal{I}:= \Biggl\{\check{k}(x,y)=\prod^d_{i=1}
\sigma^{-\alpha_i}_i\partial^{
\alpha_i,
\alpha_i }
\psi_{\sigma_i}(x_{i}-y_{i}),\\
&&\hspace*{28pt} x,y\in
\bb{R}^d \dvt (\sigma_1,\ldots, \sigma_d)
\in(0,\infty)^d \Biggr\}.
\end{eqnarray*}
We now proceed as above to obtain a bound on
$U_n(\Cal{I};(X_i)^n_{i=1})$ through $\Cal{N}(\Cal{I},\rho,\epsilon)$ by
defining
\[
\Cal{I}_i:= \bigl\{\check{k}^i(x,y)=
\sigma^{-\alpha_i}_i\partial^{
\alpha_i,\alpha_i}\psi_{\sigma_i}(x_{i}-y_{i}),
x_i,y_i\in\bb{R} \dvt \sigma_i\in(0,
\infty) \bigr\}
\]
and noting that for any $\check{k}_1,\check{k}_2\in\Cal{I}$, we have
$\rho(\check{k}_1,\check{k}_2)\le Bd^{3/2}\rho(\check{k}^i_1,\check{k}^i_2)$
where $\check{k}^i_1,\check{k}_2^i\in\Cal{I}_i$, $B:=\max_{i\in
\{1,\ldots,d\}}\sum^{\alpha_i}_{j=0}|\eta_{ij}|j^j\mathrm{e}^{-j}$ and
$\Cal{N}(\Cal{I},\rho,\epsilon)=\prod^d_{i=1}\Cal{N}(\Cal{I}_i,\rho,B^{-1}d^{
-3/2}\epsilon)$. Proceeding with the covering number analysis in
(a), it can be shown that $\Cal{I}_i$ is a VC-subgraph with
$\mathit{VC}(\Cal{I}_i)\le2$ for any $i=1,\ldots,d$ and, therefore,
$\Cal{N}(\Cal{I},\rho,\epsilon)=\mathrm{O}(\epsilon^{-2})$, which means
\begin{eqnarray*}
U_n\bigl(\Cal{K}_\alpha;(X_i)^n_{i=1}
\bigr)&\le& \Biggl(\prod^d_{i=1}A_i
\Biggr) U_n\bigl(\Cal{I};(X_i)^n_{i=1}
\bigr)\\
&=&\mathrm{O}_{\bb{P}}(n).
\end{eqnarray*}

(d) First we derive an alternate form for $k\in\Cal{K}$ which
will be useful to
prove the result. To this end, by Theorem~6.13 in Wendland \cite
{Wendland-05}, any
$k\in\Cal{K}$ can be written as the Fourier transform of
$\frac{Ac^{2\beta-d}\Gamma(\beta)}{2^{1-\beta}}(c^2+\Vert
\omega\Vert^2_2)^{-\beta}$, that is, for any $c>0$,
%
\begin{eqnarray}
\label{Eq:ft} k(x,y)&=&A\frac{\llVert
x-y\rrVert ^{\beta-{d}/{2}}_2}{c^{{d}/{2}-\beta}}\frak {K}_{{d}/{
2}-\beta} \bigl(c\Vert
x-y\Vert_2 \bigr)
\nonumber
\\[-8pt]
\\[-8pt]
\nonumber
&=&
\frac{Ac^{2\beta-d}\Gamma(\beta)}{(2\pi)^{d/2}2^{1-\beta}} \int_ { \bb{ R } ^d } \mathrm{e}^{-\sqrt{-1}(x-y)^T\omega}
\bigl(c^2+\Vert\omega\Vert^2_2
\bigr)^{-\beta} \,\mathrm{d}\omega.
\end{eqnarray}
By the Sch\"{o}nberg representation for radial
kernels (see (\ref{Eq:schoenberg})), it follows from Wendland \cite{Wendland-05}, Theorem~7.15, that
%
\begin{equation}
\bigl(c^2+\Vert \omega\Vert^2_2
\bigr)^{-\beta}=\frac{1}{\Gamma(\beta)}\int^\infty_0
\mathrm{e}^{-t\Vert\omega\Vert^2_2}t^{\beta-1}\mathrm{e}^{-c^2t} \,\mathrm{d}t.\label{Eq:sch}
\end{equation}
Combining (\ref{Eq:ft}) and (\ref{Eq:sch}), we have
\[
k(x,y)=\frac{Ac^{2\beta-d}}{(2\pi)^{d/2}2^{1-\beta}}\int_{\bb{R}^d} \mathrm{e}^{-\sqrt{-1}(x-y)^T\omega} \int
^\infty_0 \mathrm{e}^{-t\Vert\omega\Vert^2_2}t^{\beta-1}\mathrm{e}^{-c^2t}
\,\mathrm{d}t \,\mathrm{d}\omega,
\]
which after applying Fubini's theorem yields
%
\begin{eqnarray}
\label{Eq:matern-equiv} k(x,y)&=&\frac{Ac^{2\beta-d}}{(2\pi)^{d/2}2^{1-\beta}} \int^\infty_0
\int_{\bb{R}^d} \mathrm{e}^{-\sqrt{-1}(x-y)^T\omega}
\mathrm{e}^{-t\Vert\omega\Vert^2_2} \,\mathrm{d}\omega\,
t^{\beta-1}\mathrm{e}^{-c^2t} \,\mathrm{d}t
\nonumber
\\[-8pt]
\\[-8pt]
\nonumber
&=&\frac{c^{2\beta-d}}{\Gamma(\beta-{d}/{2})} \int^\infty_0
\mathrm{e}^{-{(\Vert x-y\Vert^2_2)}/{(4t)}}t^{\beta-1-{d}/{2}}\mathrm{e}^{-c^2t} \,\mathrm{d}t.
\end{eqnarray}
Note that
\[
\sup_{k\in\Cal{K},x,y\in\Cal{X}}k(x,y)\le \sup_{c\in(0,a]}
\frac{c^{2\beta-d}}{\Gamma(\beta-{d}/{2})} \int^\infty_0
t^{\beta-1-{d}/{2}}\mathrm{e}^{-c^2t} \,\mathrm{d}t=1,
\]
implying that $\Cal{K}$ satisfies (iii) in
Theorem~\ref{Thm:estim-to-p}. Using (\ref{Eq:matern-equiv})
in $U_n(\Cal{K};(X_i)^n_{i=1})$, we have
\begin{eqnarray*}
&&U_n\bigl(\Cal{K};(X_i)^n_{i=1}
\bigr)\\
&&\quad=\bb{E}_\varepsilon\sup_{k\in\Cal{K}} \Biggl\llvert
\sum^n_
{i<j}\varepsilon_i
\varepsilon_j k(X_i,X_j)\Biggr\rrvert
\nonumber
\\
&&\quad=\bb{E}_\varepsilon\sup_{
c\in
(0,a]}\frac{c^{2\beta-d}}{\Gamma(\beta-{d}/{2})}
\Biggl\llvert \sum^n_
{i<j}
\varepsilon_i\varepsilon_j \int^\infty_0
\mathrm{e}^{-{(\Vert X_i-X_j\Vert^2_2)}/{(4t)}}
t^{\beta-1-{d}/{2}}\mathrm{e}^{-c^2t} \,\mathrm{d}t\Biggr\rrvert
\nonumber
\\
&&\quad\le\bb{E}_\varepsilon\sup_{
t\in(0,\infty)
} \Biggl\llvert \sum
^n_
{i<j}\varepsilon_i
\varepsilon_j\mathrm{e}^{-{(\Vert X_i-X_j\Vert^2_2)}/{(4t)}}\Biggr\rrvert \sup_{c\in
(0,a]}
\frac{c^{2\beta-d}}{\Gamma(\beta-{d}/{2})}\biggl\llvert \int^\infty_0
t^{\beta-1-{d}/{2}}\mathrm{e}^{-c^2t} \,\mathrm{d}t\biggr\rrvert
\nonumber
\\
&&\quad=\bb{E}_\varepsilon\sup_{\sigma\in(0,\infty)} \Biggl\llvert \sum
^n_
{i<j}\varepsilon_i
\varepsilon_j\mathrm{e}^{-\sigma\Vert
X_i-X_j\Vert^2_2}\Biggr\rrvert ,
\nonumber
\end{eqnarray*}
and, therefore, it follows (see Remark~\ref{Rem:supp}(i))
that $U_n(\Cal{K};(X_i)^n_{i=1})=\mathrm{O}_\bb{P}(n)$. Now for $|\alpha
|=m\wedge r$, let
us consider
%
\begin{eqnarray}
\label{Eq:derivative}k^\prime(x,y)&:=&\partial^{\alpha,\alpha
}k(x,y)=
\frac{c^{
2\beta-d}}{\Gamma(\beta-{d}/{2})} \int^\infty_0 \bigl(
\partial^{\alpha,\alpha}\mathrm{e}^{-{(\Vert x-y\Vert^2_2)}/{(4t)}}
\bigr)t^{\beta-1-{d}/{2}}\mathrm{e}^{-c^2t}
\,\mathrm{d}t\nonumber
\\
&=&\frac{c^{2\beta-d}}{\Gamma(\beta-{d}/{2})} \int^\infty_0
\bigl((4t)^{m\wedge r}\partial^{\alpha,\alpha} \mathrm{e}^{-{(\Vert x-y\Vert^2_2)}/{(4t)}}
\bigr)t^{\beta-1-{d}/{2}}(4t)^{-(m\wedge
r)}\mathrm{e}^{-c^2t} \,\mathrm{d}t\qquad\nonumber
\\
&=&\frac{c^{2\beta-d}}{\Gamma(\beta-{d}/{2})4^{m\wedge r}} \int^\infty_0
\bigl((4t)^{m\wedge r}\partial^{\alpha,\alpha}
\mathrm{e}^{-{(\Vert x-y\Vert^2_2)}/{(4t)}}
\bigr)t^{\beta-1-{d}/{2}-(m\wedge r)}\mathrm{e}^{-c^2t} \,\mathrm{d}t,\hspace*{30pt}
\end{eqnarray}
where the equality in the first line follows from
Folland \cite{Folland-99}, Theorem~2.27(b). The above implies
\begin{eqnarray*}
\sup_{k^\prime\in\Cal{K}_\alpha,x,y\in\Cal{X}}k^\prime(x,y)&\le& \sup_{\sigma\in(0,\infty),x,y\in\Cal{X}}
\bigl\llvert \sigma^{-(m\wedge
r)}\partial^{\alpha,\alpha} \mathrm{e}^{-\sigma\Vert x-y\Vert^2_2
}\bigr
\rrvert \frac{\Gamma(\beta-{d}/{2}-m\wedge
r)a^{2(m\wedge r)}}{\Gamma(\beta-{d}/{2})4^{m\wedge r}}\\
&<&\infty,
\end{eqnarray*}
therefore satisfying (iv) in Theorem~\ref{Thm:estim-to-p}. Using
(\ref{Eq:derivative}) we now obtain a bound on
$\Cal{B}_3:=U_n(\Cal{K}_\alpha;(X_i)^n_{i=1})$ as follows by defining
$B:=\Gamma(\beta-\frac{d}{2})4^{m\wedge r}$.
\begin{eqnarray*}
\Cal{B}_3&=&\bb{E}_\varepsilon\sup_{k^\prime\in\Cal
{ K }_\alpha}
\Biggl\llvert \sum^n_
{i<j}
\varepsilon_i\varepsilon_j k^\prime(X_i,X_j)
\Biggr\rrvert
\nonumber
\\
&=&\bb{E}_\varepsilon\sup_{c\in
(0,a]}\frac{c^{
2\beta-d}}{B}
\Biggl\llvert \sum^n_
{i<j}
\varepsilon_i\varepsilon_j \int^\infty_0
\bigl((4t)^{m\wedge
r}\partial^{\alpha,\alpha} \mathrm{e}^{-{(\Vert X_i-X_j\Vert^2_2)}/{(4t)}}
\bigr)t^{\beta-1-{d}/{2}-(m\wedge r)}\mathrm{e}^{-c^2t} \,\mathrm{d}t\Biggr\rrvert
\nonumber
\\
&\le& \bb{E}_\varepsilon\sup_{t\in(0,\infty)
} \Biggl\llvert \sum
^n_
{i<j}\varepsilon_i
\varepsilon_j(4t)^{m\wedge r}\partial^{\alpha,\alpha}
\mathrm{e}^{-{(\Vert X_i-X_j\Vert^2_2)}/{(4t)}}\Biggr\rrvert\\
&&{}\times  \sup_{c\in
(0,a]}\frac{c^{
2\beta-d}}{B}
\int^\infty_0 t^{\beta-1-{d}/{2}-(m\wedge r)}\mathrm{e}^{-c^2t}
\,\mathrm{d}t
\nonumber
\\
&\le&\frac{\Gamma(\beta-{d}/{2}-m\wedge
r)a^{2(m\wedge r)}}{\Gamma(\beta-{d}/{2})4^{m\wedge
r}} \bb{E}_\varepsilon\sup_{\sigma\in(0,\infty)}
\Biggl\llvert \sum^n_
{i<j}
\varepsilon_i\varepsilon_j\sigma^{-(m\wedge r)}
\partial^{\alpha
,\alpha} \mathrm{e}^{-\sigma\Vert
X_i-X_j\Vert^2_2}\Biggr\rrvert ,
\nonumber
\end{eqnarray*}
and so $U_n(\Cal{K}_\alpha;(X_i)^n_{i=1})=\mathrm{O}_\bb{P}(n)$, which follows
from the proof of Theorem~\ref{thmm:examples}(ii).
\end{pf*}


%
\begin{rem}
Note that instead of following the indirect route~-- showing $\Cal
{K}^{j_1\cdots
j_d}_\alpha$ to be a VC-subgraph and then bounding
$U_n(\Cal{K}_\alpha;(X_i)^n_{i=1})$~-- of showing
the result in
Theorem~\ref{Thm:estim-to-p} for the Gaussian kernel family as
presented in
(a), one can directly get the result by
obtaining a bound on
$\Cal{N}(\Cal{K}_\alpha,\rho_\alpha,\epsilon)$ as presented in
Proposition~\ref{Pro:covering}, under the assumption that $\Cal{X}=(a_0,b_0)^d$
for some $-\infty<a_0<b_0<\infty$. The advantage with the analysis in
(a) is
that the result holds for
$\Cal{X}=\bb{R}^d$ rather than
a bounded subset of $\bb{R}^d$. Also the proof technique in (a) is
useful and
interesting as it avoids the difficult problem of bounding the covering numbers
of $\Cal{K}$ and $\Cal{K}_\alpha$ for kernel classes in (b)--(d) while
allowing to handle these classes easily through (a).
\end{rem}

\subsection{Proof of the claim in
Remark \texorpdfstring{\protect\ref{rem:gauss-exm}}{4.3}(iii)}\label{sec:proof-weak}
We show that $\Cal{K}$ in
(a)--(c) satisfy the conditions in Theorem~\ref{Thm:weak} and, therefore,
metrize the weak topology on $M^1_+(\bb{R}^d)$. Note that the families in
(a)--(c) are uniformly bounded and every $k\in\Cal{K}$ is such that
$k(\cdot,x)\in C_0(\bb{R}^d)$ for all $x\in\bb{R}^d$. It therefore
remains to
check (\ref{Eq:dense}) and $(\mathrm{P})$ in Theorem~\ref{Thm:weak}. By
Proposition~5 in
Sriperumbudur, Fukumizu and
Lanckriet \cite{Sriperumbudur-11a} (see (17) in its proof), it is clear
that (\ref{Eq:dense}) is satisfied for $\Cal{K}$ in (a) and (b).
For (c),
\begin{eqnarray}\label{Eq:spd}
B&:=&\int_{\bb{R}^d}\int_{\bb{R}^d} k(x,y) \,\mathrm{d}
\mu(x) \,\mathrm{d}\mu(y)\nonumber\\
& =&\int_{\bb{R}^d}\int_{\bb{R}^d}\prod
^d_{j=1} \int^\infty_0
\mathrm{e}^{-\sigma(x_j-y_j)^2} \,\mathrm{d}\Lambda_j(\sigma)
 \,\mathrm{d}\mu(x) \,\mathrm{d}\mu(y)
\nonumber
\\
&=&\int_{\bb{R}^d}\int_{\bb{R}^d}\prod
^d_{j=1} \int^\infty_0
\frac{1}{(4\pi\sigma)^{d/2}}\int_{\bb{R}}\mathrm{e}^{-\sqrt{-1}\omega_j(x_j-y_j)}
\mathrm{e}^{-{\omega^2_j}/{(4\sigma)}} \,\mathrm{d}\omega_j \,\mathrm{d}\Lambda_j(\sigma) \,\mathrm{d}\mu(x) \,\mathrm{d}
\mu(y)\qquad\quad\nonumber
\\
\nonumber
&=&
\int_{\bb{R}^d}\int_{\bb{R}^d}\int
_{\bb{R}^d}\mathrm{e}^{-\sqrt{-1}\omega^T(x-y)} \prod^d_{ j=1}
\int^\infty_0 \frac{1}{(4\pi\sigma)^{d/2}}\mathrm{e}^{-{\omega^2_j}/{(4\sigma)}}
\,\mathrm{d}\Lambda _j(\sigma) \,\mathrm{d}\omega \,\mathrm{d}\mu(x) \,\mathrm{d}\mu(y)
\\
&=&\int_{\bb{R}^d}\bigl|\widehat{\mu}(\omega)\bigr|^2
\Biggl(\prod^d_{
j=1}\int
^\infty_0 \frac{1}{(4\pi\sigma)^{d/2}}\mathrm{e}^{-{\omega^2_j}/{(4\sigma)}} \,\mathrm{d}
\Lambda_j(\sigma) \Biggr) \,\mathrm{d}\omega,
\end{eqnarray}
where we have invoked Fubini's theorem in the last two lines of
(\ref{Eq:spd}) and $\widehat{\mu}$ denotes the Fourier transform of $\mu
$. Since
$\operatorname{supp}(\Lambda_j)\ne\{0\}$ for all $j=1,\ldots,d$, the inner
integrals in
(\ref{Eq:spd}) are positive for every $\omega_j\in\bb{R}$ and so
(\ref{Eq:dense}) holds.

We now show that $(\mathrm{P})$ in Theorem~\ref{Thm:weak} is satisfied by
$\Cal{K}$
in (a)--(c). Consider $\Cal{K}$ in (b). Fix $x\in\bb{R}^d$ and
$\epsilon>0$. Define $U_{x,\epsilon}=\{y\in\bb{R}^d\dvt \Vert
x-y\Vert_2<(4\delta
\log\frac{2B}{2-\epsilon^2})^{1/4}\} $, where
$\delta$ and $B$ are as mentioned in the statement of
Theorem~\ref{thmm:examples}. Then for any $k\in\Cal{K}$ and $y\in
U_{x,\epsilon}$,
\begin{eqnarray*}
\bigl\Vert k(\cdot,x)-k(\cdot,y)\bigr\Vert^2_{\Cal{H}_k}&=& 2-2\int
^\infty_0 \mathrm{e}^{-\sigma\Vert x-y\Vert^2_2} \,\mathrm{d}\Lambda(\sigma)
\\
&\le&2-2 \biggl(\inf_{\Lambda\in\Cal{M}_A}\int^\infty_0
\mathrm{e}^{-\delta\sigma^2} \,\mathrm{d}\Lambda(\sigma) \biggr) \Bigl(\inf_{\sigma\in
(0,\infty)}\mathrm{e}^{
-\sigma\Vert x-y\Vert^2_2}\mathrm{e}^{\delta\sigma^2}
\Bigr)
\\
&\le&2-2B\mathrm{e}^{-{\Vert x-y\Vert^4_2}/{(4\delta)}}<\epsilon^2.
\end{eqnarray*}
For $\Cal{K}$ in (c), define $U_{x,\epsilon}:= \{y\in\bb
{R}^d\dvt \Vert
x-y\Vert_\infty< (4\min_i\delta_i\log\frac{2\prod^d_{i=1}B_i}{2-\epsilon^2}
 )^{1/4}  \}$ for some fixed $x\in\bb{R}^d$ and $\epsilon>0$.
Then as above, it
is easy to show that for any $k\in\Cal{K}$ and $y\in U_{x,\epsilon}$,
\begin{eqnarray*}
\bigl\Vert k(\cdot,x)-k(\cdot,y)\bigr\Vert^2_{\Cal{H}_k}&=& 2-2\prod
^d_{i=1}\int^\infty_0
\mathrm{e}^{-\sigma(x_i-y_i)^2} \,\mathrm{d}\Lambda_i(\sigma)
\\
&\le&2-2\prod^d_{i=1} \biggl(\inf
_{\Lambda_i\in\Cal{M}_{A_i}}\int^\infty_0
\mathrm{e}^{-\delta_i\sigma^2} \,\mathrm{d}\Lambda_i(\sigma) \biggr) \Bigl(\inf
_{\sigma\in
(0,\infty)} \mathrm{e}^ {
-\sigma( x_i-y_i)^2}\mathrm{e}^{\delta_i\sigma^2} \Bigr)
\\
&\le&2-2\prod^d_{i=1}B_i\mathrm{e}^{-{(x_i-y_i)^4}/{(4\delta_i)}}
\le 2-2\prod^d_{i=1}B_i\mathrm{e}^{-{\Vert x-y\Vert^4_\infty}/{(4\min_i\delta_i)}}<
\epsilon^2,
\end{eqnarray*}
thereby proving the result.

\subsection{Proof of Theorem \protect\ref{pro:singleton}}\label
{subsec:thmm-singleton}
In the following, we prove that the class $\Cal{F}_H$ induced by the family
$\Cal{K}$ in (a)--(d) are Donsker and, therefore, the result simply follows
from Theorem~\ref{thmm:examples}. To this end, we first prove that $\Cal{K}$
in (d) is Donsker which will be helpful to prove the claim for the kernel
classes in (a)--(c).

(d) Since $k$ is continuous and bounded and $\Cal{X}$ is
separable, by Steinwart and Christmann~\cite{Steinwart-08}, Lemma~4.33, the RKHS $\Cal{H}_k$ induced by $k$ is
separable and every $f\in\Cal{H}_k$ is also
continuous and bounded. In addition, the
inclusion $\mathrm{id}\dvtx\Cal{H}_k\rightarrow
C_b(\Cal{X})$ is linear and continuous (Steinwart and
Christmann \cite{Steinwart-08},
Lemma~4.28).
Therefore, by Marcus (\cite{Marcus-85}, Theorem~1.1), $\Cal{F}_H=\{f\in
\Cal{H}_k\dvt \Vert
f\Vert_{\Cal{H}_k}\le1\}$ is $\bb{P}$-Donsker, that is,
$\sqrt{n} (\bb{P}_n-\bb{P})\leadsto_{\ell^\infty(\Cal{F}_H)}\bb{G}_\bb{P}$.
Also, $\sqrt{n} (\bb{P}_n\ast K_h-\bb{P})=\sqrt{n} (\bb{P}_n\ast
K_h-\bb{P}_n)+\sqrt{n} (\bb{P}_n-\bb{P})\leadsto_{\ell^\infty(\Cal
{F}_H)}\bb{G}
_\bb{P}$ by Slutsky's lemma and Theorem~\ref{Thm:estim-to-p}.

(a)--(c) From (\ref{Eq:fh-subset}), we have
\[
\Cal{F}_H\subset\bigcup_{\sigma\in
[a,b]} \biggl
\{f\in\Cal{H}_b\dvt \Vert f\Vert_{\Cal{H}_b}\le \biggl(
\frac{b}{\sigma} \biggr)^{d/4} \biggr\}= \biggl\{f\in
\Cal{H}_b\dvt \Vert f\Vert_{\Cal{H}_b}\le \biggl(
\frac{b}{a} \biggr)^{d/4} \biggr\}=:\Cal{B}.
\]
Using the argument as in (d), it is easy to verify that
$\Cal{H}_b$ is separable and $\mathrm{id}\dvtx\Cal{H}_b\mapsto C_b(\Cal
{X})$ is linear
and continuous and, therefore, $\Cal{B}$ is $\bb{P}$-Donsker, which implies
$\Cal{F}_H$ is Donsker by van der Vaart and Wellner \cite{Vaart-96}, Theorem~2.10.1. The result therefore
follows using Slutsky's lemma and Theorem~\ref{thmm:examples}. The
proof of
(c) is similar to that of in (a) but we use
(\ref{Eq:fh-subset-1}) instead of (\ref{Eq:fh-subset}). For~(b), the result hinges on a relation similar to those
in (\ref{Eq:fh-subset}) and (\ref{Eq:fh-subset-1}), which we derive below.
Let $\Cal{K}$ be the kernel family as shown in (\ref
{Eq:multiquadrics}). Then
for $k\in\Cal{K}$, let $\Cal{H}_c$ be the induced RKHS. From
Wendland \cite{Wendland-05}, Theorems 6.13 and 10.12, it follows that
for any
$f\in\Cal{H}_c$,
\[
\Vert f\Vert^2_{\Cal{H}_c}=\frac{\Gamma(\beta)}{2^{1-\beta}}\int \bigl|\widehat{f}(
\omega)\bigr|^2 \frac{c^{-d}}{(c\Vert
\omega\Vert_2)^{\beta-{d}/{2}}\fr{K}_{{d}/{2}-\beta}
(c\Vert\omega\Vert_2)} \,\mathrm{d}\omega.
\]
By Wendland \cite{Wendland-05}, Corollary~5.12, since
for every $\nu\in\bb{R}$, $x\mapsto x^\nu\fr{K}_{-\nu}(x)$ is
nonincreasing on
$(0,\infty)$, we have that for any $0<\tau<c<\infty$,
\[
\Vert f\Vert_{\Cal{H}_\tau}\le \biggl(\frac{c}{\tau} \biggr)^{{d}/{2}}
\Vert f\Vert_{\Cal{H}_c}
\]
and so $\Cal{H}_c\subset\Cal{H}_\tau$. Therefore, we have
\[
\Cal{F}_H\subset\bigcup_{c\in[a,\infty)} \biggl
\{f\in\Cal{H}_a\dvt \Vert f\Vert_{\Cal{H}_a}\le \biggl(
\frac{c}{a} \biggr)^{d/2} \biggr\}=\Cal{H}_a.
\]
For
the choice of $\Cal{K}$ in Theorem~\ref{pro:singleton}(b), we have
%
\begin{equation}
\Cal{F}_H\subset\bigcup_{c\in[a,b]} \biggl
\{f\in\Cal {H}_a\dvt \Vert f\Vert_{\Cal{H}_a}\le \biggl(
\frac{c}{a} \biggr)^{d/2} \biggr\}= \biggl\{f\in \Cal{H}
_a\dvt \Vert f\Vert_{\Cal{H}_a}\le \biggl(\frac{b}{a}
\biggr)^{d/2} \biggr\}\label{Eq:fh-subset2}
\end{equation}
and the rest follows.

\subsection{Proof of Proposition
\texorpdfstring{\protect\ref{pro:bound}}{5.1}}\label{subsec:probound}
By definition,
\begin{eqnarray*}
\Vert \bb{P}-\bb{Q}\Vert_{\Cal{K}_\Cal{X}}&=&\sup_{k\in\Cal{K}, x\in\Cal{X}}
\biggl\llvert \int k(x,y) \,\mathrm{d}(\bb{P}-\bb{Q}) (y)\biggr\rrvert
\\
&=&\sup_{k\in\Cal{K}, x\in\Cal{X}} \biggl\llvert \int\bigl\langle k(
\cdot,x),k(\cdot,y)\bigr\rangle_{\Cal{H}_k} \,\mathrm{d}(\bb{P}-\bb{Q}) (y)\biggr\rrvert.
\end{eqnarray*}
Since $\Cal{K}$ is uniformly bounded, $k(\cdot,x)$ is Bochner-integrable
for all $k\in\Cal{K}$ and $x\in\Cal{X}$, that is,
\[
\int \bigl\Vert k(\cdot,x)\bigr\Vert_{\Cal{H}_k} \,\mathrm{d}\bb{P}(x)=\int\sqrt{k(x,x)} \,\mathrm{d}\bb
{P}(x)\le \sqrt{\nu}\qquad \forall k\in\Cal{K}, x\in\Cal{X},
\]
and, therefore,
\begin{eqnarray*}
\Vert \bb{P}-\bb{Q}\Vert_{\Cal{K}_\Cal{X}}&=&\sup_{k\in\Cal{K}, x\in\Cal{X}}
\biggl\llvert \int k(y,x) \,\mathrm{d}(\bb{P}-\bb{Q}) (y)\biggr\rrvert
\\
&=&\sup_{k\in\Cal{K}, x\in\Cal{X}} \biggl\llvert \biggl\langle k(\cdot,x),\int
k(\cdot,y) \,\mathrm{d}(\bb{P}-\bb{Q}) (y) \biggr\rangle_{\Cal{H}_k}\biggr\rrvert
\\
&\le& \sup_{k\in\Cal{K}, x\in\Cal{X}}\bigl\Vert k(\cdot,x)\bigr\Vert_{\Cal{H}_k}
\frak{D}_k(\bb{P},\bb{Q})\le\sqrt{\nu}\Vert \bb{P}-\bb{Q}
\Vert_{\Cal{F}_H},
\end{eqnarray*}
which proves the lower bound on $\Vert
\bb{P}-\bb{Q}\Vert_{\Cal{F}_H}$ in (\ref{Eq:bound-distance}). To prove the
upper bound, consider
\begin{eqnarray*}
\Vert \bb{P}-\bb{Q}\Vert^2_{\Cal{F}_H}&\stackrel{(\ref{Eq:MMD-1})}
{=} &\sup_{
k\in\Cal{ K } } \int\int k(x,y) \,\mathrm{d}(\bb{P}-\bb{Q}) (x) \,\mathrm{d}(
\bb{P}-\bb{Q}) (y)
\\
&\le&\sup_{k\in\Cal{K}}\int\biggl\llvert
\int k(x,y) \,\mathrm{d}(\bb{P}-\bb{Q}) (y)\biggr\rrvert \,\mathrm{d}|\bb{P}-\bb{Q}|(x)
\\
&\le&2 \sup_{k\in\Cal{K}}\sup_{x\in\Cal{X}}\biggl
\llvert \int k(x,y) \,\mathrm{d}(\bb{P}-\bb{Q}) (y)\biggr\rrvert =2 \Vert \bb{P}-\bb{Q}
\Vert_{\Cal{K}_\Cal{X}},
\end{eqnarray*}
thereby proving the result in (\ref{Eq:bound-distance}). Equation~(\ref{Eq:chain-pro})
simply follows from Theorem~\ref{Thm:weak} and (\ref{Eq:bound-distance}).

\subsection{Proof of Theorem \texorpdfstring{\protect\ref{thmm:uclt}}{5.2}}\label{subsec:uclt}
In order to prove Theorem~\ref{thmm:uclt}, we need a lemma (see
Lemma~\ref{lem:fatshatter} below) which is based on the notion of
\emph{fat-shattering dimension} (see  Anthony and Bartlett \cite{Anthony-99},
Definition~11.10), defined as follows.

\begin{fin}[(Fat-shattering dimension)]\label{def:fatshatter}
Let $\Cal{F}$ be a set of real-valued functions defined on $\Cal{X}$.
For every
$\epsilon>0$, a set $S=\{z_1,\ldots, z_n\}\subset\Cal{X}$ is said to be
$\epsilon$-shattered by $\Cal{F}$ if there exists real numbers
$r_1,\ldots,r_n$
such that for each $b\in\{0,1\}^n$ there is a function $f_b\in\Cal{F}$ with
$f_b(z_i)\ge r_i+\epsilon$ if $b_i=1$ and $f_b(z_i)\le r_i-\epsilon$ if
$b_i=0$, for $1\le i\le n$. The fat-shattering dimension of $\Cal{F}$ is
defined as
\[
\mathrm{fat}_{\epsilon}(\Cal{F})=\sup \bigl\{|S| {|} S\subset\Cal{X},
S \mbox{ is }
\epsilon\mbox{-shattered by } \Cal{F} \bigr\}.
\]
\end{fin}

\begin{lem}\label{lem:fatshatter}
Define
\[
\Cal{G}:= \bigl\{\mathrm{e}^{-\sigma(\cdot-x)^2} \dvt \sigma\in (0,\infty), x\in\bb{R} \bigr
\}.
\]
Then $\mathrm{fat}_{\epsilon}(\Cal{G})\le
1+\lfloor\epsilon^{-1}\rfloor$. In addition, there exists a universal constant
$c^\prime$ such that for every empirical
measure $\bb{P}_n$, and every $0<\epsilon\le1$,
\[
\log\Cal{N}\bigl(\Cal{G},L^2(\bb{P}_n),\epsilon\bigr)\le
c^\prime \biggl(1+\frac{8}{\epsilon} \biggr)\log^2 \biggl(
\frac{2}{\epsilon
}+\frac{16
}{\epsilon^2} \biggr).
\]
\end{lem}

\begin{pf}
Since\vspace*{1pt} $\int^\infty_{-\infty}\llvert \frac{\mathrm{d}g}{\mathrm{d}y}\rrvert  \,\mathrm{d}y=2<\infty$ for all
$g\in\Cal{G}$ then $\Cal{G}\subset \mathit{BV}(\bb{R})$ where $\mathit{BV}(\bb{R})$ is
the space
of functions of bounded variation on $\bb{R}$. Therefore, by
Anthony and Bartlett \cite{Anthony-99}, Theorem~11.12, we obtain $\mathrm{fat}_{\epsilon}(\Cal{G})\le
1+\lfloor\epsilon^{-1}\rfloor$ and Mendelson \cite{Mendelson-02}, Theorem~3.2, ensures
that there exists a universal constant $c^\prime$ such that for every empirical
measure $\bb{P}_n$, and every $\epsilon>0$,
\[
\log\Cal{N}\bigl(\Cal{G},L^2(\bb{P}_n),\epsilon\bigr)\le
c^\prime\mathrm{fat}_{{\epsilon}/{8}}(\Cal{G})\log^2 \biggl(
\frac{2
\mathrm{fat}_
{{\epsilon}/{8}}(\Cal{G})}{\epsilon} \biggr)\le c^\prime \biggl(1+\frac{8}{\epsilon}
\biggr)\log^2 \biggl(\frac{2}{\epsilon
}+\frac{16
}{\epsilon^2} \biggr),
\]
thereby yielding the result.
\end{pf}

\begin{pf*}{Proof of Theorem~\ref{thmm:uclt}}
(a) Define
\[
\Cal{F}_i:= \bigl\{\mathrm{e}^{-\sigma_i(\cdot-x_i)^2} \dvt \sigma_i\in
(0,\infty), x_i\in\bb{R} \bigr\},\qquad i=1,\ldots,d.
\]
By
Lemma~\ref{lem:fatshatter}, it is easy to see that there exists
$N_i(\epsilon):=\Cal{N}(\Cal{F}_i,L^2(\bb{P}_n),\epsilon)$ functions
\[
\bigl\{\mathrm{e}^{-\sigma_{i,1}(\cdot-x_{i,1})^2},\ldots, \mathrm{e}^{-\sigma_{i,N_i(\epsilon)}(\cdot-x_{i,N_i(\epsilon)})^2} \bigr\} \subset\Cal{F}
_i
\]
such that for any $\epsilon>0$ and $f\in\Cal{F}_i$, there exists $\l\in
\{1,\ldots,N_i(\epsilon)\}$ such that
\[
\bigl\llVert f-\mathrm{e}^{-\sigma_{i,l}(\cdot-x_{i,l})^2} \bigr\rrVert _{L^2(\bb{P}_n)}\le \epsilon.
\]
Now pick $l_i\in\{1,\ldots,N_i(\epsilon)\}, i=1,\ldots,d$. Then
for $k(\cdot,x)=\mathrm{e}^{-\sigma\Vert\cdot-x\Vert^2}$, we have
\begin{eqnarray*}
\Biggl\llVert \mathrm{e}^{-\sigma\Vert
\cdot-x\Vert^2}-\prod^d_{i=1}\mathrm{e}^{-\sigma_{i,l_i}(\cdot
-x_{i,l_i})^2}
\Biggr\rrVert _{
L^2(\bb{P}_n)}&=&\Biggl\llVert \prod
^d_{i=1}\mathrm{e}^{-\sigma(\cdot-x_i)^2}-\prod
^d_{i=1}\mathrm{e}^{-\sigma
_{i,l_i}(\cdot-x_{
i,l_i} )^2}\Biggr\rrVert
_ {
L^2(\bb{P}_n)}
\\
&\le&\Biggl\llVert \sum
^d_{i=1}\bigl\llvert \mathrm{e}^{-\sigma(\cdot-x_i)^2}-\mathrm{e}^{-\sigma_{i,l_i}
(\cdot-x_{i,l_i})^2}
\bigr\rrvert \Biggr\rrVert _{L^2(\bb{P}_n)}
\\
&\le&\sum^d_{i=1}\bigl\llVert
\mathrm{e}^{-\sigma(\cdot-x_i)^2}-\mathrm{e}^{-\sigma_{i,l_i}
(\cdot-x_{i,l_i})^2}\bigr\rrVert _{L^2(\bb{P}_n)}\le\epsilon d.
\end{eqnarray*}
This implies $\Cal{N}(\Cal{K}_\Cal{X},L^2(\bb{P}_n),\epsilon
d)=\prod^d_{i=1}N_i(\epsilon)$ and, therefore,
\[
\log\Cal{N}\bigl(\Cal{K}_\Cal{X},L^2(
\bb{P}_n),\epsilon d\bigr)=\sum^d_{i=1}
\log N_i(\epsilon),
\]
which by Lemma~\ref{lem:fatshatter} yields
\[
\sup_n\sup_{\bb{P}_n}\log\Cal{N}\bigl(
\Cal{K}_\Cal{X},L^2(\bb{P}_n),\epsilon \bigr)
\le c^\prime d \biggl(1+\frac{8d}{\epsilon} \biggr)\log^2
\biggl(\frac{2d}{\epsilon}+\frac{16d^2}{
\epsilon^2} \biggr),\qquad 0<\epsilon\le1.
\]
It is easy to verify that $\int^\infty_0 \sup_n\sup_{\bb{P}_n}\log
\Cal{N}(\Cal{K}_\Cal{X},L^2(\bb{P}_n),\epsilon)<\infty$. Therefore,
$\Cal{K}_\Cal{X}$ is a universal Donsker class and the UCLTs
follow.

(b) Following the setting in (a) above, for
$k(\cdot,x)=\int^\infty_0
\mathrm{e}^{-\sigma\Vert\cdot-x\Vert^2_2} \,\mathrm{d}\Lambda(\sigma), \Lambda\in\Cal{M}_A$,
we
have
\[
k(\cdot,x)-\prod^d_{i=1}\mathrm{e}^{-\sigma_{i,l_i}(\cdot-x_{i,l_i
})^2}=
\int^\infty_0 \Biggl(\mathrm{e}^{-\sigma\Vert\cdot-x\Vert^2_2}-\prod
^d_{i=1}\mathrm{e}^{-\sigma_{i,l_i}
(\cdot-x_ {
i,l_i} )^2} \Biggr) \,\mathrm{d}
\Lambda(\sigma)
\]
and so
\[
\Biggl\llVert k(\cdot,x)-\prod^d_{i=1}\mathrm{e}^{-\sigma_{i,l_i}(\cdot-x_{i,l_i
})^2}
\Biggr\rrVert _{
L^2(\bb{P}_n)}\le
\int^\infty_0\Biggl\llVert \mathrm{e}^{-\sigma\Vert\cdot-x\Vert^2_2}-
\prod^d_{i=1}\mathrm{e}^{-\sigma_{i,l_i}
(\cdot-x_ {
i,l_i} )^2}\Biggr
\rrVert _ {
L^2(\bb{P}_n)} \,\mathrm{d}\Lambda(\sigma)
\le\epsilon d,
\]
and the claim as in (a) follows.

(c) The idea is similar to that of in (b) where for
$k(\cdot,x)=\prod^d_{i=1}\int^\infty_0
\mathrm{e}^{-\sigma(\cdot-x_i)^2} \,\mathrm{d}\Lambda_i(\sigma)$, $\Lambda_i\in\Cal{M}_{A_i}$,
we have
\begin{eqnarray*}
\Biggl\llVert k(\cdot,x)-\prod^d_{i=1}\mathrm{e}^{-\sigma_{i,l_i}
(\cdot-x_{i, l_i
})^2}
\Biggr\rrVert _{
L^2(\bb{P}_n)}&\le&\Biggl\llVert \sum
^d_{i=1}\biggl\llvert \int^\infty_0
\mathrm{e}^{-\sigma(\cdot-x_i)^2} \,\mathrm{d}\Lambda_i(\sigma)-\mathrm{e}^{-\sigma_{i,l_i}
(\cdot-x_{i, l_i})^2}\biggr\rrvert
\Biggr\rrVert _{L^2(\bb{P}_n)}
\\
&\le&\sum^d_{i=1}\biggl\llVert \int
^\infty_0 \mathrm{e}^{-\sigma(\cdot-x_i)^2} \,\mathrm{d}
\Lambda_i(\sigma)-\mathrm{e}^{-\sigma_{i,l_i}
(\cdot-x_{i, l_i})^2}\biggr\rrVert _{L^2(\bb{P}_n)}
\\
&\le&\sum^d_{i=1}
\int^\infty_0 \bigl\llVert \mathrm{e}^{-\sigma(\cdot-x_i)^2}-\mathrm{e}^{-\sigma_{i,l_i}
(\cdot-x_{i,
l_i})^2}
\bigr\rrVert _{L^2(\bb{P}_n)} \,\mathrm{d}\Lambda_i(\sigma)
\\
&\le&\epsilon d,
\end{eqnarray*}
and the claim as in (a) follows.

(d) From (\ref{Eq:matern-equiv}), we have
\[
k(x,y)=\frac{(c^2/4)^{\beta-{d}/{2}}}{\Gamma(\beta-{d}/{2})} \int^\infty_0
\mathrm{e}^{-\sigma\Vert x-y\Vert^2_2} \sigma^{{d}/{2}-
\beta-1}\mathrm{e}^{-{c^2}/{(4\sigma)}} \,\mathrm{d}\sigma,
\]
which is of the form in (b) where
$\mathrm{d}\Lambda(\sigma)=\frac{(c^2/4)^{\beta-{d}/{2}}}{\Gamma(\beta-{d}/{2})}
\sigma^{{d}/{2}-\beta-1}\mathrm{e}^{-{c^2}/{(4\sigma)}} \,\mathrm{d}\sigma$ and the result
follows from (b).
\end{pf*}

%
\begin{appendix}\label{app}
\section{Supplementary results}
In the following, we present supplementary results that are used in the proofs
of Theorems \ref{Thm:consistency} and \ref{Thm:estim-to-p}.
Before we present a result to bound $U_n(\Cal{K};(X_i)^n_{i=1})$, we
need the
following lemma. We refer the reader to de la Pe{\~{n}}a and
Gin{\'{e}} (\cite{delaPena-99}, Proposition~4.3.1 and equation
5.1.9) for generalized versions of this result. However,
here, we
provide a
bound with explicit constants.

\begin{appxlem}\label{lem:massart}
Let $\Cal{A}$ be a finite subset of $\bb{R}^{{l(l-1)}/{2}}$ and
$(\varepsilon_i)^l_{i=1}$ be independent Rademacher variables. For any
$a\in\Cal{A}$, define $a:=(a_{ij})_{1\le i<j\le n}$. Suppose
$\sup_{a\in\Cal{A}}\Vert a\Vert_2\le R<\infty$, then for any $0<\theta<1$,
%
\begin{equation}
\bb{E}\sup_{a\in\Cal{A}}\Biggl\llvert \sum
^l_{i<j} \varepsilon_i
\varepsilon_ja_ { ij } \Biggr\rrvert \le \frac{\mathrm{e}R}{\theta}
\log\frac{|\Cal{A}|}{1-\theta}\label{Eq:theta}
\end{equation}
and, therefore,
%
\begin{equation}
\bb{E}\sup_{a\in\Cal{A}}\Biggl\llvert \sum
^l_{i<j} \varepsilon_i
\varepsilon_ja_ { ij } \Biggr\rrvert < \mathrm{e}R \bigl(1+\sqrt{\log|
\Cal{A}|} \bigr)^2.\label{Eq:theta-1}
\end{equation}
\end{appxlem}

\begin{pf}
For $\lambda>0$, consider
\begin{eqnarray*}
\mathrm{e}^{\lambda
\bb{E}\sup_{a\in\Cal{A}}\llvert \sum^l_{i<j}\varepsilon_i\varepsilon_ja_{ij}
\rrvert }&\le& \bb{E} \mathrm{e}^{\lambda\sup_{a\in\Cal{A}}\llvert \sum^l_{i<j}
\varepsilon_i\varepsilon_ja_{ij}
\rrvert }= \bb{E}\sup
_{a\in\Cal{A}}\mathrm{e}^{\lambda\llvert \sum^l_{i<j}
\varepsilon_i\varepsilon_j
a_{ij}\rrvert }
\\
& \le& \sum_{a\in\Cal{A}}\bb{E} \mathrm{e}^{\lambda\llvert \sum^l_{i<j}
\varepsilon_i\varepsilon_ja_{ij}
\rrvert }=\sum
_{a\in\Cal{A}}\bb{E}\sum^\infty_{c=0}
\frac{\lambda^c\llvert \sum^l_{
i<j}
\varepsilon_i\varepsilon_ja_{ij}
\rrvert ^c}{c !}.
\end{eqnarray*}
By the hypercontractivity of homogeneous Rademacher chaos of
degree 2 (de la Pe{\~{n}}a and
Gin{\'{e}}~\cite{delaPena-99}, Theorem~3.2.2), we have
\[
\bb{E}\Biggl\llvert \sum^l_{i<j}
\varepsilon_i\varepsilon_ja_{ij} \Biggr\rrvert
^c\le(c-1)^c \Biggl(\bb{E}\Biggl\llvert \sum
^l_{i<j} \varepsilon_i
\varepsilon_ja_{ij}\Biggr\rrvert ^2
\Biggr)^{c/2}\le (c-1)^c \Biggl(\sum
^l_{i<j}a^2_{ij}
\Biggr)^{c/2},\qquad c\ge2
\]
and
\[
\bb{E}\Biggl\llvert \sum^l_{i<j}
\varepsilon_i\varepsilon_ja_{ij} \Biggr\rrvert
\le \Biggl(\bb{E}\Biggl\llvert \sum^l_{i<j}
\varepsilon_i\varepsilon_ja_{ij} \Biggr\rrvert
^2 \Biggr)^{1/2}\le \Biggl(\sum
^l_{i<j}a^2_{ij}
\Biggr)^{1/2},
\]
which implies
\[
\mathrm{e}^{\lambda
\bb{E}\sup_{a\in\Cal{A}}\llvert \sum^l_{i<j}\varepsilon_i\varepsilon_ja_{ij}
\rrvert }\le\sum_{a\in\Cal{A}}\sum
^\infty_{c=0}\frac{\lambda^c
c^c\Vert a\Vert^{c}_2}{c !}.
\]
Using $c^c/c !\le \mathrm{e}^{c}$ and choosing $\lambda=\frac{\theta}{\mathrm{e}R}$ for some
$0<\theta<1$, we obtain the desired result in (\ref{Eq:theta}). Using
$-\log(1-\theta)<\theta/(1-\theta)$ for $0<\theta<1$ in (\ref
{Eq:theta}) and
taking infimum over $\theta\in(0,1)$ (where the infimum is
obtained at $\theta=\sqrt{\log|\Cal{A}|}/(1+\sqrt{\log|\Cal{A}|})$) yields
(\ref{Eq:theta-1}).
\end{pf}

The following result is based on the standard chaining argument to
obtain a bound on the expected suprema of the Rademacher chaos process
of degree
2. While the reader can to refer to de la Pe{\~{n}}a and
Gin{\'{e}} \cite{delaPena-99}, Corollary~5.18, for a
general result
to bound the expected suprema of the Rademacher chaos process of degree
$m$, we
present a bound with explicit constants and with the lower limit of the entropy
integral away from zero. This allows one to handle classes whose
entropy number
grows polynomially (for $\beta\ge1$ in Theorem~\ref{Thm:consistency}) in
contrast to the entropy integral bound in de la Pe{\~{n}}a and
Gin{\'{e}} \cite{delaPena-99}, Equation 5.1.22,
where the
integral diverges to infinity. Similar modification to the Dudley entropy
integral bound on the expected suprema of empirical processes is
carried out in
Mendelson \cite{Mendelson-02}.

\begin{appxlem}\label{lem:chaining}
Suppose $\Cal{G}$ is a class of real-valued functions on
$\Cal{X}\times\Cal{X}$ and
$(\varepsilon_i)^n_{i=1}$ be a independent Rademacher variables. Define
$\beta:=\sup_{g_1,g_2\in\Cal{G}}\rho(g_1,g_2)$. Then, for any
$(x_i)^n_{i=1}\subset\Cal{X}$ and $0<\theta<1$,
\begin{eqnarray*}
\bb{E}\sup_{g\in\Cal{G}}\Biggl\llvert \sum
^n_{i<j} \varepsilon_i
\varepsilon_jg(x_i,x_j)\Biggr\rrvert
\le 2\sqrt{2} n^2 \biggl(\inf_{\alpha>0} \biggl\{
\alpha+\frac{3\mathrm{e}}{
\theta}\int^{\beta}_{\alpha}
\frac{1}{n}\log\frac{
\Cal{ N } (\Cal{ G }, \rho
,\epsilon)}{\sqrt{1-\theta}} \,\mathrm{d}\epsilon \biggr\} \biggr) +\frac{n}{\sqrt{2}}\sup_{g\in\Cal{G}}\rho(g,0),
\end{eqnarray*}
where for any
$g_1,g_2\in\Cal{G}$, $\rho(g_1,g_2)=\sqrt{\frac{2}{n^2}\sum^n_{i<j}
 (g_1(x_i, x_j)-g_2(x_i,x_j) )^2 }$ and therefore
\begin{eqnarray*}
\bb{E}\sup_{g\in\Cal{G}}\Biggl\llvert \sum
^n_{i<j} \varepsilon_i
\varepsilon_jg(x_i,x_j)\Biggr\rrvert &<&
2\sqrt{2} n^2 \biggl(\inf_{\alpha>0} \biggl\{\alpha+
\frac{3\mathrm{e}}{n} \int^{\beta}_{\alpha} \bigl(1+\sqrt{
\log \Cal{ N } (\Cal{ G }, \rho,\epsilon)} \bigr)^2 \,\mathrm{d}\epsilon \biggr\}
\biggr)
\\
& &{} +\frac{n}{\sqrt{2}}\sup_{g\in\Cal{G}}\rho(g,0).
\end{eqnarray*}
\end{appxlem}

\begin{pf}
Let $\delta_0:=\sup_{g_1,g_2\in\Cal{G}}\rho(g_1,g_2)$ and for any $l\in
\bb{N}$, let $\delta_l:=2^{-l}\delta_0$. For each $l\in\bb{N}\cup\{0\}
$, let
$\Cal{G}_l:=\{g^1_l,\ldots,g^{\Cal{N}(\Cal{G},\rho,\delta_l)}_l\}$ be a
$\rho$-cover of $\Cal{G}$ at scale $\delta_l$. For any $M$, any $g\in
\Cal{G}$
can be expressed as
\[
g=(g-g_M)+\sum^M_{l=1}(g_{l}-g_{l-1})+g_0,
\]
where $g_l\in\Cal{G}_l$ and $\Cal{G}_0:=\Cal{G}$. Note that
$\rho(g_l,g_{l-1})\le\rho(g,g_l)+\rho(g,g_{l-1})\le
\delta_l+\delta_{l-1}=3\delta_l$. Consider
%
\begin{eqnarray}
\label{Eq:chain}&&\bb{E}\sup_{g\in\Cal{G}}\Biggl\llvert \sum
^n_{i<j} \varepsilon_i
\varepsilon_jg(x_i,x_j)\Biggr\rrvert \nonumber\\
&&\quad\le
\bb{E}\sup_{g\in\Cal{G}}\Biggl\llvert \sum
^n_{i<j}\varepsilon_i
\varepsilon_j\bigl(g(x_i, x_j)-g_M(x_i,x_j)
\bigr)\Biggr\rrvert +\bb{E} \Biggl\llvert \sum^n_{i<j}
\varepsilon_i\varepsilon_jg_0(x_i,x_j)
\Biggr\rrvert\nonumber
\\
& &\qquad{} +\sum^M_{l=1}\bb{E}\mathop{\sup
_{g_l\in\Cal{G}_l,g_{l-1}\in
\Cal{G}_{l-1}}}_{ \rho(g_l, g_ { l-1 } )\le3\delta_l} \Biggl\llvert \sum
^n_ {
i<j } \varepsilon_i
\varepsilon_j\bigl(g_l(x_i,
x_j)-g_{l-1}(x_i,x_j)\bigr)\Biggr
\rrvert .
\end{eqnarray}
Note that
%
\begin{eqnarray}
\label{Eq:1} \bb{E}\sup_{g\in\Cal{G}}\Biggl\llvert \sum
^n_{i<j}\varepsilon_i
\varepsilon_j\bigl(g(x_i, x_j)-g_M(x_i,x_j)
\bigr)\Biggr\rrvert &\le &\bb{E}\sum^n_{j=1}
\varepsilon^2_j\sup_{g\in\Cal{G}}\sqrt
{\sum^n_{i<j}
\bigl(g(x_i,x_j )-g_M(x_i,x_j)
\bigr)^2}
\nonumber
\\[-8pt]
\\[-8pt]
\nonumber
&=& \frac{n^2}{\sqrt{2}}\sup_{g\in\Cal{G}}\rho(g,g_M)
\le \frac{n^2\delta_M}{\sqrt{2}},
\\
\label{Eq:2} \bb{E} \Biggl\llvert \sum^n_{i<j}
\varepsilon_i\varepsilon_jg_0(x_i,x_j)
\Biggr\rrvert &\le& \Biggl(\bb{E} \Biggl\llvert \sum
^n_{i<j}\varepsilon_i
\varepsilon_jg_0(x_i,x_j)\Biggr
\rrvert ^2 \Biggr)^{1/2}
\nonumber
\\[-8pt]
\\[-8pt]
\nonumber
&\le& 
 \Biggl(\sum^n_{i<j}g^2_0(x_i,x_j)
\Biggr)^{1/2}= \frac{n}{\sqrt{2}}\sup_{g\in\Cal{G}}
\rho(g,0),
\end{eqnarray}
and by Lemma~\ref{lem:massart},
%
\begin{eqnarray}
\label{Eq:3} &&\bb{E}\mathop{\sup_{g_l\in\Cal{G}_l,g_{l-1}\in\Cal{G}_{
l-1}}}_{ \rho(g_l, g_{l-1 } )\le3\delta_l}
\Biggl\llvert \sum^n_ { i<j }
\varepsilon_i\varepsilon_j\bigl(g_l(x_i,
x_j)-g_{l-1}(x_i,x_j)\bigr)\Biggr
\rrvert \nonumber\\
&&\quad\le \frac{3\mathrm{e}\delta_ln}{\theta\sqrt{2}}\log\frac{\Cal{N}(\Cal{G},\rho,\delta
_l)\Cal{N
}(\Cal{G},\rho,\delta_{l-1})}{1-\theta}
\nonumber\\
&&\quad\le \frac{6\mathrm{e}\delta_ln}{\theta\sqrt{2}}\log\frac{\Cal{N}(\Cal
{G},\rho,
\delta_l)}{\sqrt{1-\theta}}
\end{eqnarray}
for any $0<\theta<1$.
Using (\ref{Eq:1})--(\ref{Eq:3}) in (\ref{Eq:chain}), we have
%
\begin{eqnarray}
\label{Eq:4}&& \bb{E}\sup_{g\in\Cal{G}}\Biggl\llvert \sum
^n_{i<j} \varepsilon_i
\varepsilon_jg(x_i,x_j)\Biggr\rrvert \nonumber\\
&&\quad\le
\frac{n^2\delta_M}{\sqrt{2}}+ \frac{n}{\sqrt{2}}\sup_{g\in\Cal{G}}
\rho(g,0)+\frac{6\mathrm{e}n}{
\theta\sqrt{2}}\sum^M_{l=1}
\delta_l\log\frac{\Cal{N}(\Cal{G},\rho,
\delta_l)}{\sqrt{1-\theta}}
\nonumber
\\[-8pt]
\\[-8pt]
\nonumber
&&\quad\le\frac{n^2\delta_M}{\sqrt{2}}+ \frac{n}{\sqrt{2}}\sup_{g\in\Cal{G}}
\rho(g,0)+\frac{12\mathrm{e}n}{
\theta\sqrt{2}}\sum^M_{l=1}(
\delta_l-\delta_{l+1})\log\frac{\Cal{N}(\Cal
{G},\rho
,\delta_l)}{\sqrt{1-\theta}}
\\
&&\quad\le\frac{n^2\delta_M}{\sqrt{2}}+ \frac{n}{\sqrt{2}}\sup_{g\in\Cal{G}}
\rho(g,0)+\frac{12\mathrm{e}n}{
\theta\sqrt{2}}\int^{\delta_0}_{\delta_{M+1}}\log
\frac{
\Cal{ N } (\Cal{ G }, \rho
,\epsilon)}{\sqrt{1-\theta}} \,\mathrm{d}\epsilon.\nonumber
\end{eqnarray}
For any $\alpha>0$, pick $M:=\sup\{l\dvt \delta_l>2\alpha\}$. This means
$\delta_{M+1}\le2\alpha$ and, therefore, $\delta_M=2\delta_{M+1}\le
4\alpha$. On
the other hand, $\delta_{M+1}>\alpha$ since $\delta_M>2\alpha$. Using these
bounds in (\ref{Eq:4}), we obtain
%
\begin{eqnarray}
\bb{E}\sup_{g\in\Cal{G}}\Biggl\llvert \sum
^n_{i<j} \varepsilon_i
\varepsilon_jg(x_i,x_j)\Biggr\rrvert &
\le& 2\sqrt{2} n^2\alpha+\frac{6\sqrt{2}\mathrm{e}n}{
\theta}\int
^{\sup_{g_1,g_2\in\Cal{G}}\rho(g_1,g_2)}_{\alpha} \log\frac{
\Cal{ N } (\Cal{ G }, \rho
,\epsilon)}{\sqrt{1-\theta}} \,\mathrm{d}\epsilon
\nonumber
\\[-8pt]
\\[-8pt]
\nonumber
& &{} +\frac{n}{\sqrt{2}}\sup_{g\in\Cal{G}}\rho(g,0).
\end{eqnarray}
Since $\alpha$ is arbitrary, taking infimum over $\alpha>0$ yields the result.
\end{pf}

%

\section{Bound on
\texorpdfstring{$\Cal{N}(\Cal{K}_\alpha,\rho_\alpha,\epsilon)$}{$N(K_alpha,rho_alpha,epsilon)$} in
Theorem~\texorpdfstring{\protect\ref{thmm:examples}}{4.2}(a)}
The following result presents a bound on
$\Cal{N}(\Cal{K}_\alpha,\rho_\alpha,\epsilon)$ when
\[
\Cal{K}= \bigl\{\mathrm{e}^{-\sigma\Vert x-y\Vert^2_2}, x,y\in (a_0,b_0)^d,
-\infty<a_0<b_0<\infty: \sigma\in(0,a] \bigr\},
\]
using which it is easy to check that $\omega_\star<1$ in
Theorem~\ref{Thm:estim-to-p} and, therefore, the claims shown in
Theorem~\ref{thmm:examples} follow.

\begin{appxpro}\label{Pro:covering}
Define
$\Sigma:=(0,a]$ and $\Cal{K}_\alpha:=\{\partial^{\alpha,\alpha}
\psi_\sigma(x-y), x,y\in
(a_0,b_0)^d, -\infty<a_0<b_0<\infty: \sigma\in\Sigma\}$, where
$\psi_\sigma(x)=\mathrm{e}^{-\sigma\Vert x-y\Vert^2_2}$ and $|\alpha|=r$. Then
\[
\Cal{N}(\Cal{K}_\alpha,\rho_\alpha,\epsilon)=
\frac{C}{\epsilon},
\]
where $\rho_\alpha$ is defined in Theorem~\ref{Thm:estim-to-p} and $C$
is a
constant that depends on $a$, $a_0$, $b_0$, $d$ and $r$.
\end{appxpro}

\begin{pf}
Let
$\Cal{N}(\Sigma,\Vert\cdot\Vert_1,\tau)$ be the $\tau$-covering number of
$\Sigma$ and it is easy to verify that
\[
N(\tau):=\Cal{N}\bigl(\Sigma,\Vert\cdot\Vert_1,\tau\bigr)=\frac{a}{\tau}.
\]
Let
$\Sigma(\tau):=\{\sigma_1,\ldots, \sigma_{N(\tau)}\}$ be the $L^1$
cover of
$\Sigma$. Define $\widetilde{\Cal{K}_\alpha}:=\{\partial^{\alpha,\alpha}
\psi_\sigma(x-y), x,y\in
(a_0,b_0)^d, -\infty<a_0<b_0<\infty: \sigma\in\Sigma(\tau)\}$. Using the
expression for $\partial^{\alpha,\alpha}\psi_\sigma$ in (\ref
{Eq:partial-psi-1}),
we have
\begin{eqnarray*}
&&\bigl\llvert \partial^{\alpha,\alpha}\psi_{\sigma}-\partial^{
\alpha, \alpha}
\psi_ { \sigma_l}\bigr\rrvert (x-y)\\
&&\quad\le \sum
^{\alpha_1}_{j_1=0}\cdots\sum^{\alpha_d}_{j_d=0}A_{j_1\cdots
j_d}
\bigl| \sigma^{r+\sum^d_{i=1} j_i}\mathrm{e}^{-\sigma\Vert x-y\Vert^2_2}-\sigma^{r+\sum^d_{i=1}
j_i}_l\mathrm{e}^{-\sigma_l\Vert x-y\Vert^2_2} \bigr|,
\end{eqnarray*}
where $A_{j_1\cdots j_d}:=\prod^d_{i=1}
\llvert \eta_ { ij_i}\rrvert  (x_i-y_i)^{2j_i}\le
(b_0-a_0)^{2m}\prod^d_{i=1}
\llvert \eta_ { ij_i}\rrvert =:B_{j_1\cdots j_d}$. Note that
\[
C:=\bigl\llvert \sigma^{r+\sum^d_{i=1} j_i}\mathrm{e}^{-\sigma\Vert x-y\Vert^2_2}-\sigma^{r+\sum^d_{i=1}
j_i}_l\mathrm{e}^{-\sigma_l\Vert x-y\Vert^2_2}
\bigr\rrvert
\]
can be bounded as
\begin{eqnarray*}
C&\le& \bigl\llvert \sigma^{r+\sum^d_{i=1} j_i}-\sigma^{r+\sum^d_{i=1}
j_i}_l
\bigr\rrvert +a^{r+\sum^d_{i=1}
j_i}\bigl\llvert \mathrm{e}^{-\sigma\Vert x-y\Vert^2_2}-\mathrm{e}^{-\sigma_l\Vert
x-y\Vert^2_2}
\bigr\rrvert
\\
&\le& \Biggl(r+\sum^d_{i=1}j_i-1
\Biggr)a^{r+\sum^d_{i=1}j_i-1} \llvert \sigma-\sigma_l\rrvert
+a^{r+\sum^d_{i=1}j_i} \llvert \sigma-\sigma_l\rrvert \Vert x-y
\Vert^2_2 \mathrm{e}^{a\Vert x-y\Vert^2_2}
\\
&\le& (2r-1 )a^{2r-1} \llvert \sigma-\sigma_l\rrvert
+d(b_0-a_0)^2a^{2r}\llvert
\sigma-\sigma _l\rrvert \mathrm{e}^{
ad(b_0-a_0)^2}\le\mu\tau,
\end{eqnarray*}
where $\mu$ is a constant that depends on $a$, $a_0$, $b_0$,
$d$ and $r$. Therefore,
\[
\rho_\alpha\bigl(\partial^{\alpha,\alpha}\psi_\sigma,
\partial^{\alpha,\alpha
}\psi_{
\sigma_l}\bigr)\le\bigl\llVert
\partial^{\alpha,\alpha}\psi_{\sigma}-\partial^{\alpha,\alpha}
\psi_{\sigma_l}\bigr\rrVert _\infty\le \mu\tau\sum
^{\alpha_1}_{j_1=0}\cdots\sum^{\alpha_d}_{j_d=0}B_{j_1\cdots
j_d},
\]
which yields the result.
\end{pf}

\end{appendix}
\section*{Acknowledgements}
The work was carried out while the author was a Research Fellow in the
Statistical Laboratory, Department of Pure Mathematics and Mathematical
Statistics,
University of Cambridge. The author profusely thanks Richard Nickl for
many valuable comments and
insightful discussions. The author also thanks the associate editor and
two anonymous reviewers for their careful review and constructive
comments which significantly improved the manuscript.

%

\printhistory

\begin{thebibliography}{41}

\bibitem{Anthony-99}
\begin{bbook}[mr]
\bauthor{\bsnm{Anthony},~\bfnm{Martin}\binits{M.}} \AND
\bauthor{\bsnm{Bartlett},~\bfnm{Peter~L.}\binits{P.L.}}
(\byear{1999}).
\btitle{Neural Network Learning: Theoretical Foundations}.
\blocation{Cambridge}:
\bpublisher{Cambridge Univ. Press}.
\bid{doi={10.1017/CBO9780511624216}, mr={1741038}}
\end{bbook}
%
\bptok{imsref}%
\endbibitem

\bibitem{Aronszajn-50}
\begin{barticle}[mr]
\bauthor{\bsnm{Aronszajn},~\bfnm{N.}\binits{N.}}
(\byear{1950}).
\btitle{Theory of reproducing kernels}.
\bjournal{Trans. Amer. Math. Soc.}
\bvolume{68}
\bpages{337--404}.
\bid{issn={0002-9947}, mr={0051437}}
\end{barticle}
%
\bptok{imsref}%
\endbibitem

\bibitem{Bartlett-05}
\begin{barticle}[mr]
\bauthor{\bsnm{Bartlett},~\bfnm{Peter~L.}\binits{P.L.}},
\bauthor{\bsnm{Bousquet},~\bfnm{Olivier}\binits{O.}} \AND
\bauthor{\bsnm{Mendelson},~\bfnm{Shahar}\binits{S.}}
(\byear{2005}).
\btitle{Local {R}ademacher complexities}.
\bjournal{Ann. Statist.}
\bvolume{33}
\bpages{1497--1537}.
\bid{doi={10.1214/009053605000000282}, issn={0090-5364}, mr={2166554}}
\end{barticle}
%
\bptok{imsref}%
\endbibitem

\bibitem{Berg-84}
\begin{bbook}[mr]
\bauthor{\bsnm{Berg},~\bfnm{Christian}\binits{C.}},
\bauthor{\bsnm{Christensen},~\bfnm{Jens~Peter~Reus}\binits{J.P.R.}} \AND
\bauthor{\bsnm{Ressel},~\bfnm{Paul}\binits{P.}}
(\byear{1984}).
\btitle{Harmonic Analysis on Semigroups: Theory of Positive Definite and Related Functions}.
\bseries{Graduate Texts in Mathematics}
\bvolume{100}.
\blocation{New York}:
\bpublisher{Springer}.
\bid{doi={10.1007/978-1-4612-1128-0}, mr={0747302}}
\end{bbook}
%
\bptok{imsref}%
\endbibitem

\bibitem{Berlinet-04}
\begin{bbook}[mr]
\bauthor{\bsnm{Berlinet},~\bfnm{Alain}\binits{A.}} \AND
\bauthor{\bsnm{Thomas-Agnan},~\bfnm{Christine}\binits{C.}}
(\byear{2004}).
\btitle{Reproducing Kernel {H}ilbert Spaces in Probability and Statistics}.
\blocation{Boston, MA}:
\bpublisher{Kluwer Academic}.
\bid{doi={10.1007/978-1-4419-9096-9}, mr={2239907}}
\end{bbook}
%
\bptok{imsref}%
\endbibitem

\bibitem{Bickel-03}
\begin{barticle}[mr]
\bauthor{\bsnm{Bickel},~\bfnm{Peter~J.}\binits{P.J.}} \AND
\bauthor{\bsnm{Ritov},~\bfnm{Ya'acov}\binits{Y.}}
(\byear{2003}).
\btitle{Nonparametric estimators which can be ``plugged-in''}.
\bjournal{Ann. Statist.}
\bvolume{31}
\bpages{1033--1053}.
\bid{doi={10.1214/aos/1059655904}, issn={0090-5364}, mr={2001641}}
\end{barticle}
%
\bptok{imsref}%
\endbibitem


\bibitem{delaPena-99}
\begin{bbook}[mr]
\bauthor{\bsnm{de~la Pe{\~n}a},~\bfnm{V{\'{\i}}ctor~H.}\binits{V.H.}} \AND
\bauthor{\bsnm{Gin{\'e}},~\bfnm{Evarist}\binits{E.}}
(\byear{1999}).
\btitle{Decoupling: From Dependence to Independence}.
\bseries{Probability and Its Applications (New York)}.
\blocation{New York}:
\bpublisher{Springer}.
\bid{doi={10.1007/978-1-4612-0537-1}, mr={1666908}}
\end{bbook}
%
\bptok{imsref}%
\endbibitem

\bibitem{Devroye-85}
\begin{bbook}[mr]
\bauthor{\bsnm{Devroye},~\bfnm{Luc}\binits{L.}} \AND
\bauthor{\bsnm{Gy{\"o}rfi},~\bfnm{L{\'a}szl{\'o}}\binits{L.}}
(\byear{1985}).
\btitle{Nonparametric Density Estimation: The $L_{1}$ View}.
\bseries{Wiley Series in Probability and Mathematical Statistics: Tracts on Probability and Statistics}.
\blocation{New York}:
\bpublisher{Wiley}.
\bid{mr={0780746}}
\end{bbook}
%
\bptok{imsref}%
\endbibitem


\bibitem{Diestel-77}
\begin{bbook}[mr]
\bauthor{\bsnm{Diestel},~\bfnm{J.}\binits{J.}} \AND
\bauthor{\bsnm{Uhl},~\bfnm{J.~J.}\binits{J.J.} \bsuffix{Jr.}}
(\byear{1977}).
\btitle{Vector Measures}.
\blocation{Providence, RI}:
\bpublisher{Amer. Math. Soc.}
\bid{mr={0453964}}
\end{bbook}
%
\bptok{imsref}%
\endbibitem

\bibitem{Dudley-99}
\begin{bbook}[mr]
\bauthor{\bsnm{Dudley},~\bfnm{R.~M.}\binits{R.M.}}
(\byear{1999}).
\btitle{Uniform Central Limit Theorems}.
\bseries{Cambridge Studies in Advanced Mathematics}
\bvolume{63}.
\blocation{Cambridge}:
\bpublisher{Cambridge Univ. Press}.
\bid{doi={10.1017/CBO9780511665622}, mr={1720712}}
\end{bbook}
%
\bptok{imsref}%
\endbibitem

\bibitem{Dudley-02}
\begin{bbook}[mr]
\bauthor{\bsnm{Dudley},~\bfnm{R.~M.}\binits{R.M.}}
(\byear{2002}).
\btitle{Real Analysis and Probability}.
\bseries{Cambridge Studies in Advanced Mathematics}
\bvolume{74}.
\blocation{Cambridge}:
\bpublisher{Cambridge Univ. Press}.
\bid{doi={10.1017/CBO9780511755347}, mr={1932358}}
\end{bbook}
%
\bptok{imsref}%
\endbibitem

\bibitem{Folland-99}
\begin{bbook}[author]
\bauthor{\bsnm{Folland},~\bfnm{G.~B.}\binits{G.B.}}
(\byear{1999}).
\btitle{Real Analysis: Modern Techniques and Their Applications}.
\blocation{New York}:
\bpublisher{Wiley}.
\end{bbook}
%
\bptok{imsref}%
\endbibitem

\bibitem{Fukumizu-08a}
\begin{binproceedings}[author]
\bauthor{\bsnm{Fukumizu},~\bfnm{K.}\binits{K.}},
\bauthor{\bsnm{Gretton},~\bfnm{A.}\binits{A.}},
\bauthor{\bsnm{Sun},~\bfnm{X.}\binits{X.}} \AND
\bauthor{\bsnm{Sch{\"{o}}lkopf},~\bfnm{B.}\binits{B.}}
(\byear{2008}).
\btitle{Kernel measures of conditional dependence}.
In \bbooktitle{Advances in Neural Information Processing Systems 20}
(\beditor{\bfnm{J.~C.}\binits{J.C.}~\bsnm{Platt}},
\beditor{\bfnm{D.}\binits{D.}~\bsnm{Koller}},
\beditor{\bfnm{Y.}\binits{Y.}~\bsnm{Singer}} \AND
\beditor{\bfnm{S.}\binits{S.}~\bsnm{Roweis}}, eds.)
\bpages{489--496}.
\blocation{Cambridge, MA}:
\bpublisher{MIT Press}.
\end{binproceedings}
%
\bptok{imsref}%
\endbibitem

\bibitem{Fukumizu-08b}
\begin{binproceedings}[author]
\bauthor{\bsnm{Fukumizu},~\bfnm{K.}\binits{K.}},
\bauthor{\bsnm{Sriperumbudur},~\bfnm{B.~K.}\binits{B.K.}},
\bauthor{\bsnm{Gretton},~\bfnm{A.}\binits{A.}} \AND
\bauthor{\bsnm{Sch{\"{o}}lkopf},~\bfnm{B.}\binits{B.}}
(\byear{2009}).
\btitle{Characteristic kernels on groups and semigroups}.
In \bbooktitle{Advances in Neural Information Processing Systems 21}
(\beditor{\bfnm{D.}\binits{D.}~\bsnm{Koller}},
\beditor{\bfnm{D.}\binits{D.}~\bsnm{Schuurmans}},
\beditor{\bfnm{Y.}\binits{Y.}~\bsnm{Bengio}} \AND
\beditor{\bfnm{L.}\binits{L.}~\bsnm{Bottou}}, eds.)
\bpages{473--480}.
\blocation{Cambridge, MA}:
\bpublisher{MIT Press}.
\end{binproceedings}
%
\bptok{imsref}%
\endbibitem

\bibitem{Gine-08a}
\begin{barticle}[mr]
\bauthor{\bsnm{Gin{\'e}},~\bfnm{Evarist}\binits{E.}} \AND
\bauthor{\bsnm{Nickl},~\bfnm{Richard}\binits{R.}}
(\byear{2008}).
\btitle{Uniform central limit theorems for kernel density estimators}.
\bjournal{Probab. Theory Related Fields}
\bvolume{141}
\bpages{333--387}.
\bid{doi={10.1007/s00440-007-0087-9}, issn={0178-8051}, mr={2391158}}
\end{barticle}
%
\bptok{imsref}%
\endbibitem

\bibitem{Gine-08b}
\begin{barticle}[mr]
\bauthor{\bsnm{Gin{\'e}},~\bfnm{E.}\binits{E.}} \AND
\bauthor{\bsnm{Nickl},~\bfnm{R.}\binits{R.}}
(\byear{2008}).
\btitle{Adaptation on the space of finite signed measures}.
\bjournal{Math. Methods Statist.}
\bvolume{17}
\bpages{113--122}.
\bid{doi={10.3103/S1066530708020026}, issn={1066-5307}, mr={2429123}}
\end{barticle}
%
\bptok{imsref}%
\endbibitem

\bibitem{Gine-09b}
\begin{barticle}[mr]
\bauthor{\bsnm{Gin{\'e}},~\bfnm{Evarist}\binits{E.}} \AND
\bauthor{\bsnm{Nickl},~\bfnm{Richard}\binits{R.}}
(\byear{2009}).
\btitle{Uniform limit theorems for wavelet density estimators}.
\bjournal{Ann. Probab.}
\bvolume{37}
\bpages{1605--1646}.
\bid{doi={10.1214/08-AOP447}, issn={0091-1798}, mr={2546757}}
\end{barticle}
%
\bptok{imsref}%
\endbibitem

\bibitem{Gine-09a}
\begin{barticle}[mr]
\bauthor{\bsnm{Gin{\'e}},~\bfnm{Evarist}\binits{E.}} \AND
\bauthor{\bsnm{Nickl},~\bfnm{Richard}\binits{R.}}
(\byear{2009}).
\btitle{An exponential inequality for the distribution function of the kernel density estimator, with applications to adaptive estimation}.
\bjournal{Probab. Theory Related Fields}
\bvolume{143}
\bpages{569--596}.
\bid{doi={10.1007/s00440-008-0137-y}, issn={0178-8051}, mr={2475673}}
\end{barticle}
%
\bptok{imsref}%
\endbibitem

\bibitem{Gine-10}
\begin{barticle}[mr]
\bauthor{\bsnm{Gin{\'e}},~\bfnm{Evarist}\binits{E.}} \AND
\bauthor{\bsnm{Nickl},~\bfnm{Richard}\binits{R.}}
(\byear{2010}).
\btitle{Adaptive estimation of a distribution function and its density in sup-norm loss by wavelet and spline projections}.
\bjournal{Bernoulli}
\bvolume{16}
\bpages{1137--1163}.
\bid{doi={10.3150/09-BEJ239}, issn={1350-7265}, mr={2759172}}
\end{barticle}
%
\bptok{imsref}%
\endbibitem

\bibitem{Gine-86}
\begin{barticle}[mr]
\bauthor{\bsnm{Gin{\'e}},~\bfnm{Evarist}\binits{E.}} \AND
\bauthor{\bsnm{Zinn},~\bfnm{Joel}\binits{J.}}
(\byear{1986}).
\btitle{Empirical processes indexed by {L}ipschitz functions}.
\bjournal{Ann. Probab.}
\bvolume{14}
\bpages{1329--1338}.
\bid{issn={0091-1798}, mr={0866353}}
\end{barticle}
%
\bptok{imsref}%
\endbibitem

\bibitem{Gretton-06}
\begin{binproceedings}[author]
\bauthor{\bsnm{Gretton},~\bfnm{A.}\binits{A.}},
\bauthor{\bsnm{Borgwardt},~\bfnm{K.~M.}\binits{K.M.}},
\bauthor{\bsnm{Rasch},~\bfnm{M.}\binits{M.}},
\bauthor{\bsnm{Sch{\"{o}}lkopf},~\bfnm{B.}\binits{B.}} \AND
\bauthor{\bsnm{Smola},~\bfnm{A.}\binits{A.}}
(\byear{2007}).
\btitle{A kernel method for the two sample problem}.
In \bbooktitle{Advances in Neural Information Processing Systems 19}
(\beditor{\bfnm{B.}\binits{B.}~\bsnm{Sch{\"{o}}lkopf}},
\beditor{\bfnm{J.}\binits{J.}~\bsnm{Platt}} \AND
\beditor{\bfnm{T.}\binits{T.}~\bsnm{Hoffman}}, eds.)
\bpages{513--520}.
\blocation{Cambridge, MA}:
\bpublisher{MIT Press}.
\end{binproceedings}
%
\bptok{imsref}%
\endbibitem

\bibitem{Hardle-98}
\begin{bbook}[mr]
\bauthor{\bsnm{H{\"a}rdle},~\bfnm{Wolfgang}\binits{W.}},
\bauthor{\bsnm{Kerkyacharian},~\bfnm{Gerard}\binits{G.}},
\bauthor{\bsnm{Picard},~\bfnm{Dominique}\binits{D.}} \AND
\bauthor{\bsnm{Tsybakov},~\bfnm{Alexander}\binits{A.}}
(\byear{1998}).
\btitle{Wavelets, Approximation, and Statistical Applications}.
\bseries{Lecture Notes in Statistics}
\bvolume{129}.
\blocation{New York}:
\bpublisher{Springer}.
\bid{doi={10.1007/978-1-4612-2222-4}, mr={1618204}}
\end{bbook}
%
\bptok{imsref}%
\endbibitem

\bibitem{Lepski-97}
\begin{barticle}[mr]
\bauthor{\bsnm{Lepski},~\bfnm{O.~V.}\binits{O.V.}},
\bauthor{\bsnm{Mammen},~\bfnm{E.}\binits{E.}} \AND
\bauthor{\bsnm{Spokoiny},~\bfnm{V.~G.}\binits{V.G.}}
(\byear{1997}).
\btitle{Optimal spatial adaptation to inhomogeneous smoothness: An approach based on kernel estimates with variable bandwidth selectors}.
\bjournal{Ann. Statist.}
\bvolume{25}
\bpages{929--947}.
\bid{doi={10.1214/aos/1069362731}, issn={0090-5364}, mr={1447734}}
\end{barticle}
%
\bptok{imsref}%
\endbibitem

\bibitem{Marcus-85}
\begin{barticle}[mr]
\bauthor{\bsnm{Marcus},~\bfnm{David~J.}\binits{D.J.}}
(\byear{1985}).
\btitle{Relationships between {D}onsker classes and {S}obolev spaces}.
\bjournal{Z. Wahrsch. Verw. Gebiete}
\bvolume{69}
\bpages{323--330}.
\bid{doi={10.1007/BF00532737}, issn={0044-3719}, mr={0787601}}
\end{barticle}
%
\bptok{imsref}%
\endbibitem

\bibitem{Mendelson-02}
\begin{barticle}[mr]
\bauthor{\bsnm{Mendelson},~\bfnm{Shahar}\binits{S.}}
(\byear{2002}).
\btitle{Rademacher averages and phase transitions in {G}livenko--{C}antelli classes}.
\bjournal{IEEE Trans. Inform. Theory}
\bvolume{48}
\bpages{251--263}.
\bid{doi={10.1109/18.971753}, issn={0018-9448}, mr={1872178}}
\end{barticle}
%
\bptok{imsref}%
\endbibitem

\bibitem{Nickl-07}
\begin{barticle}[mr]
\bauthor{\bsnm{Nickl},~\bfnm{Richard}\binits{R.}}
(\byear{2007}).
\btitle{Donsker-type theorems for nonparametric maximum likelihood estimators}.
\bjournal{Probab. Theory Related Fields}
\bvolume{138}
\bpages{411--449}.
\bid{doi={10.1007/s00440-006-0031-4}, issn={0178-8051}, mr={2299714}}
\end{barticle}
%
\bptok{imsref}%
\endbibitem

\bibitem{Radulovic-00}
\begin{bincollection}[mr]
\bauthor{\bsnm{Radulovi{\'c}},~\bfnm{Dragan}\binits{D.}} \AND
\bauthor{\bsnm{Wegkamp},~\bfnm{Marten}\binits{M.}}
(\byear{2000}).
\btitle{Weak convergence of smoothed empirical processes: Beyond {D}onsker classes}.
In \bbooktitle{High Dimensional Probability, II ({S}eattle, WA, 1999)}.
\bseries{Progress in Probability}
\bvolume{47}
\bpages{89--105}.
\blocation{Boston, MA}:
\bpublisher{Birkh\"auser}.
\bid{mr={1857317}}
\end{bincollection}
%
\bptok{imsref}%
\endbibitem

\bibitem{Rudin-91}
\begin{bbook}[mr]
\bauthor{\bsnm{Rudin},~\bfnm{Walter}\binits{W.}}
(\byear{1991}).
\btitle{Functional Analysis},
\bedition{2nd} ed.
\bseries{International Series in Pure and Applied Mathematics}.
\blocation{New York}:
\bpublisher{McGraw-Hill}.
\bid{mr={1157815}}
\end{bbook}
%
\bptok{imsref}%
\endbibitem

\bibitem{Srebro-10}
\begin{binproceedings}[author]
\bauthor{\bsnm{Srebro},~\bfnm{N.}\binits{N.}},
\bauthor{\bsnm{Sridharan},~\bfnm{K.}\binits{K.}} \AND
\bauthor{\bsnm{Tewari},~\bfnm{A.}\binits{A.}}
(\byear{2010}).
\btitle{Smoothness, low noise and fast rates}.
In \bbooktitle{Advances in Neural Information Processing Systems 23}
(\beditor{\bfnm{J.}\binits{J.}~\bsnm{Lafferty}},
\beditor{\bfnm{C.~K.~I.}\binits{C.K.I.}~\bsnm{Williams}},
\beditor{\bfnm{J.}\binits{J.}~\bsnm{Shawe-Taylor}},
\beditor{\bfnm{R.~S.}\binits{R.S.}~\bsnm{Zemel}} \AND
\beditor{\bfnm{A.}\binits{A.}~\bsnm{Culotta}}, eds.)
\bpages{2199--2207}.
\bpublisher{MIT Press}.
\end{binproceedings}
%
\bptok{imsref}%
\endbibitem

\bibitem{Sriperumbudur-09c}
\begin{binproceedings}[author]
\bauthor{\bsnm{Sriperumbudur},~\bfnm{B.~K.}\binits{B.K.}},
\bauthor{\bsnm{Fukumizu},~\bfnm{K.}\binits{K.}},
\bauthor{\bsnm{Gretton},~\bfnm{A.}\binits{A.}},
\bauthor{\bsnm{Lanckriet},~\bfnm{G.~R.~G.}\binits{G.R.G.}} \AND
\bauthor{\bsnm{Sch{\"{o}}lkopf},~\bfnm{B.}\binits{B.}}
(\byear{2009}).
\btitle{Kernel choice and classifiability for RKHS embeddings of probability distributions}.
In \bbooktitle{Advances in Neural Information Processing Systems 22}
(\beditor{\bfnm{Y.}\binits{Y.}~\bsnm{Bengio}},
\beditor{\bfnm{D.}\binits{D.}~\bsnm{Schuurmans}},
\beditor{\bfnm{J.}\binits{J.}~\bsnm{Lafferty}},
\beditor{\bfnm{C.~K.~I.}\binits{C.K.I.}~\bsnm{Williams}} \AND
\beditor{\bfnm{A.}\binits{A.}~\bsnm{Culotta}}, eds.)
\bpages{1750--1758}.
\blocation{Cambridge, MA}:
\bpublisher{MIT Press}.
\end{binproceedings}
%
\bptok{imsref}%
\endbibitem

\bibitem{Sriperumbudur-12}
\begin{barticle}[mr]
\bauthor{\bsnm{Sriperumbudur},~\bfnm{Bharath~K.}\binits{B.K.}},
\bauthor{\bsnm{Fukumizu},~\bfnm{Kenji}\binits{K.}},
\bauthor{\bsnm{Gretton},~\bfnm{Arthur}\binits{A.}},
\bauthor{\bsnm{Sch{\"o}lkopf},~\bfnm{Bernhard}\binits{B.}} \AND
\bauthor{\bsnm{Lanckriet},~\bfnm{Gert~R.~G.}\binits{G.R.G.}}
(\byear{2012}).
\btitle{On the empirical estimation of integral probability metrics}.
\bjournal{Electron. J. Stat.}
\bvolume{6}
\bpages{1550--1599}.
\bid{doi={10.1214/12-EJS722}, issn={1935-7524}, mr={2988458}}
\end{barticle}
%
\bptok{imsref}%
\endbibitem

\bibitem{Sriperumbudur-11c}
\begin{binproceedings}[author]
\bauthor{\bsnm{Sriperumbudur},~\bfnm{B.~K.}\binits{B.K.}},
\bauthor{\bsnm{Fukumizu},~\bfnm{K.}\binits{K.}} \AND
\bauthor{\bsnm{Lanckriet},~\bfnm{G.}\binits{G.}}
(\byear{2011}).
\btitle{Learning in {H}ilbert vs. {B}anach spaces: A measure embedding viewpoint}.
In \bbooktitle{Advances in Neural Information Processing Systems~24}
(\beditor{\bfnm{J.}\binits{J.}~\bsnm{Shawe-Taylor}},
\beditor{\bfnm{R.~S.}\binits{R.S.}~\bsnm{Zemel}},
\beditor{\bfnm{P.}\binits{P.}~\bsnm{Bartlett}},
\beditor{\bfnm{F.~C.~N.}\binits{F.C.N.}~\bsnm{Pereira}} \AND
\beditor{\bfnm{K.~Q.}\binits{K.Q.}~\bsnm{Weinberger}}, eds.)
\bpages{1773--1781}.
\blocation{Cambridge, MA}:
\bpublisher{MIT Press}.
\end{binproceedings}
%
\bptok{imsref}%
\endbibitem

\bibitem{Sriperumbudur-11a}
\begin{barticle}[mr]
\bauthor{\bsnm{Sriperumbudur},~\bfnm{Bharath~K.}\binits{B.K.}},
\bauthor{\bsnm{Fukumizu},~\bfnm{Kenji}\binits{K.}} \AND
\bauthor{\bsnm{Lanckriet},~\bfnm{Gert~R.~G.}\binits{G.R.G.}}
(\byear{2011}).
\btitle{Universality, characteristic kernels and RKHS embedding of measures}.
\bjournal{J. Mach. Learn. Res.}
\bvolume{12}
\bpages{2389--2410}.
\bid{issn={1532-4435}, mr={2825431}}
\end{barticle}
%
\bptok{imsref}%
\endbibitem

\bibitem{Sriperumbudur-10a}
\begin{barticle}[mr]
\bauthor{\bsnm{Sriperumbudur},~\bfnm{Bharath~K.}\binits{B.K.}},
\bauthor{\bsnm{Gretton},~\bfnm{Arthur}\binits{A.}},
\bauthor{\bsnm{Fukumizu},~\bfnm{Kenji}\binits{K.}},
\bauthor{\bsnm{Sch{\"o}lkopf},~\bfnm{Bernhard}\binits{B.}} \AND
\bauthor{\bsnm{Lanckriet},~\bfnm{Gert~R.~G.}\binits{G.R.G.}}
(\byear{2010}).
\btitle{Hilbert space embeddings and metrics on probability measures}.
\bjournal{J. Mach. Learn. Res.}
\bvolume{11}
\bpages{1517--1561}.
\bid{issn={1532-4435}, mr={2645460}}
\end{barticle}
%
\bptok{imsref}%
\endbibitem

\bibitem{Steinwart-08}
\begin{bbook}[mr]
\bauthor{\bsnm{Steinwart},~\bfnm{Ingo}\binits{I.}} \AND
\bauthor{\bsnm{Christmann},~\bfnm{Andreas}\binits{A.}}
(\byear{2008}).
\btitle{Support Vector Machines}.
\bseries{Information Science and Statistics}.
\blocation{New York}:
\bpublisher{Springer}.
\bid{mr={2450103}}
\end{bbook}
%
\bptok{imsref}%
\endbibitem

\bibitem{Vaart-94}
\begin{barticle}[mr]
\bauthor{\bsnm{van~der Vaart},~\bfnm{Aad}\binits{A.}}
(\byear{1994}).
\btitle{Weak convergence of smoothed empirical processes}.
\bjournal{Scand. J. Stat.}
\bvolume{21}
\bpages{501--504}.
\bid{issn={0303-6898}, mr={1310093}}
\end{barticle}
%
\bptok{imsref}%
\endbibitem

\bibitem{Vaart-98}
\begin{bbook}[mr]
\bauthor{\bsnm{van~der Vaart},~\bfnm{A.~W.}\binits{A.W.}}
(\byear{1998}).
\btitle{Asymptotic Statistics}.
\bseries{Cambridge Series in Statistical and Probabilistic Mathematics}
\bvolume{3}.
\blocation{Cambridge}:
\bpublisher{Cambridge Univ. Press}.
\bid{doi={10.1017/CBO9780511802256}, mr={1652247}}
\end{bbook}
%
\bptok{imsref}%
\endbibitem

\bibitem{Vaart-96}
\begin{bbook}[mr]
\bauthor{\bsnm{van~der Vaart},~\bfnm{Aad~W.}\binits{A.W.}} \AND
\bauthor{\bsnm{Wellner},~\bfnm{Jon~A.}\binits{J.A.}}
(\byear{1996}).
\btitle{Weak Convergence and Empirical Processes}.
\bseries{Springer Series in Statistics}.
\blocation{New York}:
\bpublisher{Springer}.
\bid{doi={10.1007/978-1-4757-2545-2}, mr={1385671}}
\end{bbook}
%
\bptok{imsref}%
\endbibitem

\bibitem{Wendland-05}
\begin{bbook}[mr]
\bauthor{\bsnm{Wendland},~\bfnm{Holger}\binits{H.}}
(\byear{2005}).
\btitle{Scattered Data Approximation}.
\bseries{Cambridge Monographs on Applied and Computational Mathematics}
\bvolume{17}.
\blocation{Cambridge}:
\bpublisher{Cambridge Univ. Press}.
\bid{mr={2131724}}
\end{bbook}
%
\bptok{imsref}%
\endbibitem

\bibitem{Ying-10}
\begin{barticle}[mr]
\bauthor{\bsnm{Ying},~\bfnm{Yiming}\binits{Y.}} \AND
\bauthor{\bsnm{Campbell},~\bfnm{Colin}\binits{C.}}
(\byear{2010}).
\btitle{Rademacher chaos complexities for learning the kernel problem}.
\bjournal{Neural Comput.}
\bvolume{22}
\bpages{2858--2886}.
\bid{doi={10.1162/NECO-a-00028}, issn={0899-7667}, mr={2760540}}
\end{barticle}
%
\bptok{imsref}%
\endbibitem

\bibitem{Yukich-92}
\begin{barticle}[mr]
\bauthor{\bsnm{Yukich},~\bfnm{J.~E.}\binits{J.E.}}
(\byear{1992}).
\btitle{Weak convergence of smoothed empirical processes}.
\bjournal{Scand. J. Stat.}
\bvolume{19}
\bpages{271--279}.
\bid{issn={0303-6898}, mr={1183201}}
\end{barticle}
%
\bptok{imsref}%
\endbibitem

\end{thebibliography}
\end{document}